\documentclass[a4paper, 11pt]{amsart}
\pdfoutput=1
\usepackage{amscd}
\usepackage{amssymb}
\usepackage[all]{xy}
\usepackage[T1]{fontenc}
\usepackage{lmodern}
\usepackage{graphicx}
\date{7 May 2015. Revised: a few corrections, no major changes.}
\title{A Course on Derived Categories}
\usepackage{hyperref}
\hypersetup{colorlinks=false}

\author{Amnon Yekutieli}

\address{Yekutieli: Department of  Mathematics
Ben Gurion University,
Be'er Sheva 84105,
Israel}
\email{amyekut@math.bgu.ac.il}

\newtheorem{thm}[equation]{Theorem}
\newtheorem{cor}[equation]{Corollary}
\newtheorem{prop}[equation]{Proposition}
\newtheorem{lem}[equation]{Lemma}
\theoremstyle{definition}
\newtheorem{dfn}[equation]{Definition}
\newtheorem{rem}[equation]{Remark}
\newtheorem{exa}[equation]{Example}

\newtheorem{exer}[equation]{Exercise}

\numberwithin{equation}{subsection}
\newcommand{\iso}{\xrightarrow{\simeq}}
\newcommand{\inj}{\hookrightarrow}
\newcommand{\surj}{\twoheadrightarrow}
\newcommand{\xar}{\xrightarrow}
\newcommand{\opn}{\operatorname}

\newcommand{\ol}{\overline}

\newcommand{\rmitem}[1]{\item[\text{\textup{(#1)}}]}
\newcommand{\mfrak}[1]{\mathfrak{#1}}
\newcommand{\mcal}[1]{\mathcal{#1}}
\newcommand{\mc}[1]{\mathcal{#1}}

\newcommand{\mbf}[1]{\mathbf{#1}}
\newcommand{\mrm}[1]{\mathrm{#1}}
\newcommand{\mbb}[1]{\mathbb{#1}}

\newcommand{\tup}[1]{\textup{#1}}
\newcommand{\bsym}[1]{\boldsymbol{#1}}

\newcommand{\ot}{\otimes}

\newcommand{\til}[1]{\tilde{#1}}
\newcommand{\what}[1]{\widehat{#1}}

\newcommand{\K}{\mbb{K}}
\newcommand{\N}{\mbb{N}}
\newcommand{\Z}{\mbb{Z}}

\renewcommand{\a}{\mfrak{a}}

\newcommand{\p}{\mfrak{p}}
\newcommand{\q}{\mfrak{q}}

\newcommand{\al}{\alpha}
\newcommand{\be}{\beta}
\newcommand{\ga}{\gamma}
\newcommand{\la}{\lambda}
\newcommand{\de}{\delta}
\newcommand{\ze}{\zeta}
\newcommand{\ep}{\epsilon}
\renewcommand{\th}{\theta}

\renewcommand{\d}{\mrm{d}}

\newcommand{\sbmat}[1]{\left[ \begin{smallmatrix} #1
\end{smallmatrix} \right]}
\newcommand{\bmat}[1]{\begin{bmatrix} #1 \end{bmatrix}}

\newcommand{\cat}[1]{\mathsf{#1}}
\newcommand{\dcat}[1]{\boldsymbol{\mathsf{#1}}}

\newcommand{\lb}{\linebreak}

\renewcommand{\o}{\circ}
\newcommand{\OX}{\mcal{O}_X}

\begin{document}

\begin{center}
\includegraphics[scale=0.38]{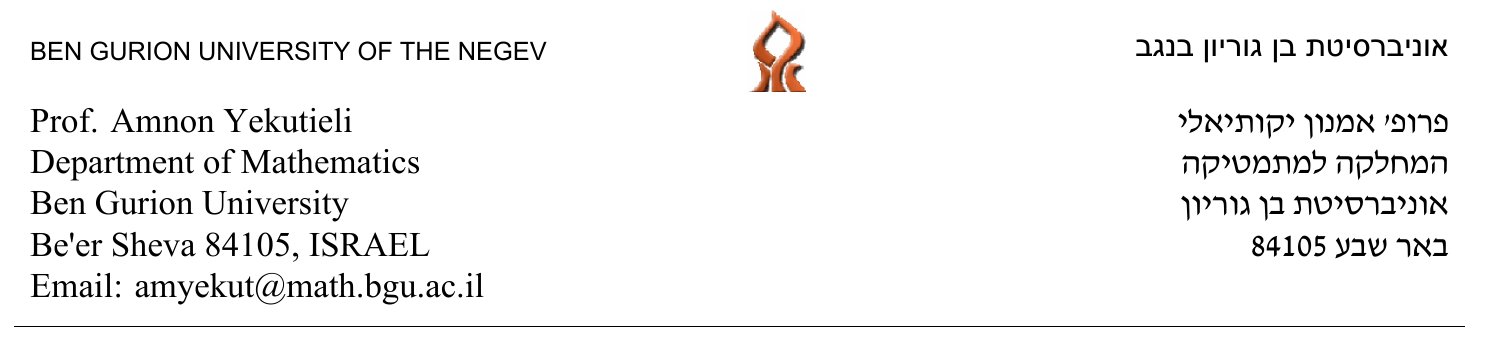}
\end{center}

\vspace{2em}
\maketitle

\tableofcontents

\setcounter{section}{-1}
\section{Introduction}

These are notes for an advanced course given at Ben Gurion University in Spring
2012. In this course I am following various sources, mostly \cite{RD},
\cite{Sc}, \cite{KS2} and \cite{Wei}, but going in a sufficiently different
route to make written notes desirable.  More resources are available on the
course web page \cite{CWP}.
\footnote{For future version: (1) Improve discussion 
of K-injectives and 
K-projectives. Talk about semi-free complexes.
(2) Include DG rings and their derived categories.}

I want to thank the participants of the course for correcting many of my
mistakes, both in real time during the lectures, and in writing.  
Thanks also to J. Lipman, P. Schapira, A. Neeman and C. Weibel for helpful
discussions on the material.

\subsection{A motivating discussion: duality}
By way of introduction to the subject, let us consider {\em duality}. 
Take a field $K$. Given a $K$-module $M$ (i.e.\  a vector space), let
\[ D(M) := \opn{Hom}_K(M, K) , \]
be the dual module. There is a canonical homomorphism 
\[ \eta_M : M \to D(D(M)) ,   \]
$\eta_M(m)(\phi) := \phi(m)$ for $m \in M$ and $\phi \in D(M)$. 
If $M$ is finitely generated then $\eta_M$ is an isomorphism (actually this is
``if and only if''). 

To formalize this situation, let $\cat{Mod}\, K$ denote the category of
$K$-modules. Then 
\[ D : \cat{Mod}\, K \to \cat{Mod}\, K \]
is a contravariant functor, and 
\[ \eta : \bsym{1} \to D \circ D \]
is a natural transformation. Here $\bsym{1}$ is the identity functor of
$\cat{Mod}\, K$. 

Now let us replace $K$ by any (nonzero) commutative ring $A$. Again we can
define a contravariant functor 
\[ D : \cat{Mod}\, A \to \cat{Mod}\, A , \quad
D(M) := \opn{Hom}_A(M, A) , \]
and a natural transformation $\eta : \bsym{1} \to D \circ D$.
It is easy to see that $\eta_M : M \to D(D(M))$ is an isomorphism if $M$
is a finitely generated free module. 
Of course we can't expect reflexivity (i.e.\ $\eta_M$ being an isomorphism) if
$M$ is not finitely generated; but what about a finitely generated module that
is not free?

In order to understand this better, let us concentrate on the ring $A = \Z$.
A finitely generated $\Z$-module $M$, namely a finitely generated abelian
group, is of the form 
$M \cong G \oplus H$, with $G$ free and $H$ finite. It is important to note
that this is not a canonical isomorphism: there is a canonical short exact
sequence 
\[ 0 \to H \to M \to G \to 0 , \]
and the decomposition $M \cong G \oplus H$ comes from choosing a splitting of
this sequence. 

We know that for the free abelian group $G$ there is reflexivity. But for the
finite abelian group $H$ we have
\[ D(H) = \opn{Hom}_{\Z}(H, \Z) = 0 . \]
Thus, whenever $H \neq 0$, reflexivity fails: $\eta_M : M \to D(D(M))$ is not
an isomorphism.

On the other hand, for an abelian group $M$ we can define another sort of dual:
\[ D'(M) := \opn{Hom}_{\Z}(M, \mbb{Q} / \Z) . \]
(We may view the abelian group $\mbb{Q} / \Z$ as the group of roots of $1$ in
$\mbb{C}$, via the exponential map.) 
There is a natural transformation 
$\eta' : \bsym{1} \to D' \circ D'$, 
and if $H$ is a finite abelian group then $\eta'_H$ is an isomorphism.
So $D'$ is a duality for finite abelian groups. Yet for a finitely
generated free abelian group $G$ we get
$D'(D'(G)) = \what{G}$, the profinite completion of $G$. So once more this is
not a good duality for all finitely generated abelian groups.

We could try to be more clever and ``patch'' the two dualities $D$ and $D'$,
into something that we will call $D \oplus D'$. 
This looks pleasing at first -- but then we recall that the decomposition 
$M \cong G \oplus H$ of a finitely generated group is not functorial, so that 
 $D \oplus D'$ can't be a functor.

Later in the course we will introduce the {\em derived category} 
$\dcat{D}(\cat{Mod}\, \Z)$.
The objects of $\dcat{D}(\cat{Mod}\, \Z)$ are the complexes of $\Z$-modules.
There is a contravariant {\em triangulated functor} 
\[ \mrm{R} D : \dcat{D}(\cat{Mod}\, \Z) \to \dcat{D}(\cat{Mod}\, \Z) , \]
\[ \mrm{R} D(M) := \opn{RHom}_{\Z}(M, \Z) . \]
This is the {\em right derived Hom functor}. 
And there is a natural transformation of triangulated functors
\[ \eta : \bsym{1} \to \mrm{R} D \circ \mrm{R} D . \]
If $M$ is a {\em bounded complex} with {\em finitely generated cohomology
modules} then 
$\eta_M : M \to \mrm{R} D(\mrm{R} D(M))$
is an isomorphism in $\dcat{D}(\cat{Mod}\, \Z)$. 
 
We can take a  $\Z$-module $M$ and view it as a
complex as follows:
\begin{equation} \label{eqn:1}
\cdots \to 0 \to M \to 0 \to \cdots 
\end{equation}
where $M$ is in degree $0$. 
This is a fully faithful embedding of $\cat{Mod}\, \Z$ in 
$\dcat{D}(\cat{Mod}\, \Z)$.
If $M$ is a finitely generated module then $\eta_M$ is an isomorphism. 
Thus we have a duality $\mrm{R} D$ that holds for all finitely generated
$\Z$-modules!

Here is the connection between $\mrm{R} D$ and the ``classical'' dualities $D$
and $D'$. Take a finitely generated free abelian group $G$. There is a
functorial isomorphism
\[ \mrm{H}^0 (\mrm{R} D(G)) \cong \opn{Hom}_{\Z}(G, \Z) = D(G) , \]
and 
$\mrm{H}^i (\mrm{R} D(G)) = 0$ for $i \neq 0$. 
For a finite abelian group $H$ there is a functorial isomorphism
\[ \mrm{H}^1 (\mrm{R} D(H)) \cong 
\opn{Ext}^1_{\Z}(H, \Z) \cong D'(H) , \]
and $\mrm{H}^i (\mrm{R} D(H)) = 0$ for $i \neq 1$. 
Therefore, if $M$ is a finitely generated abelian group, and we choose a
decomposition $M \cong G \oplus H$ where $G$ is free and $H$ is finite,
there are (noncanonical) isomorphisms 
\[ \mrm{H}^0 (\mrm{R} D(M)) \cong  D(G) \]
and 
\[ \mrm{H}^1 (\mrm{R} D(M)) \cong  D'(H) . \]
We see that if $M$ is neither free nor finite, then both 
$\mrm{H}^0 (\mrm{R} D(M))$ and $\mrm{H}^1 (\mrm{R} D(M))$ are nonzero.

This sort of duality holds for many noetherian commutative rings $A$. 
But the formula for the duality functor 
\[ \mrm{R} D : \dcat{D}(\cat{Mod}\, A) \to \dcat{D}(\cat{Mod}\, A) \]
is somewhat different -- it is
\[ \mrm{R} D(M) := \opn{RHom}_{A}(M, R) , \]
where $R \in \dcat{D}(\cat{Mod}\, A)$ is a {\em dualizing complex}. 
Such a dualizing complex is unique (up to shift and tensoring with an
invertible module). 

Interestingly, the structure of the dualizing complex $R$ depends on the
geometry of the ring $A$ (i.e.\ of the scheme $\opn{Spec} A$). 
If $A$ is a regular ring (like $\Z$) then $R = A$ is dualizing. 
If $A$ is Cohen-Macaulay then $R$ is a single $A$-module. 
But if $A$ is a more complicated ring then $R$ must live in several degrees. 
For example, consider an affine algebraic variety 
$X \subset \mbf{A}^3_{\mbb{R}}$
which is the union of a plane and a line, say with coordinate ring
\[ A = \mbb{R}[t_1, t_2, t_3] / (t_3 t_1, t_3 t_2) . \]
See figure \ref{fig:1}. 
The dualizing complex $R$ must live in two adjacent degrees;
namely there is some $i$ s.t.\ 
$\mrm{H}^i(R)$ and $\mrm{H}^{i+1}(R)$ are nonzero. 

\begin{figure}
\includegraphics[scale=0.5]{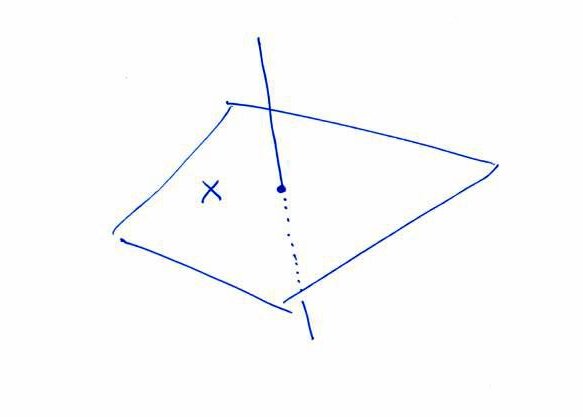}
\caption{} 
\label{fig:1}
\end{figure}

One can also talk about dualizing complexes over {\em noncommutative rings}. 
I am not sure if we will have time to do that in the course. (But this is a
favorite topic for me!)

\cleardoublepage
\section{Basics Facts on Categories}

\subsection{Set Theory}
In this course we will not try to be precise about issues of set theory. The
blanket assumption is that we are given a {\em Grothendieck universe}
$\cat{U}$. This is an infinite set, closed under most set theoretical
operations. A {\em small set}  (or a $\cat{U}$-small set)
is a set $S \in \cat{U}$. A category $\cat{C}$ is a {\em $\cat{U}$-category} if
the set of objects $\opn{Ob}(\cat{C})$ is a subset of $\cat{U}$, 
and for every $C, D \in \opn{Ob}(\cat{C})$ the set of morphisms 
$\opn{Hom}_{\cat{C}}(C, D)$ is small. 
See \cite[Section 1.1]{KS2}; or see \cite{Ne} for another approach. 

We denote by $\cat{Set}$ the category of all small sets. So 
$\opn{Ob}(\cat{Set}) = \cat{U}$, and $\cat{Set}$ is a $\cat{U}$-category.
An abelian group (or a ring, etc.) is called small if its underlying set is
small. For a small ring $A$ we denote by $\cat{Mod}\, A$ the category of all
small left $A$-modules. 

By default we work with $\cat{U}$-categories, and from now on $\cat{U}$ will
remain implicit. The one exception is when we deal with localization of
categories, where we shall briefly encounter a set theoretical issue; but for
most interesting cases this issue has an easy solution.

\subsection{Zero objects}
Let $\cat{C}$ be a category. A morphism $f : C \to D$ in $\cat{C}$ is called an
{\em epimorphism} if it has the right cancellation property: for any
 $g, g' : D \to E$, $g \circ f = g' \circ f$ implies $g = g'$. 
The morphism $f : C \to D$ is called a {\em monomorphism} if it has the left
cancellation property: for any
 $g, g' : E \to C$, $f \circ g = f \circ g'$ implies $g = g'$. 

\begin{exa}
In $\cat{Set}$ the monomorphisms are the injections, and the epimorphisms are
the surjections. A morphism $f : C \to D$ in $\cat{Set}$ that is both a
monomorphism and an epimorphism is an isomorphism. The same holds in 
$\cat{Mod}\, A$. 
\end{exa}

\begin{rem}
The property of being a monomorphism or an epimorphism is sensitive to the
category in question. For instance, consider the category of rings
$\cat{Ring}$. The forgetful functor 
$\cat{Ring} \to \cat{Set}$ respects monomorphisms, but it does not respect
epimorphisms. The easiest example is inclusion $\Z \to \mbb{Q}$, which is an 
epimorphism in $\cat{Ring}$.  
\footnote{In a previous version we claimed that this happens also for the 
category of groups $\cat{Grp}$; but this is false. We thank Vincent Beck for 
this correction.}
\end{rem}

By a {\em subobject} $C'$ of an object $C$ we mean that there is given a
monomorphism $f : C' \to C$. We sometimes write $C' \subset C$ in this
situation, but this is only notational (and does not mean inclusion of sets). 
Likewise, by a {\em quotient} $\bar{C}$ of $C$ we
mean that there is given an epimorphism $g : C \to \bar{C}$. 

An {\em initial object} in a category $\cat{C}$ is an object $C_0 \in \cat{C}$,
such that for every object $C \in \cat{C}$ there is exactly one morphism 
$C_0 \to C$. Thus the set $\opn{Hom}_{\cat{C}}(C_0, C)$ is a singleton.
An {\em terminal object} in $\cat{C}$ is an object $C_{\infty} \in \cat{C}$,
such that for every object $C \in \cat{C}$ there is exactly one morphism 
$C \to C_{\infty}$.

\begin{dfn}
A {\em zero object} in a category $\cat{C}$ is an object which is both initial
and terminal. 
\end{dfn}

Initial, terminal and zero objects are unique up to unique isomorphisms (but
they need not exist).

\begin{exa}
In $\cat{Set}$, $\emptyset$ is an initial object, and any singleton is a
terminal object. There is no zero object.
\end{exa}

\begin{exa}
In $\cat{Mod}\, A$, any trivial module (with only the zero element) is a 
zero object, and we denote this module by $0$. This is allowed, since any other
zero module is uniquely isomorphic to it.  
\end{exa}

\subsection{Products and Coproducts}
Let $\cat{C}$ be a category. For a collection $\{ C_i \}_{i \in I}$ of objects
of $\cat{C}$, indexed by a set $I$, their {\em product}
is a pair $(C, \{ p_i \}_{i \in I})$
consisting of an object $C$ and morphisms $p_i : C \to C_i$. 
The morphisms $p_i : C \to C_i$ are called projections. 
The pair $(C, \{ p_i \}_{i \in I})$ 
must have this universal property: given any object $D$ and morphisms 
$f_i : D \to C_i$, there is a unique morphism $f : D \to C$  s.t.\ 
$f_i = p_i \circ f$. Of course if a product
$(C, \{ p_i \}_{i \in I})$ exists then it is unique up to a
unique isomorphism; and we write
$\prod_{i \in I} C_i := C$.

\begin{exa}
In $\cat{Set}$ and $\cat{Mod}\, A$ all products (indexed by small sets) exist,
and they are the usual cartesian products. 
\end{exa}

For a collection $\{ C_i \}_{i \in I}$ of objects
of $\cat{C}$, their {\em coproduct}
is a pair $(C, \{ e_i \}_{i \in I})$
consisting of an object $C$ and morphisms $e_i : C_i \to C$. 
The morphisms $e_i : C_i \to C$ are called embeddings. 
The pair $(C, \{ e_i \}_{i \in I})$ 
must have this universal property: given any object $D$ and morphisms 
$f_i : C_i \to D$, there is a unique morphism $f : C \to D$  s.t.\ 
$f_i = f \circ e_i$. If a product
$(C, \{ e_i \}_{i \in I})$ exists then it is unique up to a
unique isomorphism; and we write
$\coprod_{i \in I} C_i := C$.

\begin{exa} \label{exa:103}
In $\cat{Set}$ the coproduct is the disjoint union. 
In $\cat{Mod}\, A$ the coproduct is the direct sum. 
\end{exa}

\subsection{Equivalence}
Recall that a functor $F : \cat{C} \to \cat{D}$ is an {\em equivalence} if
there exists a functor $G : \cat{D} \to \cat{C}$, and natural isomorphisms 
$G \circ F \cong \bsym{1}_{\cat{C}}$ and 
$F \circ G \cong \bsym{1}_{\cat{D}}$. Such a functor $G$ is called a {\em
quasi-inverse} of $F$. 

We know that $F : \cat{C} \to \cat{D}$ is an  equivalence iff these two
conditions hold:
\begin{enumerate}
\rmitem{i} $F$ is essentially surjective on objects. This means that for every
$D \in \cat{D}$  there is some $C \in \cat{C}$ and an isomorphism 
$F(C) \iso D$. 

\rmitem{ii} $F$ is fully faithful. This means that for every 
$C_0, C_1 \in \cat{C}$ the function 
\[ F : \opn{Hom}_{\cat{C}}(C_0, C_1) \to 
 \opn{Hom}_{\cat{D}}(F(C_0), F(C_1)) \]
is bijective. 
\end{enumerate}

\subsection{Bifunctors} \label{subsec:bifunc}

Let $\cat{C}$ and $\cat{D}$ be categories. Their product is the category 
$\cat{C} \times \cat{D}$ defined as follows: the set of objects is 
\[ \opn{Ob}(\cat{C} \times \cat{D}) := 
\opn{Ob}(\cat{C}) \times \opn{Ob}(\cat{D}) . \] 
The sets of morphisms are 
\[ \opn{Hom}_{\cat{C} \times \cat{D}} \big( (C_0, D_0), (C_1, D_1) \big) :=
\opn{Hom}_{\cat{C} } (C_0, C_1) \times 
\opn{Hom}_{\cat{D} } (D_0, D_1)  . \]
The composition is 
\[ (\phi_1, \psi_1) \circ (\phi_0, \psi_0) := 
(\phi_1 \circ \phi_0, \psi_1 \circ \psi_0) , \]
and the identity morphisms are $(1_C, 1_D)$. 

A {\em bifunctor} 
\[ F : \cat{C} \times \cat{D} \to \cat{E} \]
is by definition a functor from the product category $\cat{C} \times \cat{D}$
to $\cat{E}$. 
We say ``bifunctor'' because it is a functor of two arguments:
$F(C, D) \in \cat{E}$.

\cleardoublepage
\section{Abelian Categories}

\subsection{Linear categories} \label{subsec:1}

\begin{dfn}
Let $\K$ be a commutative ring. 
A  {\em $\K$-linear category} is a category $\cat{A}$, endowed with 
a $\K$-module structure on each of the morphism sets 
$\opn{Hom}_{\cat{A}}(M_0, M_1)$
for all $M_0, M_1 \in \cat{A}$. The condition is this:
\begin{itemize}
\item For all $M_0, M_1, M_2 \in \cat{A}$ the composition function 
\[ \opn{Hom}_{\cat{A}}(M_0, M_1) \times \opn{Hom}_{\cat{A}}(M_1, M_2)
\xar{\circ} \opn{Hom}_{\cat{A}}(M_0, M_2) \]
is $\K$-bilinear.
\end{itemize}

If $\K = \Z$ we say that $\cat{A}$ is a {\em linear category}.
\end{dfn}

Observe that for any object $M$ of a $\K$-linear category $\cat{A}$, the set 
\[ \opn{End}_{\cat{A}}(M) := \opn{Hom}_{\cat{A}}(M, M) \]
is a $\K$-algebra. In these notes a $\K$-algebra $A$ is (by default) unital
and associative; so in fact $A$ is a ring, together with a ring homomorphism
from $\K$ to the center of $A$.

This observation can be reversed:

\begin{exa} \label{exa:104}
Let $A$ be a $\K$-algebra.  
Define a category $\cat{A}$ like this: there is a single object $M$, and its
set of morphisms is
\[ \opn{Hom}_{\cat{A}}(M, M) := A . \]
Composition in $\cat{A}$ is the multiplication of $A$. Then $\cat{A}$ is a
$\K$-linear category.
\end{exa}

\subsection{Additive categories}

\begin{dfn} 
An  {\em additive category} is a linear category $\cat{M}$
satisfying these conditions:
\begin{itemize}
\rmitem{i} $\cat{M}$ has a zero object $0$.

\rmitem{ii} $\cat{M}$ has finite coproducts.
\end{itemize}
\end{dfn}

Observe that $\opn{Hom}_{\cat{M}}(M, N) \neq \emptyset$, since this is an
abelian group. Also 
\[ \opn{Hom}_{\cat{M}}(M, 0) = \opn{Hom}_{\cat{M}}(0, M) = 0 , \] 
the zero abelian group.  We denote the unique arrows $0 \to M$ and 
$M \to 0$ also by $0$. So the numeral $0$ has a lot of meanings; but they are
clear from the contexts. 
The coproduct in the additive category $\cat{M}$ is denoted by $\bigoplus$;
cf.\ Example \ref{exa:103}.

\begin{exa}
Let $A$ be a ring. The category $\cat{Mod}\, A$ is additive. 
The full subcategory $\cat{M} \subset \cat{Mod}\, A$ on the free modules is
also additive. 
\end{exa}

for an object $M$ we denote by $1_M : M \to M$ the identity morphism.

\begin{prop} \label{prop:1}
Let $\cat{M}$ be an additive category. Let $\{ M_i \}_{i \in I}$ be a  
finite collection of objects of $\cat{M}$, and let 
$M := \bigoplus_{i \in I} M_i$ be the coproduct, with embeddings 
$e_i : M_i \to M$. 

\begin{enumerate}
\item For any $i$ let $p_i : M \to M_i$ be the unique morphism 
s.t.\ $p_i \circ e_i = 1_{M_i}$, and 
 $p_i \circ e_j = 0$ for $j \neq i$. 
Then $(M, \{ p_i \}_{i \in I})$ is a product of the collection 
$\{ M_i \}_{i \in I}$.

\item $\sum_{i \in I} e_i \circ p_i = 1_M$.
\end{enumerate}
\end{prop}

\begin{proof}
Exercise.
\end{proof}

\begin{exa}
One could ask if the linear category $\cat{A}$ from Example \ref{exa:104},
built from a ring $A$, is
additive, i.e.\ does it have finite direct sums? It appears that this depends
on whether or not $A \cong A \oplus A$ as left $A$-modules. Thus if $A$ is
nonzero and commutative, or nonzero and noetherian, then this is false. 
On the other hand if we take a field $\K$, and a countable rank $\K$-module
$M$, then $A := \opn{End}_{\K}(M)$ will satisfy $A \cong A \oplus A$.
\end{exa}

\subsection{Abelian categories}

\begin{dfn}
Let $\cat{M}$ be an additive category, and let $f : M \to N$ be a morphism in
$\cat{M}$. A {\em kernel} of $f$ is a pair $(K, k)$, consisting of an object
$K \in \cat{M}$ and a morphism $k : K \to M$, with these properties:
\begin{enumerate}
\rmitem{i}  $f \circ k = 0$. 
\rmitem{ii} If $k' : K' \to M$ is a morphism in $\cat{M}$ such that 
$f \circ k' = 0$, then there is a unique morphism $g : K' \to K$ 
such that $k' = k \circ g$.
\end{enumerate}
\end{dfn}

In other words, the object $K$ represents the functor 
$\cat{M}^{\mrm{op}} \to \cat{Ab}$, 
\[ K' \mapsto \{ k' \in \opn{Hom}_{\cat{M}}(K', M) \mid f \circ k' = 0 \} . \]
The kernel of $f$ is of course unique up to a unique isomorphism (if it
exists), and we denote if by $\opn{Ker}(f)$. Sometimes $\opn{Ker}(f)$ refers
only to the object $K$, and other times it refers only to the morphism
$k$.

\begin{dfn}
Let $\cat{M}$ be an additive category, and let $f : M \to N$ be a morphism in
$\cat{M}$. A {\em cokernel} of $f$ is a pair $(C, c)$, consisting of an object
$C \in \cat{M}$ and a morphism $c : N \to C$, with these properties:
\begin{enumerate}
\rmitem{i}  $c \circ f = 0$. 
\rmitem{ii} If $c' : N \to C'$ is a morphism in $\cat{M}$ such that 
$c' \circ f = 0$, then there is a unique morphism $g : C \to C'$ 
such that $c' = g \circ c$.
\end{enumerate}
\end{dfn}

The cokernel $\opn{Coker}(f)$ is unique up to a unique isomorphism.

\begin{exa}
In $\cat{Mod}\, A$ all kernels and cokernels exist. 
Given $f : M \to N$, the kernel is 
$k : K \to M$, where 
\[ K := \{ m \in M \mid f(m) = 0 \} , \]
and the $k$ is the inclusion. The cokernel is $c : N \to C$, where 
$C := N / f(M)$, and $c$ is the canonical projection.
\end{exa}

\begin{prop}
Let $f : M \to N$ be a morphism, let $k : K \to M$ be a kernel of $f$,
and let $c : N \to C$ be a cokernel of $f$. Then $k$  is a monomorphism, and
$c$ is an epimorphism.
\end{prop}

\begin{proof}
Exercise.
\end{proof}

\begin{dfn}
Assume the additive category $\cat{M}$ has kernels and cokernels. 
Let $f : M \to N$ be a morphism in $\cat{M}$.
\begin{enumerate}
\item Define the {\em image} of $f$ to be 
\[ \opn{Im}(f) := \opn{Ker}(\opn{Coker}(f)) . \]

\item Define the {\em coimage} of $f$ to be 
\[ \opn{Coim}(f) := \opn{Coker}(\opn{Ker}(f)) . \]
\end{enumerate}
\end{dfn}

Consider the following commutative diagram (solid arrows):
\[ \UseTips  \xymatrix @C=6ex @R=6ex {
K
\ar[r]^{k}
\ar[dr]_{0}
& 
M
\ar[r]^{f}
\ar[d]_{\al}
\ar@{-->}[dr]^{\ga}
&
N
\ar[r]^{c}
&
C
\\
&
M'
\ar@{-->}[r]_{f'}
&
N'
\ar[ur]_{0}
\ar[u]_{\be}
} \]
where $\al = \opn{Coker}(k) = \opn{Coim}(f)$ and 
$\be = \opn{Ker}(c) = \opn{Im}(f)$.
Since $c \circ f = 0$ there is a unique morphism $\ga$ 
making the diagram commutative. Now 
$\be \circ \ga \circ k = f \circ k = 0$; and $\be$ is a monomorphism; so 
$\ga \circ k = 0$. Hence there is a unique morphism $f' : M' \to N'$ 
making the diagram commutative.
We conclude that $f : M \to N$ induces a morphism 
\begin{equation}  \label{eqn:2}
f' : \opn{Coim}(f) \to \opn{Im}(f) . 
\end{equation}

\begin{dfn} \label{dfn:1}
An {\em abelian category} is an additive category $\cat{M}$ with these extra
properties:
\begin{enumerate}
\rmitem{i} All morphisms in $\cat{M}$ admit kernels and cokernels.
\rmitem{ii} For any $f : M \to N$ in $\cat{M}$ the induced morphism $f'$ of
equation (\ref{eqn:2}) is an isomorphism.
\end{enumerate}
\end{dfn}

A less precise but (maybe) easier to remember way to state property (ii) is:
\[ \opn{Ker}(\opn{Coker}(f)) = \opn{Coker}(\opn{Ker}(f)) . \]
{}From now on we forget all about the coimage. 

\begin{exa}
The category $\cat{Mod}\, A$ is abelian. 
\end{exa}

\begin{dfn}
Let $\cat{M}$ be an abelian category, and let $\cat{N}$ be a full subcategory
of $\cat{M}$. We say that $\cat{N}$ is a {\em full abelian subcategory} of
$\cat{M}$ if $\cat{N}$ is closed under direct sums, kernels and cokernels.
\end{dfn}

\begin{exa}
Let $\cat{M}_1$ be the category of finitely generated abelian groups,
and let $\cat{M}_0$ be the category of finite abelian groups. 
Then $\cat{M}_0$ is a full abelian subcategory of $\cat{M}_1$, and 
 $\cat{M}_1$ is a full
abelian subcategory of $\cat{Ab}$. 
\end{exa}

\begin{exa}
Let $\cat{N}$ be the full subcategory of $\cat{Ab}$ whose objects are the
finitely generated free abelian groups. It is an additive subcategory of
$\cat{Ab}$  (since it is closed under direct sums), but clearly it is not a full
abelian subcategory,
since it is not closed under cokernels. 

What is more interesting is that the additive category $\cat{N}$ does have its
own intrinsic cokernels, but still it fails to be an abelian category.
\end{exa}

\begin{exa}
A ring $A$ is {\em left noetherian} iff the category $\cat{Mod}_{\mrm{f}}\, A$
of  finitely generated modules is a full abelian subcategory of 
$\cat{Mod}\, A$. Here the issue is kernels.
\end{exa}

\begin{exa} \label{exa:1}
Let $(X, \mcal{A})$ be a ringed space; namely $X$ is a topological space and
$\mcal{A}$ is a sheaf of rings on $X$. 
We denote by $\cat{PMod}\, \mcal{A}$ the category of presheaves of left 
$\mcal{A}$-modules on $X$. This is an abelian category. 
Given a morphism $f : \mcal{M} \to \mcal{N}$ in $\cat{PMod}\, \mcal{A}$, its
kernel is the presheaf $\mcal{K}$ defined by 
\[ \Gamma(U, \mcal{K}) := 
\opn{Ker} \big( f : \Gamma(U, \mcal{M}) \to \Gamma(U, \mcal{N}) \big) . \]
The cokernel is the  presheaf $\mcal{C}$ defined by 
\[ \Gamma(U, \mcal{C}) := 
\opn{Coker} \big( f : \Gamma(U, \mcal{M}) \to \Gamma(U, \mcal{N}) \big) . \]

Now let $\cat{Mod}\, \mcal{A}$ be the full subcategory of 
$\cat{PMod}\, \mcal{A}$ consisting of sheaves. 
We know that $\cat{Mod}\, \mcal{A}$ is not closed under cokernels inside 
$\cat{PMod}\, \mcal{A}$, and hence it is not a full abelian subcategory.

However $\cat{Mod}\, \mcal{A}$ is itself an abelian category, but
with different cokernels. 
Indeed, for a morphism $f : \mcal{M} \to \mcal{N}$ in $\cat{Mod}\, \mcal{A}$,
its cokernel $\opn{Coker}_{\cat{Mod}\, \mcal{A}}(f)$ is the 
sheafification of the presheaf 
$\opn{Coker}_{\cat{PMod}\, \mcal{A}}(f)$. 
\end{exa}

For educational purposes we state:

\begin{thm}[Freyd \& Mitchell]
Let $\cat{M}$ be a small abelian category. Then $\cat{M}$ is equivalent to a
full abelian subcategory of $\cat{Mod}\, A$, for a suitable ring $A$. 
\end{thm}
 
This means that most of the time we can pretend that 
$\cat{M} \subset \cat{Mod}\, A$; this could be a helpful heuristic. 

\begin{prop}
\begin{enumerate}
\item Let $\cat{M}$ be an additive category. Then the opposite category 
 $\cat{M}^{\mrm{op}}$ is also additive.

\item Let $\cat{M}$ be an abelian category. Then the opposite category 
 $\cat{M}^{\mrm{op}}$ is also abelian. 
\end{enumerate}
\end{prop}

\begin{proof}
(1) First note that
\[ \opn{Hom}_{\cat{M}^{\mrm{op}}}(M, N) = 
\opn{Hom}_{\cat{M}}(N, M) , \]
so this is an abelian group. 
The bilinearity of the composition in $\cat{M}^{\mrm{op}}$ is clear, and the
zero objects are the same. Existence of finite coproducts in 
$\cat{M}^{\mrm{op}}$ is because of existence of finite products in $\cat{M}$;
see Proposition \ref{prop:1}(1). 

\medskip \noindent 
(2) $\cat{M}^{\mrm{op}}$ has kernels and cokernels, since 
$\opn{Ker}_{\cat{M}^{\mrm{op}}}(f) = \opn{Coker}_{\cat{M}}(f)$ and vice versa. 
Also the symmetric condition (ii) of Definition \ref{dfn:1} holds. 
\end{proof}

\begin{prop}
Let $f : M \to N$ be a morphism in an abelian category $\cat{M}$. 
\begin{enumerate}
\item $f$ is a monomorphism iff $\opn{Ker}(f) = 0$. 
\item $f$ is an epimorphism iff $\opn{Coker}(f) = 0$. 
\item $f$ is an isomorphism iff it is both a monomorphism and an epimorphism.
\end{enumerate}
\end{prop}

\begin{proof}
Exercise.
\end{proof}

\subsection{Additive Functors}

\begin{dfn}
Let $\cat{M}$ and $\cat{N}$ be $\K$-linear categories. A functor 
$F : \cat{M} \to \cat{N}$ is called a {\em $\K$-linear functor} if 
for every $M_0, M_1 \in \cat{M}$ the function 
\[ F : \opn{Hom}_{\cat{M}}(M_0, M_1) \to
\opn{Hom}_{\cat{N}}(F(M_0), F(M_1)) \]
is a $\K$-linear homomorphism. 

A $\Z$-linear functor is also called an {\em additive functor}. 
\end{dfn}

Additive functors commute with finite direct sums. More precisely:

\begin{prop} \label{prop:2}
Let $F : \cat{M} \to \cat{N}$ be an additive functor between linear categories,
let $\{ M_i \}_{i \in I}$ be a finite collection of objects of $\cat{M}$, and
assume that  the direct sum $(M, \{ e_i \}_{i \in I})$ of the collection 
$\{ M_i \}_{i \in I}$ exists in $\cat{M}$. Then 
$\big( F(M), \{ F(e_i) \}_{i \in I} \big)$ is a direct sum of the collection 
$\{ F(M_i) \}_{i \in I}$ in $\cat{N}$.  
\end{prop}

\begin{proof}
Exercise. (Hint: use Proposition \ref{prop:1}.)
\end{proof}

\begin{exa}
Let $A \to B$ be a ring homomorphism. The corresponding forgetful functor 
\[ F : \cat{Mod}\, B \to  \cat{Mod}\, A \]
(also called restriction of scalars) is additive. 
The functor 
\[ G : \cat{Mod}\, A \to  \cat{Mod}\, B \]
defined by 
$G(M) := B \ot_A M$, called extension of scalars, is also additive. 
\end{exa}

\begin{prop} \label{prop:109}
Let $F : \cat{M} \to \cat{N}$ be an additive functor between additive
categories. Then $F(0_{\cat{M}}) = 0_{\cat{N}}$.
\end{prop}

\begin{proof}
For any object $M \in \cat{M}$ we have a ring $\opn{End}_{\cat{M}}(M)$; and 
$\opn{Hom}_{\cat{A}}(M_0, M_1)$ is an 
$\opn{End}_{\cat{M}}(M_1)$-$\opn{End}_{\cat{M}}(M_0)$-bimodule.
An object $M \in \cat{M}$ is a zero object iff
$\opn{End}_{\cat{M}}(M)$ is the zero ring, i.e.\ 
$1 = 0$ in $\opn{End}_{\cat{M}}(M)$.

Now $F : \opn{End}_{\cat{M}}(M) \to \opn{End}_{\cat{N}}(F(M))$ is a ring
homomorphism, so it sends the zero ring to the zero ring. 
\end{proof}

\begin{dfn}
Let $F : \cat{M} \to \cat{N}$ be an additive functor between abelian
categories. 
\begin{enumerate}
\item $F$ is called {\em left exact} if it commutes with kernels. 
Namely for any morphism $\phi : M_0 \to M_1$ in $\cat{M}$, with 
kernel $k : K \to M_0$, the morphism $F(k) : F(K) \to F(M_0)$ is 
a kernel of $F(\phi) : F(M_0) \to F(M_1)$.

\item $F$ is called {\em right exact} if it commutes with cokernels. 
Namely for any morphism $\phi : M_0 \to M_1$ in $\cat{M}$, with 
cokernel $c : M_1 \to C$, the morphism $F(c) : F(M_1) \to F(C)$ is 
a cokernel of $F(\phi) : F(M_0) \to F(M_1)$.

\item $F$ is called {\em exact} if it both left exact and right exact.
\end{enumerate}
\end{dfn}

This is illustrated in the following diagrams. Suppose $\phi : M_0 \to M_1$
is a morphism in $\cat{M}$, with  kernel $K$ and cokernel $C$. Applying $F$ to
the diagram 
\[ \UseTips  \xymatrix @C=6ex @R=6ex {
K
\ar[r]^{k}
& 
M_0
\ar[r]^{\phi}
&
M_1
\ar[r]^{c}
&
C
} \]
we get the solid arrows in
\[ \UseTips  \xymatrix @C=6ex @R=6ex {
F(K)
\ar[r]^{F(k)}
\ar@{-->}[dr]_{\psi}
& 
F(M_0)
\ar[r]^{F(\phi)}
&
F(M_1)
\ar@{-->}[d]
\ar[r]^{F(c)}
&
F(C)
\\
&
\opn{Ker}_{\cat{N}}(F(\phi))
\ar@{-->}[u]
&
\opn{Coker}_{\cat{N}}(F(\phi))
\ar@{-->}[ur]_{\chi}
} \]
The dashed arrows are from the structure of $\cat{N}$. 
Left exactness requires $\psi$ to be an isomorphism, and 
right exactness requires $\chi$ to be an isomorphism.

\begin{dfn}
Let $\cat{M}$ be an abelian category. An {\em exact sequence} in $\cat{M}$ is a
diagram 
\[ \cdots M_0 \xar{\phi_0}  M_1 \xar{\phi_1}  M_2 \cdots \]
  (finite or infinite on either side) s.t.\ 
$\opn{Ker}(\phi_i) = \opn{Im}(\phi_{i-1})$ for all $i$ (for which 
$\phi_{i}$ and $\phi_{i-1}$ are defined). 
\end{dfn}

As usual, a {\em short exact sequence} is one of the form 
\begin{equation} \label{eqn:3}
0 \to M_0 \to M_1 \to M_2 \to 0 .  
\end{equation}

\begin{prop}
Let $F : \cat{M} \to \cat{N}$ be an additive functor between abelian
categories. 
\begin{enumerate}
\item The functor $F$ is left exact  iff
for every short exact sequence \tup{(\ref{eqn:3})} in $\cat{M}$, the sequence 
\[ 0 \to F(M_0) \to F(M_1) \to F(M_2)  \]
is exact in $\cat{N}$.

\item The functor $F$ is right exact  iff
for every short exact sequence \tup{(\ref{eqn:3})} in $\cat{M}$, the sequence 
\[ F(M_0) \to F(M_1) \to F(M_2) \to 0 \]
is exact in $\cat{N}$.
\end{enumerate}
\end{prop}

\begin{proof}
Exercise. (Hint: $M_0 \cong \opn{Ker}(M_1 \to M_2)$ etc.)
\end{proof}

\begin{exa}
Let $A$ be a commutative ring, and let $M$ be a fixed $A$-module. 
Define functors 
$F, G : \cat{Mod}\, A \to \cat{Mod}\, A$
and 
$H : (\cat{Mod}\, A)^{\mrm{op}} \to \cat{Mod}\, A$
like this:
$F(N) := M \ot_A N$, $G(N) := \opn{Hom}_A(M, N)$ and 
$H(N) := \opn{Hom}_A(N, M)$. Then $F$ is right exact, and $G$ and $H$ are left
exact.
\end{exa}

\begin{prop}
Let $F : \cat{M} \to \cat{N}$ be an additive functor between abelian
categories. If $F$ is an equivalence then it is exact.
\end{prop}

\begin{proof}
We will prove that $F$ respects kernels; the proof for cokernels is similar.
Take a morphism $\phi : M_0 \to M_1$ in $\cat{M}$, with kernel $K$. 
We have this diagram (solid arrows):
\[ \UseTips  \xymatrix @C=6ex @R=6ex {
M 
\ar@{-->}[d]_{\psi}
\ar@{-->}[dr]^{\theta}
\\
K
\ar[r]^{k}
& 
M_0
\ar[r]^{\phi}
&
M_1
} \]
Applying $F$ we obtain 
 this diagram (solid arrows):
\[ \UseTips  \xymatrix @C=6ex @R=6ex {
N = F(M)
\ar@{-->}[d]_{F(\psi)}
\ar@{-->}[dr]^{\bar{\theta}}
\\
F(K)
\ar[r]^{F(k)}
& 
F(M_0)
\ar[r]^{F(\phi)}
&
F(M_1)
} \]
in $\cat{N}$. 
Suppose $\bar{\theta} : N \to F(M_0)$ is a morphism in $\cat{N}$ s.t.\ 
$F(\phi) \circ \bar{\theta} = 0$.
Since $F$ is essentially surjective on objects, there is some $M \in \cat{M}$
with an isomorphism $\al : F(M) \iso N$. After replacing $N$ with $F(M)$ and 
$\bar{\theta}$ with $\bar{\theta} \circ \al$, we can assume that  
$N = F(M)$. 

Now since $F$ is fully faithful, there is a unique $\theta : M \to M_0$ s.t.\ 
$F(\theta) = \bar{\theta}$; and $\phi \circ \theta = 0$. 
So there is a unique $\psi : M \to K$ s.t.\  
$\theta = k \circ \psi$. It follows that 
$F(\psi) : F(M) \to F(M_0)$ is the unique morphism s.t.\ 
$\bar{\theta} = F(k) \circ F\psi)$.
\end{proof}

Here is a result that could afford another proof of the previous proposition. 

\begin{prop}
Let $F : \cat{M} \to \cat{N}$ be an additive functor between linear
categories. The following conditions are equivalent:
\begin{enumerate}
\rmitem{i} The functor $F$ has a quasi-inverse.
\rmitem{ii} The functor $F$ has an additive quasi-inverse.
\end{enumerate}
\end{prop}

\begin{proof}
Exercise. 
\end{proof}

\cleardoublepage
\section{Projective and Injective Objects}

Here $\cat{M}$ is an abelian category.

\subsection{Projectives}
A {\em splitting} of an epimorphism 
$\psi : M \to M''$ in $\cat{M}$ is a morphism
$\al : M'' \to M$ s.t.\ $\psi \circ \al = 1_{M''}$.
A splitting of a monomorphism 
$\phi : M' \to M$ is  a morphism $\be : M \to M'$ s.t.\ 
$\be \circ \phi = 1_{M'}$.
A splitting of a short exact sequence 
\[ 0 \to M' \xar{\phi} M \xar{\psi} M'' \to 0  \]
is a splitting of the epimorphism $\psi$, or equivalently a splitting of the
monomorphism $\phi$. The short exact sequence is said to be {\em split} if it
has some splitting. 

\begin{dfn}
An object $P \in \cat{M}$ is called a {\em projective object} if any diagram 
(solid arrows)
\[ \UseTips  \xymatrix @C=6ex @R=6ex {
&
P
\ar[d]^{\ga}
\ar@{-->}[dl]_{\til{\ga}}
\\
M 
\ar[r]_{\psi}
&
N
} \]
in which $\psi$ is an epimorphism, can be completed (dashed arrow). 
\end{dfn}

\begin{prop}
The following conditions are equivalent for $P \in \cat{M}$\tup{:}
\begin{enumerate}
\rmitem{i} $P$ is projective.
\rmitem{ii} The additive functor 
\[ \opn{Hom}_{\cat{M}}(P, -) : \cat{M} \to \cat{Ab} \]
is exact.
\end{enumerate}
\end{prop}

\begin{proof}
Exercise.
\end{proof}

\begin{dfn}
We say $\cat{M}$ {\em has enough projectives} if every $M \in \cat{M}$ admits
an epimorphism $P \to M$ with $P$ a projective object. 
\end{dfn}

\begin{exa}
Let $A$ be a ring. An $A$-module $P$ is projective iff it is a direct summand of
a free module; i.e.\ $P \oplus P' \cong Q$ for some module $P'$ and free module
$Q$. The category $\cat{Mod}\, A$ has enough projectives.
\end{exa}

\begin{exa}
Let $\cat{M}$ be the category of finite abelian groups. The only projective
object in $\cat{M}$ is $0$. So $\cat{M}$ does not have enough projectives. 
\end{exa}

\begin{exa}
Consider the scheme $X := \mbf{P}^1_{\K}$, the projective line over a field
$\K$ (we can assume $\K$ is algebraically closed, so this is a classical
algebraic variety). The structure sheaf (sheaf of functions) is $\mcal{O}_X$. 
The category $\cat{Coh}\, \mcal{O}_X$ of coherent $\mcal{O}_X$-modules is
abelian (it is a full abelian subcategory of $\cat{Mod}\, \mcal{O}_X$, cf.\ 
Example \ref{exa:1}). One can show that the only projective object of 
$\cat{Coh}\, \mcal{O}_X$ is $0$, but this is quite involved. 

Let us only indicate why  $\mcal{O}_X$ is not projective. 
Denote by $t_0, t_1$ the homogenous coordinates of $X$. These belong to 
$\Gamma(X, \mcal{O}_X(1))$, so each determines a homomorphism of sheaves
$t_j : \mcal{O}_X(i) \to \mcal{O}_X(i+1)$. We get a sequence 
\[ 0 \to \mcal{O}_X(-2) \xar{\sbmat{t_0 & -t_1}} \mcal{O}_X(-1)^2 
\xar{\sbmat{t_0 \\ t_1}} \mcal{O}_X \to 0  \]
in $\cat{Coh}\, \mcal{O}_X$, which is known to be exact, and also not split. 
\end{exa}

\subsection{Injectives}
\begin{dfn}
An object $I \in \cat{M}$ is called an {\em injective object} if any diagram 
(solid arrows)
\[ \UseTips  \xymatrix @C=6ex @R=6ex {
I
\\
M 
\ar[u]^{\ga}
\ar[r]_{\psi}
&
N
\ar@{-->}[ul]_{\til{\ga}}
} \]
in which $\psi$ is a monomorphism, can be completed (dashed arrow). 
\end{dfn}

\begin{prop}
The following conditions are equivalent for $I \in \cat{M}$\tup{:}
\begin{enumerate}
\rmitem{i} $I$ is injective.
\rmitem{ii} The additive functor 
\[ \opn{Hom}_{\cat{M}}(-, I) : \cat{M}^{\mrm{op}} \to \cat{Ab} \]
is exact.
\end{enumerate}
\end{prop}

\begin{proof}
Exercise.
\end{proof}

\begin{exa}
Let $A$ be a ring. Unlike projectives, the structure of injective objects in
$\cat{Mod}\, A$ is very complicated, and not much is known (except that they
exist). However if $A$ is a commutative noetherian ring then we know this:
every injective module $I$ is a direct sum of indecomposable injective modules. 
And these indecomposables are parametrized by $\opn{Spec} A$, the set of prime
ideals of $A$. These facts are due to Matlis; see \cite[pages 120-122]{RD} for
details.
\end{exa}

\begin{dfn}
We say $\cat{M}$ {\em has enough injectives} if every $M \in \cat{M}$ admits
a monomorphism $M \to I$ with $I$ an  injective object. 
\end{dfn}

Here are a few results about injective objects. 

\begin{prop}
Let $f : A \to B$ be a ring homomorphism, and let $I$ be an injective left
$A$-module. Then 
$J := \opn{Hom}_A(B, I)$ is an  injective left $B$-module.
\end{prop}

\begin{proof}
Note that $B$ is a left $A$-module via $f$, and a right $B$-module. This makes
$J$ into a left $B$-module. In a formula: for $\phi \in J$ and $b, b' \in B$ we
have
$(b \phi)(b') = \phi(b' b)$. 

Now given any $N \in \cat{Mod}\, B$ there is an isomorphism 
\begin{equation} \label{eqn:4}
\opn{Hom}_B(N, J) =  \opn{Hom}_B(N, \opn{Hom}_A(B, I)) \cong 
\opn{Hom}_A(N, I) .
\end{equation}
 This is a natural isomorphism (of functors in $N$). So the functor 
$\opn{Hom}_B(-, J)$ is exact, and hence $J$ is injective. 
\end{proof}

We quote the following result:

\begin{thm}[Baer Criterion] \label{thm:151}
Let $A$ be a ring and $I$ a left $A$-module. $I$ is injective iff for every
left ideal $\a \subset A$, every homomorphism $\ga : \a \to I$ extends to a
homomorphism $\til{\ga} : A \to I$.
\end{thm}

\begin{lem}
The $\Z$-module $\mbb{Q} / \Z$ is injective. 
\end{lem}

\begin{proof}
By the Baer criterion, it is enough to consider a homomorphism
$\ga : \a \to \mbb{Q} / \Z$ for $\a = n \Z \subset \Z$. We may assume that 
$n \neq 0$. Say $\ga(n) = r + \Z$ with $r \in \mbb{Q}$. 
Then we can extend $\ga$ to $\til{\ga} : \Z \to \mbb{Q} / \Z$ with 
$\til{\ga}(1) := r / n + \Z$. 
\end{proof}

\begin{exer}
Try proving this lemma directly, without using the Baer criterion.
\end{exer}

\begin{lem}
Let $\{ I_x \}_{x \in X}$ be a collection of injective objects of $\cat{M}$. If
the product $I := \prod_{x \in X} I_x$ exists, then it is an injective object.
\end{lem}

\begin{proof}
Exercise.
\end{proof}

\begin{thm}
Let $A$ be any ring. The category $\cat{Mod}\, A$ has enough injectives.
\end{thm}

\begin{proof}
Step 1. Here $A = \Z$. Take any nonzero $\Z$-module $M$ and any nonzero 
$m \in M$. Consider the cyclic submodule $M' := \Z m \subset M$. 
There is a homomorphism $\ga : M' \to \mbb{Q} / \Z$ s.t.\ 
$\ga(m) \neq 0$. Indeed, if $M' \cong \Z$ then we can take any 
$r \in \mbb{Q} - \Z$ and define $\ga(m) := r + \Z \in \mbb{Q} / \Z$.
If $M' \cong \Z / (n)$ for some $n \neq 0$, then we take 
$r := 1 / n$. Since $\mbb{Q} / \Z$ is an injective $\Z$-module, $\ga$ extends to
a homomorphism $\til{\ga} : M \to \mbb{Q} / \Z$.

\medskip \noindent 
Step 2. Again $A = \Z$. Let $M$ be a nonzero $\Z$-module. By step 1, for any
nonzero  $m \in M$ there is a homomorphism $\phi_m : M \to \mbb{Q} / \Z$ s.t.\ 
$\phi_m(m) \neq 0$.  
Define the $\Z$-module $I := \prod_{m \in M - \{ 0 \}} (\mbb{Q} / \Z)$,
and the homomorphism $\phi := \prod_m \phi_m : M \to I$. 
Then $\phi$ is a monomorphism, and $I$ is an injective $\Z$-module. 

\medskip \noindent 
Step 3. Now $A$ is any ring, and $M$ is any (left) $A$-module. We may assume 
$M \neq 0$. Viewing $M$ as a $\Z$-module, choose any embedding  
$\phi : M \to I$ into an injective $\Z$-module $I$. 
Let $J := \opn{Hom}_{\Z}(A, I)$, which is an injective $A$-module. 
Let $\tau : J \to I$ be the $\Z$-module homomorphism that
sends an element $\chi \in J$ to $\chi(1) \in I$.
The adjunction formula (\ref{eqn:4}) gives a unique $A$-module homomorphism 
$\psi : M \to J$ s.t.\ $\tau \circ \psi = \phi$. 
This $\psi$ is a monomorphism.
\end{proof}

\begin{exa}
Let $\cat{N}$ be the category of torsion abelian groups, and $\cat{M}$ the
category  of finite abelian groups. Then 
$\cat{M} \subset \cat{N}$ and $\cat{N} \subset \cat{Ab} = \cat{Mod}\, \Z$ are
full abelian subcategories. $\cat{M}$ has no projectives nor injectives except
$0$. The only projective in $\cat{N}$ is $0$. But $\cat{N}$ has enough
injectives: this is because $\mbb{Q} / \Z \in \cat{N}$, $\cat{N}$ is closed
under infinite direct sums in $\cat{Ab}$, and the next proposition.
\end{exa}

\begin{prop}
If $A$ is a left noetherian ring, then any direct sum of injective $A$-modules
is an injective module. 
\end{prop}

\begin{proof}
Exercise. (Hint: use the Baer criterion.)
\end{proof}

\begin{prop}
Let $(X, \mcal{A})$ be a ringed space. Then $\cat{Mod}\, \mcal{A}$ has enough
injectives. 
\end{prop}

\begin{proof}
Let $\mcal{M}$ be a left $\mcal{A}$-module. Take a point $x \in X$. The stalk 
$\mcal{M}_x$ is a module over the ring $\mcal{A}_x$, and we can
find an embedding $\phi_x : \mcal{M}_x \to I_x$ into an injective
$\mcal{A}_x$-module. 
Let $g_x : \{ x \} \to X$ be the inclusion, which we may view as a map of
ringed spaces from $( \{ x \}, \mcal{A}_x)$ to $(X, \mcal{A})$.
Define 
$\mcal{I}_x := {g_x}_* I_x$, which is an $\mcal{A}$-module (in fact it is a
constant sheaf of the closed set $\ol{  \{ x \} } \subset X$). 
The adjunction formula gives rise to a sheaf homomorphism 
$\psi_x : \mcal{M} \to  \mcal{I}_x$. 
Since the functor $g_x^* : \cat{Mod}\, \mcal{A} \to \cat{Mod}\, \mcal{A}_x$
is exact, the adjunction formula shows that $\mcal{I}_x$ is an injective
object. 

Finally let $\mcal{I} := \prod_{x \in X} \mcal{I}_x$. This is an injective
$\mcal{A}$-module. There is a homomorphism 
$\psi :=  \prod_{x \in X} \psi_x : \mcal{M} \to \mcal{I}$, and this is a
monomorphism,  since it is a monomorphism at each stalk.
\end{proof}

\cleardoublepage
\section{Outline: the Derived Category}

I will now explain where we are going. Some of the definitions, statements and
proofs will be full (the easy ones...), and the rest (the hard ones...) will
be given later.  

\subsection{The category of complexes}
Let $\cat{M}$ be an additive category. 

\begin{dfn} \label{dfn:2}
A {\em complex} of objects of $\cat{M}$ (or a complex in $\cat{M}$) is a
diagram 
\[ M = ( \cdots \to M^{-1} \xar{\d_M^{-1}}  M^0 \xar{\d_M^0}  
M^1 \xar{\d_M^1} M^2 \to \cdots ) \]
of objects and morphisms in $\cat{M}$, s.t.\ $\d_M^{i+1} \circ \d_M^i = 0$. 

Let $N$ be another such complex. A morphism of complexes 
$\phi : M \to N$ is a collection 
$\phi = \{ \phi^i \}_{i \in \Z}$ of morphisms $\phi^i : M^i \to N^i$ in
$\cat{M}$, s.t.\ 
\[ \d^i_N \circ \phi^i = \phi^{i+1} \circ \d^i_M . \]
\end{dfn}

The collection $\d_M := \{ \d^i_M \}$ is called the {\em differential} of $M$,
or the {\em coboundary operator}. We sometimes we write $\d$ instead of $\d_M$
or $\d_M^i$.  

Note that a morphism $\phi : M \to N$ can be viewed as a commutative diagram 
\[ \UseTips  \xymatrix @C=6ex @R=6ex {
\cdots 
\ar[r]
&
M^i
\ar[r]^{\d_M^i}
\ar[d]_{\phi^i}
&
M^{i+1}
\ar[r]
\ar[d]_{\phi^{i+1}}
&
\cdots
\\
\cdots 
\ar[r]
&
N^i
\ar[r]^{\d_N^i}
&
N^{i+1}
\ar[r]
&
\cdots
} \]

Let us denote by $\dcat{C}(\cat{M})$ the category of complexes in $\cat{M}$. 
This is an additive category: direct sums are degree-wise, i.e.\ 
$(M \oplus N)^i = M^i \oplus N^i$.
If $\cat{M}$ is abelian, then so is $\dcat{C}(\cat{M})$, again with kernels and
cokernels made degree-wise, e.g.\ the kernel of 
$\phi : M \to N$ is the complex $K$ with 
$K^i = \opn{Ker}(\phi^i) \in \cat{M}$.

If  $\cat{N}$ is a full additive subcategory of $\cat{M}$, then  
$\dcat{C}(\cat{N})$ is a full additive subcategory of $\dcat{C}(\cat{M})$.

Any single object $M \in \cat{M}$ can be viewed as a complex 
\[ M' :=  ( \cdots \to 0 \to M \to 0 \to \cdots ) , \] 
where $M$ is in degree $0$; the differential of this complex is of course zero.
The assignment 
$M \mapsto M'$ is a fully faithful additive functor 
$\cat{M} \to \dcat{C}(\cat{M})$.

\begin{exa}
Take $\cat{M} := \cat{Mod}\, A$ for some ring $A$, and 
$\cat{P} \subset \cat{M}$ the full subcategory of the projective $A$-modules. 
Then $\dcat{C}(\cat{M})$ is abelian, and $\dcat{C}(\cat{P})$ is a full additive
subcategory of it.
\end{exa}

\begin{dfn} \label{dfn:5}
Let $\cat{M}$ be an abelian category. For a complex $M \in \dcat{C}(\cat{M})$ we
denote 
\[ \opn{Z}^i(M) := \opn{Ker}(\d : M^i \to M^{i+1}) \subset M^i  \]
and
\[ \opn{B}^i(M) := \opn{Im}(\d : M^{i-1} \to M^{i})  \subset M^i ; \]
these are the objects of {\em $i$-cocycles} and {\em $i$-coboundaries},
respectively, of $M$. Since $\d \circ \d = 0$ we have 
$\opn{B}^i(M) \subset \opn{Z}^i(M)$, and we let
\[ \opn{H}^i(M) := \opn{Z}^i(M) / \opn{B}^i(M) .  \]
This is the {\em $i$-th cohomology} of $M$. 
\end{dfn}

\begin{dfn}  \label{dfn:3}
Let $M, N \in \dcat{C}(\cat{M})$. We define a complex
$\opn{Hom}_{\cat{M}}(M, N) \in \dcat{C}(\cat{Ab})$ as follows. 
In degree $i$ we take 
\[ \opn{Hom}_{\cat{M}}(M, N)^i := \prod_{j \in \Z} 
\opn{Hom}_{\cat{M}}(M^j, N^{j + i}) \in \cat{Ab} . \]
The differential 
\[ \d : \opn{Hom}_{\cat{M}}(M, N)^i \to \opn{Hom}_{\cat{M}}(M, N)^{i+1} \]
is 
\[ \d(\phi) := \d_N \circ \phi - (-1)^i \phi \circ \d_M . \]
\end{dfn}

It is easy to check that $\d \circ \d = 0$. 

In a diagram, an element $\phi \in \opn{Hom}_{\cat{M}}(M, N)^i$
is a collection $\phi = \{ \phi^j \}$ that looks like this:
\[ \UseTips  \xymatrix @C=6ex @R=6ex {
\cdots 
\ar@{-->}[r]
&
M^j
\ar@{-->}[r]
\ar[drr]^{\phi^j}
&
M^{j+1}
\ar@{-->}[r]
\ar[drr]^{\phi^{j+1}}
\ar@{-->}[r]
&
\cdots
\\
& & 
\cdots 
\ar@{-->}[r]
&
N^{j+i}
\ar@{-->}[r]
&
N^{j+1+i}
\ar@{-->}[r]
&
\cdots
} \]
Since $\phi$ does not have to commute with the differentials, they are drawn as
dashed arrows. 

\begin{rem}
A possible ambiguity could arise in the meaning of $\opn{Hom}_{\cat{M}}(M, N)$ 
if $M, N \in \cat{M}$: does it mean the set of morphisms in the category
$\cat{M}$~? Or, if we view $M$ and $N$ as complexes by the canonical
embedding $\cat{M} \to \dcat{C}(\cat{M})$, does $\opn{Hom}_{\cat{M}}(M, N)$
mean the complex from the definition above? 
However in this case the complex $\opn{Hom}_{\cat{M}}(M, N)$
is concentrated in degree $0$, so the ambiguity is eliminated: this is an
instance of the canonical embedding $\cat{Ab} \to \dcat{C}(\cat{Ab})$.
\end{rem}

\begin{prop}
Let $M, N \in \dcat{C}(\cat{M})$. Then there is equality 
\[ \opn{Hom}_{\dcat{C}(\cat{M})}(M, N) = 
\opn{Z}^0 ( \opn{Hom}_{\cat{M}}(M, N) ) . \]
In other words, a morphism of complexes $\phi : M \to N$ is the same as a
$0$-cocycle in the complex $\opn{Hom}_{\cat{M}}(M, N)$. 
\end{prop}

\begin{proof}
Compare Definitions \ref{dfn:2} and \ref{dfn:3}. 
\end{proof}

\subsection{The homotopy category}
Again $\cat{M}$ is an additive category. 

\begin{dfn}
A morphism $\phi : M \to N$ in $\dcat{C}(\cat{M})$ is called {\em
null-homotopic} if it is a $0$-coboundary in 
$\opn{Hom}_{\cat{M}}(M, N)$. Namely if
$\phi = \d(\chi)$ for some 
$\chi \in \opn{Hom}_{\cat{M}}(M, N)^{-1}$. 

Two morphisms $\phi, \phi' : M \to N$ in $\dcat{C}(\cat{M})$ are said to be {\em
homotopic} if $\phi - \phi'$ is null-homotopic. In this case we write 
$\phi \sim \phi'$. 

A morphism $\phi : M \to N$ in $\dcat{C}(\cat{M})$ is called a {\em
homotopy equivalence} if there is some $\psi : N \to M$ s.t.\ 
$\psi \circ \phi \sim 1_M$ and $\phi \circ \psi \sim 1_N$.
\end{dfn}

We already noted that a linear category $\cat{A}$ behaves like a noncommutative
ring (cf.\ Subsection \ref{subsec:1}). 
We shall carry this analogy further, including in the next result. 

Suppose $\cat{A}$ is a linear category.
By a {\em two-sided ideal} $\cat{J}$ in $\cat{A}$ we mean the data of
an abelian subgroup 
$\cat{J}(M, N) \subset \opn{Hom}_{\cat{A}}(M, N)$ for any pair of objects
$M, N \in \cat{A}$, such that for 
any $\phi \in \cat{J}(M, N)$,
$\psi \in \opn{Hom}_{\cat{A}}(M', M)$ and
$\chi \in \opn{Hom}_{\cat{A}}(N, N')$ we have
$\phi \circ \psi \in \cat{J}(M', N)$
and
$\chi \circ \phi \in \cat{J}(M, N')$. 

\begin{prop}
Suppose $\cat{A}$ is a linear category, and $\cat{J}$ is a two-sided ideal in 
$\cat{A}$. Then there is a unique linear category $\cat{B}$, with 
additive functor $F : \cat{A} \to \cat{B}$, s.t.\ 
$\opn{Ob}(\cat{B}) = \opn{Ob}(\cat{A})$, $F$ is the identity on objects, 
 $F$ is surjective on morphism, and 
\[ \opn{Ker} \big( F : \opn{Hom}_{\cat{A}}(M, N) \to \opn{Hom}_{\cat{B}}(M, N)
\big) = \cat{J}(M, N) . \]

Furthermore, if $\cat{A}$ is additive then so is $\cat{B}$.
\end{prop}

\begin{proof}
This is the same as in ring theory. As for the ``furthermore'': direct sums in
$\cat{B}$ come from direct sums in $\cat{A}$ -- see Proposition \ref{prop:2}.
\end{proof}

The category $\cat{B}$ will be called the quotient of $\cat{A}$ by $\cat{J}$.

\begin{prop}
The null-homotopic morphisms in $\dcat{C}(\cat{M})$ form a $2$-sided ideal. 
Namely if $\phi : M \to N$ is null-homotopic, and 
$\psi : M' \to M$, $\chi : N \to N'$ are arbitrary morphisms, then 
$\phi \circ \psi$ and $\chi \circ \phi$ are also null-homotopic.
\end{prop}

\begin{proof}
Say $\phi = \d(\tau)$ with $\tau \in \opn{Hom}_{\cat{M}}(M, N)^{-1}$. 
It is easy to see that $\d$ satisfies the Leibniz rule, so 
\[ \d(\chi \circ \tau) = \d(\chi) \circ \tau + \chi \circ \d(\tau) = 
\chi \circ \phi . \]
We are using $\d(\chi) = 0$ of course. 
This shows that $\chi \circ \phi$ is null-homotopic. Likewise for 
$\phi \circ \psi$.
\end{proof}

\begin{dfn}
The {\em homotopy category} of complexes in $\cat{M}$ is the additive category 
$\dcat{K}(\cat{M})$ gotten as the quotient of $\dcat{C}(\cat{M})$ by the ideal
of
null-homotopic morphisms. 
\end{dfn}

Observe that 
\[ \opn{Hom}_{\dcat{K}(\cat{M})}(M, N) = 
\mrm{H}^0 \big( \opn{Hom}_{\cat{M}}(M, N) \big) . \]
If we denote by $F$ the canonical functor 
$\dcat{C}(\cat{M}) \to \dcat{K}(\cat{M})$,
then $F(\phi) = 0$ iff $\phi$ is null-homotopic; 
and $F(\phi)$ is an isomorphism iff $\phi$ is a homotopy equivalence.

\subsection{The homotopy category is triangulated}
Even if $\dcat{C}(\cat{M})$ is an abelian category, the homotopy category 
$\dcat{K}(\cat{M})$ is usually not. It has another structure: a {\em
triangulated
category}. A full definition will be given later. Here is only a sketch. 

\begin{dfn} \label{dfn:20}  \mbox{}
\begin{enumerate}
\item A {\em T-additive category}
is an additive category $\cat{K}$,  equipped
with an additive automorphism $T$ called the {\em translation}.  

\item Suppose $\cat{K}$ and $\cat{L}$ are T-additive categories.
A {\em T-additive functor} is
an additive functor $F : \cat{K} \to \cat{L}$, together with a natural
isomorphism 
\[ \xi :  F \circ T_{\cat{K}}  \iso T_{\cat{L}} \circ F  . \]

\item Let
\[ (F, \xi), (G, \nu) : \cat{K} \to \cat{L} \] 
be T-additive functors between T-additive
categories. A {\em morphism of T-additive  functors}
\[ \eta : (F, \xi) \to (G, \nu) \] 
is a natural transformation $\eta : F \to G$ s.t.\ this diagram is
commutative:
\[ \UseTips  \xymatrix @C=6ex @R=6ex {
F \circ T_{\cat{K}}
\ar[r]^{\xi}
\ar[d]_{\eta \circ 1}
&
T_{\cat{L}} \circ F
\ar[d]^{1 \circ \eta}
\\
G \circ T_{\cat{K}}
\ar[r]^{\nu}
&
T_{\cat{L}} \circ G  \ . 
} \]
\end{enumerate}
\end{dfn}

The translation $T$ is sometimes called ``shift'' or ``suspension''. 
In \cite{Sc} a T-additive category is called an ``additive category with
translation''. This concept is not so important, as it will be subsumed in
``triangulated category''.

A {\em triangle} in a T-additive category  $\cat{K}$ is a diagram 
\begin{equation} \label{eqn:5}
L \xar{\al} M \xar{\be} N \xar{\ga} T(L) . 
\end{equation}
The objects $L, M , N$ are called the vertices of the triangle. 

A {\em triangulated category} is a T-additive category 
$\cat{K}$, together with a set of triangles called {\em distinguished
triangles}, that satisfy a list of axioms. All these details will come later. 

Given triangulated categories $\cat{K}$ and $\cat{L}$, a {\em triangulated
functor} $F : \cat{K} \to \cat{L}$ is a T-additive functor  that sends
distinguished triangles to distinguished triangles. Namely 
if (\ref{eqn:5}) is a distinguished triangle in $\cat{K}$, then 
\[ F(L) \xar{F(\al)} F(M)  \xar{F(\be)} F(N)  \xar{\xi \circ F(\ga)} 
T(F(L)) \]
is a distinguished triangle in $\cat{L}$.
Morphisms between triangulated functors  are those of T-additive functors
(Definition \ref{dfn:20}(3)). 

Getting back to complexes: 

\begin{dfn}
Let $\cat{M}$ be an additive category. 
Define an additive automorphism $T$ of the category $\dcat{C}(\cat{M})$ as
follows. 
For a complex $M \in \dcat{C}(\cat{M})$, we define $T(M)$ to be the
complex whose $i$-th degree component is $T(M)^i := M^{i+1}$. The differential
is  $\d_{T(M)} := - \d_M$. For a morphism $\phi : M \to N$ the corresponding
morphism $T(\phi) : T(M) \to T(N)$ is $T(\phi)^i := \phi^{i+1}$. 

We usually write $M[k] := T^k(M)$, for an integer $k$; this is the
$k$-th translation of $M$.
\end{dfn}

The automorphism $T$ of $\dcat{C}(\cat{M})$ induces an automorphism $T$ of the
quotient category $\dcat{K}(\cat{M})$. So $\dcat{K}(\cat{M})$ is a T-additive
category. It turns out that there is a structure of triangulated category on
$\dcat{K}(\cat{M})$. We will not specify now what are the distinguished
triangles in $\dcat{K}(\cat{M})$ -- this will be done later.

\subsection{Quasi-isomorphisms and localization}
Here $\cat{M}$ is an abelian category. 
For every $i$ we have an additive functor 
$\mrm{H}^i : \dcat{C}(\cat{M}) \to \cat{M}$;
see Definition \ref{dfn:5}.

\begin{prop}
Suppose $\phi, \phi' : M \to N$ are morphisms in $\dcat{C}(\cat{M})$ that are
homotopic, i.e.\ $\phi \sim \phi'$. Then 
$\mrm{H}^i(\phi) = \mrm{H}^i(\phi')$. 
\end{prop}

\begin{proof}
Exercise.
\end{proof}

It follows that there are well-defined functors 
\[ \mrm{H}^i : \dcat{K}(\cat{M}) \to \cat{M} . \]

\begin{dfn}
A morphism $\phi : M \to N$ in $\dcat{K}(\cat{M})$ is called a {\em
quasi-isomorphism} if the morphisms 
$\mrm{H}^i(\phi) :  \mrm{H}^i(M) \to  \mrm{H}^i(N)$
are isomorphisms for all $i$. 
\end{dfn}

Let us denote by $\cat{S}(M, N)$ the set of quasi-isomorphisms  from $M$ to $N$.

Here is another definition borrowed from ring theory. 

Clearly the quasi-isomorphisms $\cat{S}$ in $\dcat{K}(\cat{M})$ are a
multiplicatively closed set. We will prove that this is in fact a 
{\em left and right denominator set}. 
Just as in ring theory, there is an Ore localization: a linear category 
\[ \dcat{D}(\cat{M}) := \dcat{K}(\cat{M})_{\cat{S}} \]
whose objects are the same as those of $\dcat{K}(\cat{M})$, with an additive
functor 
\[ Q : \dcat{K}(\cat{M}) \to \dcat{D}(\cat{M})  \]
that is the identity on objects. 
For every quasi-isomorphism $\psi$ in $\dcat{K}(\cat{M})$, the morphism 
$Q(\psi)$ is invertible. Every morphism in $\dcat{D}(\cat{M})$ is of the
form $F(\phi) \circ F(\psi)^{-1}$ 
with $\phi \in \dcat{K}(\cat{M})$ and $\psi \in \cat{S}$. 
And $F(\phi) = 0$ iff $\phi \circ \psi = 0$ for some $\psi \in \cat{S}$. 

We will prove that $\dcat{D}(\cat{M})$ is a triangulated category, and that $Q$
is a triangulated functor. Given a triangulated category $\cat{L}$, and a
triangulated functor $F : \dcat{K}(\cat{M}) \to \cat{L}$
such that $F(\psi)$ is invertible for every $\psi \in \cat{S}$, 
there is a unique triangulated functor 
$F_{\cat{S}} : \dcat{D}(\cat{M}) \to \cat{L}$
such that 
$F = F_{\cat{S}} \circ Q$.

\begin{exa}
Consider a short exact sequence 
\[ 0 \to L \xar{\al}  M \xar{\be}  N \to 0 \]
in $\cat{M}$. It turns out that there is an induced morphism 
$\ga : N \to L[1]$ in $\dcat{D}(\cat{M})$, and the triangle 
\[  L \xar{\al}  M \xar{\be}  N \xar{\ga} L[1] \]
is distinguished. 

We will show that the functor $\cat{M} \to  \dcat{D}(\cat{M})$,
sending an object $M$ to the complex concentrated in degree $0$, 
 is fully faithful. And that the exact sequences in $\cat{M}$ can be recovered
as the distinguished triangles in $\dcat{D}(\cat{M})$ whose vertices are in
$\cat{M}$.
\end{exa}

\begin{rem}
Why ``triangle''? This is because sometimes a triangle 
\[ M \xar{\al} N \xar{\be} L \xar{\ga} M[1] \]
is written as a diagram
\[ \UseTips  \xymatrix @C=4ex @R=4ex {
& 
L 
\ar[dl]_{\ga}
\\
M
\ar[rr]^{\al}
& &
N
\ar[lu]_{\be}
} \]
But here $\ga$ is a map of degree $1$. 
\end{rem}

\cleardoublepage
\section{Outline: Derived Functors}

\subsection{From additive to triangulated functors}
Let $\cat{M}$ and $\cat{N}$ be abelian categories, and let 
$F : \cat{M} \to \cat{N}$ be an additive functor. 
The functor $F$ extends to an additive functor 
$\dcat{C}(F) : \dcat{C}(\cat{M}) \to  \dcat{C}(\cat{N})$ 
by sending a complex 
\[ M = ( \cdots \to M^{-1} \xar{\d}  M^0 \xar{\d}  M^1 \to \cdots ) \]
in $\cat{M}$ to the complex 
\[ \dcat{C}(F) (M) := ( \cdots \to F(M^{-1}) \xar{F(\d)}  
F(M^{0}) \xar{F(\d)}  F(M^{1})  \to \cdots ) \]
in $\cat{N}$; and likewise for morphisms. 
$F$ also induces a homomorphism of complexes of abelian groups 
\[ F : \opn{Hom}_{\cat{M}}(M, M') \to \opn{Hom}_{\cat{N}}(F(M), F(M')) , \]
so null-homotopic morphisms in $\dcat{C}(\cat{M})$ go to null-homotopic
morphisms in $\dcat{C}(\cat{N})$. Thus there is an induced additive functor 
\begin{equation} \label{eqn:6}
\dcat{K}(F) : \dcat{K}(\cat{M}) \to  \dcat{K}(\cat{N}) . 
\end{equation}
By construction the functor $\dcat{K}(F)$ respects the translations: 
\[ \dcat{K}(F)(T(M)) = T(\dcat{K}(F)(M)) ; \]
so it is a $T$-additive functor. We will see (this will be an easy consequence
of the definition) that $\dcat{K}(F)$ is in fact a {\em triangulated functor}. 

We want to derive the functor $\dcat{K}(F)$.

\subsection{Derived functors}
Changing notation slightly, suppose that $\cat{M}$ is an abelian category,
$\cat{E}$ is a triangulated category (e.g.\ $\cat{E} = \dcat{D}(\cat{N})$ for
some abelian category $\cat{N}$), and  
$F : \dcat{K}(\cat{M}) \to \cat{E}$ is a triangulated functor (e.g.\ 
$F$ is the functor $Q \circ \dcat{K}(F)$ as in (\ref{eqn:6})).

\begin{dfn} \label{dfn:8}
Let $\cat{M}$ be an abelian category,
$\cat{E}$ a triangulated category, and  
$F : \dcat{K}(\cat{M}) \to \cat{E}$ a triangulated functor. 
A {\em right derived functor} of $F$ 
is a triangulated functor 
\[ \mrm{R} F : \dcat{D}(\cat{M}) \to \cat{E} , \]
together with a morphism 
\[ \eta : F \to \mrm{R} F \circ Q \]
of triangulated functors $\dcat{K}(\cat{M}) \to \cat{E}$. The pair 
$(\mrm{R} F, \eta)$ must have this universal property:
\begin{itemize}
\item[($*$)] Given any triangulated functor $G : \dcat{D}(\cat{M}) \to \cat{E}$,
and a morphism of triangulated functors $\eta' : F \to G \circ Q$, there is a
unique morphism of triangulated functors $\th : \mrm{R} F \to G$
s.t.\ $\eta' = \th \circ \eta$. 
\end{itemize}
\end{dfn}

Pictorially: there is a diagram 
\[ \UseTips  \xymatrix @C=10ex @R=8ex  {
\dcat{K}(\cat{M})
\ar[r]^F _{}="q" 
\ar[d]_{Q}
& 
\cat{E} 
\\
\dcat{D}(\cat{M})
\ar[ur]_{\mrm{R} F} ^(0.5){}="f"
\ar@{=>}  "q";"f" _{\eta}
} \]
where the plain arrows are  triangulated functors, and the double arrow is 
a morphism of triangulated functors\footnote{A correction here relative to 
previous version.}. For any other $(G, \eta')$ there is a unique $\th$ 
\[ \UseTips  \xymatrix @C=14ex @R=12ex {
\dcat{K}(\cat{M})
\ar[r]^F _{}="q"  ^(0.6){}="q2" 
\ar[d]_{Q}
& 
\cat{E} 
\\
\dcat{D}(\cat{M})
\ar[ur] ^(0.5){}="f" ^(0.5){}="f1" 
\ar@{=>}  "q";"f" _{\eta}
\ar@(r,d)[ru]_{G} ^(0.55){}="g"  ^(0.48){}="g1"
\ar@{=>}  "q2";"g" ^(0.7){\eta'}
\ar@{=>}  "f";"g1" _(0.5){\th}
} \]
s.t.\ $\eta' = \th \circ \eta$. 

We record this result, that is an immediate consequence of the definition:

\begin{prop} \label{prop:3}
If a right derived functor $(\mrm{R} F, \eta)$ exists, then it is unique, up to
a unique isomorphism of triangulated functors.
\end{prop}

In complete symmetry we have:

\begin{dfn}
Let $\cat{M}$ be an abelian category,
$\cat{E}$ a triangulated category, and  
$F : \dcat{K}(\cat{M}) \to \cat{E}$ a triangulated functor. 
A {\em left derived functor} of $F$ 
is a triangulated functor 
\[ \mrm{L} F : \dcat{D}(\cat{M}) \to \cat{E} , \]
together with a morphism 
\[ \eta : \mrm{L} F \circ Q \to F  \]
of triangulated functors $\dcat{K}(\cat{M}) \to \cat{E}$. The pair 
$(\mrm{L} F, \eta)$ must have this universal property:
\begin{itemize}
\item[($*$)] Given any triangulated functor $G : \dcat{D}(\cat{M}) \to \cat{E}$,
and a morphism of triangulated functors $\eta' :  G \circ Q \to F$, there is
a unique morphism of triangulated functors $\th : G \to \mrm{L} F$
s.t.\ $\eta' = \eta \circ \th$. 
\end{itemize}
\end{dfn}

Again:

\begin{prop}
If a left derived functor $(\mrm{L} F, \eta)$ exists, then it is unique, up to a
unique isomorphism of triangulated functors.
\end{prop}


\cleardoublepage
\section{Outline: Existence of Derived Functors}

\subsection{Full triangulated subcategories}

\begin{dfn} \label{dfn:22}
Let $\cat{K}$ be a triangulated category. A {\em full triangulated subcategory}
of $\cat{K}$ is a full subcategory $\cat{J} \subset \cat{K}$, s.t.\ these
conditions hold:
\begin{enumerate}
\rmitem{i} $\cat{J}$ is closed under shifts, i.e.\ $I \in \cat{J}$ iff
$I[1] \in \cat{J}$.

\rmitem{ii} $\cat{J}$ is closed under distinguished triangles, i.e.\ if 
\[ I' \to I \to I'' \to I[1] \] 
is a  distinguished triangle in $\cat{K}$ s.t.\ $I', I \in \cat{J}$, then 
also $I'' \in \cat{J}$.
\end{enumerate}
\end{dfn}

When $\cat{J}$ is a full triangulated subcategory of $\cat{K}$, then $\cat{J}$
itself is triangulated, and the inclusion $\cat{J} \to \cat{K}$ is a
triangulated functor.

\subsection{Right derived functor}
Recall that an additive functor $F : \cat{M} \to \cat{N}$ between abelian
categories induces a triangulated functor 
\[ \dcat{K}(F) : \dcat{K}(\cat{M}) \to \dcat{K}(\cat{N}) . \]
In the next theorem we consider a slightly more general situation. 

\begin{thm} \label{thm:101}
Let $\cat{M}$ be an abelian category,
$\cat{E}$ a triangulated category, and 
$F : \dcat{K}(\cat{M}) \to \cat{E}$ a triangulated functor. 
Assume there is a full triangulated subcategory 
$\cat{J} \subset \dcat{K}(\cat{M})$
with these two properties:

\begin{enumerate}
\rmitem{i} If $\phi : I \to I'$ is a quasi-isomorphism in $\cat{J}$, then 
$F(\phi) : F(I) \to F(I')$ is an isomorphism in $\cat{E}$.
 
\rmitem{ii} Every $M \in \dcat{K}(\cat{M})$ admits a quasi-isomorphism 
$M \to I$ for some $I \in \cat{J}$.
\end{enumerate}

Then the right derived functor 
$\mrm{R} F : \dcat{D}(\cat{M}) \to \cat{E}$
exists. Moreover, for any $I \in \cat{J}$ the morphism  
\[ \eta_I : F(I) \to (\mrm{R} F \circ Q)(I) \]
in $\cat{E}$ is an isomorphism.
\end{thm}

\begin{proof}[Sketch of Proof]
(A complete proof will be given later.)
Recall that $\cat{S} \subset \dcat{K}(\cat{M})$ is the category
(multiplicatively
closed set of morphisms) consisting of the quasi-isomorphisms in
$\dcat{K}(\cat{M})$, and 
\[ Q : \dcat{K}(\cat{M}) \to \dcat{K}(\cat{M})_{\cat{S}} = 
\dcat{D}(\cat{M}) \]
is the localization functor. 

Condition (ii) implies that $\cat{S} \cap \cat{J}$, the quasi-isomorphisms in
$\cat{J}$, is a left denominator set in $\cat{J}$, and the inclusion 
\begin{equation} \label{eqn:101}
\cat{J}_{\cat{S} \cap \cat{J}} \to \dcat{K}(\cat{M})_{\cat{S}} = 
\dcat{D}(\cat{M}) 
\end{equation}
is an equivalence (of triangulated categories). 

For every $M \in \dcat{K}(\cat{M})$ we choose a quasi-isomorphism 
$\zeta_M : M \to I(M)$ with $I(M) \in \cat{J}$. 
We take care so that $I(M)$ and $\zeta_M$ commute with shifts, and that 
$I(M) = M$ and $\zeta_M = 1_M$ when $M \in \cat{J}$. In this way we
obtain a triangulated functor 
\[ I : \dcat{D}(\cat{M}) \to \cat{J}_{\cat{S} \cap \cat{J}} \]
which splits the inclusion (\ref{eqn:101}), with a natural isomorphism 
$\zeta : \bsym{1}_{\dcat{D}(\cat{M})}  \to I$.

We now invoke condition (i). By the universal property of localization there
is a unique functor 
\[ F_{\cat{S} \cap \cat{J}} : \cat{J}_{\cat{S} \cap \cat{J}} \to \cat{E} \]
extending 
$F|_{\cat{J}} : \cat{J} \to \cat{E}$.
We define 
\[ \mrm{R} F := F_{\cat{S} \cap \cat{J}} \circ I 
: \dcat{D}(\cat{M}) \to \cat{E} . \]
For any $M \in \dcat{K}(\cat{M})$ we define 
\[ \eta_M : F(M) \to (\mrm{R} F \circ Q)(M) = F ( I (M)) \]
to be $\eta_M := F(\zeta_M)$. 
This is a natural transformation 
\[ \eta : F \to \mrm{R} F \circ Q . \]
In a diagram (commutative via $\eta$):
\[ \UseTips  \xymatrix @C=10ex @R=7ex {
\cat{J}
\ar[r]^{Q_{\cat{J}}}
\ar[d]_{\mrm{inc}}
&
\cat{J}_{\cat{S} \cap \cat{J}}
\ar[dr]^{F_{\cat{S} \cap \cat{J}}}
\\
\dcat{K}(\cat{M})
\ar[r]^{Q} _{}="q"
\ar@(d,d)[rr]_F  ^(0.34){}="f" 
&
\dcat{D}(\cat{M})
\ar[u]_{I}
\ar[r]^{\mrm{R} F}
&
\cat{E}
\ar@{=>}  "f";"q" _{\eta}
} \]

It remains to check that the pair $(\mrm{R}F, \eta)$ has the universal property.
\end{proof}

\subsection{K-injectives}

\begin{dfn}
Let $\cat{M}$ be an abelian category.
A complex $N \in \dcat{K}(\cat{M})$ is called {\em acyclic} if 
$\mrm{H}^i(N) = 0$ for all $i$. In other words, if the morphism 
$0 \to N$ is a quasi-isomorphism.
\end{dfn}

\begin{dfn} \label{dfn:100}
Let $\cat{M}$ be an abelian category.

\begin{enumerate}
\item A complex $I \in \dcat{K}(\cat{M})$ is called {\em K-injective} if for
every
acyclic $N \in \dcat{K}(\cat{M})$, the complex 
$\opn{Hom}_{\cat{M}}(N, I)$ is also acyclic.

\item Let $M \in \dcat{K}(\cat{M})$. A {\em K-injective resolution} of $M$ is a
quasi-isomorphism $M \to I$ in $\dcat{K}(\cat{M})$, where $I$ is K-injective.

\item We say that $\dcat{K}(\cat{M})$ {\em has enough K-injectives} if every 
$M \in \dcat{K}(\cat{M})$ has some K-injective resolution.
\end{enumerate}
\end{dfn}

The concept of K-injective complex was introduced by Spaltenstein \cite{Sp} in
1988. At about the same time other authors (Keller \cite{Ke}, Bockstedt-Neeman
\cite{BN}, \ldots) discovered this concept independently, with other names (such
as {\em homotopically injective complex}).

\begin{exa}
Let $\cat{M}$ be either $\cat{Mod}\, A$, for some ring $A$,
or $\cat{Mod}\, \mcal{A}$, for some ringed space $(X, \mcal{A})$.
Then $\dcat{K}(\cat{M})$ has enough K-injectives. 
We will prove this later. (?)
\end{exa}

\begin{exa} \label{exa:101}
Let $\cat{M}$ be an abelian category. Any bounded below
complex of injectives is K-injective.

Now assume that $\cat{M}$ has enough injectives. Let 
$\cat{K}^+(\cat{M})$ be the category of bounded below complexes. 
Any $M \in \cat{K}^+(\cat{M})$ admits a quasi-isomorphism 
$M \to I$, with $I$ a bounded below complex of injectives.
This generalizes the ``old-fashioned'' injective resolution 
\[ 0 \to M \to I^0 \to I^1 \to \cdots \]
for $M \in \cat{M}$. 
Thus $\cat{K}^+(\cat{M})$  has enough K-injectives. 

These facts were already known in [RD], in 1966, but of course without the name
``K-injective''.
\end{exa}

There are two wonderful things when $\dcat{K}(\cat{M})$ has enough K-injectives.

\begin{prop} \label{prop:101}
Any quasi-isomorphism $\phi : I \to I'$ between K-injective complexes is a
homotopy equivalence; i.e.\ it is an isomorphism in $\dcat{K}(\cat{M})$.
\end{prop}

This will be proved later. 

Let us denote by $\dcat{K}(\cat{M})_{\mrm{inj}}$ the full subcategory of 
$\dcat{K}(\cat{M})$ on the K-injective complexes. 

\begin{cor} \label{cor:101}
If $\dcat{K}(\cat{M})$ has enough K-injectives, then any triangulated functor 
$F : \dcat{K}(\cat{M}) \to \cat{E}$ \tup{(}cf.\ Theorem \tup{\ref{thm:101})} has
a right derived functor. 
\end{cor}

\begin{proof}
Take $\cat{J} := \dcat{K}(\cat{M})_{\mrm{inj}}$ in the theorem.
\end{proof}

\begin{exa} \label{exa:102}
Suppose we are in the situation of Example \ref{exa:101}.
Let $\dcat{D}^+(\cat{M}) := \cat{K}^+(\cat{M})_{\cat{S}^+}$, the localization
of $\cat{K}^+(\cat{M})$ with respect to 
$\cat{S}^+ := \cat{S} \cap \cat{K}^+(\cat{M})$. 
If $F : \cat{M} \to \cat{N}$ is an additive functor to some other abelian
category $\cat{N}$, then (by a variant of Corollary \ref{cor:101}) we get a
right derived functor 
\[ \mrm{R} F : \dcat{D}^+(\cat{M}) \to \dcat{D}^+(\cat{N}) . \]

For $M \in \cat{M}$ let $\mrm{R}^i F(M) := \mrm{H}^i (\mrm{R}^i F(M))$.
We get an additive functor 
\[ \mrm{R}^i F : \cat{M} \to \cat{N} , \]
 which is the usual right derived functor. 
If $F$ is left exact then the natural transformation 
$\eta : F \to \mrm{R}^0 F$ is an isomorphism.
\end{exa}

Here is the second good thing. 

\begin{prop} \label{prop:102}
The functor 
\begin{equation} \label{eqn:102}
Q : \dcat{K}(\cat{M})_{\mrm{inj}} \to \dcat{D}(\cat{M}) 
\end{equation}
is fully faithful. 

Hence, if $\dcat{K}(\cat{M})$ has enough K-injectives, then
\tup{(\ref{eqn:102})}
is an equivalence of triangulated categories.
\end{prop}

This will be proved later. The benefit here is that we can avoid the
localization process (inverting the quasi-isomorphisms).

\subsection{Left derived functor}

This is the dual of Theorem \ref{thm:101}. The proof is the same.

\begin{thm} \label{thm:102}
Let $\cat{M}$ be an abelian category, $\cat{E}$ a triangulated category, and 
$F : \dcat{K}(\cat{M}) \to \cat{E}$ a triangulated functor. 
Assume there is a full triangulated subcategory 
$\cat{P} \subset \dcat{K}(\cat{M})$
with these two properties:

\begin{enumerate}
\rmitem{i} If $\phi : P \to P'$ is a quasi-isomorphism in $\cat{P}$, then 
$F(\phi) : F(P) \to F(P')$ is an isomorphism in $\cat{E}$.
 
\rmitem{ii} Every $M \in \dcat{K}(\cat{M})$ admits a quasi-isomorphism 
$P \to M$ for some $P \in \cat{P}$.
\end{enumerate}

Then the left derived functor 
$\mrm{L} F : \dcat{D}(\cat{M}) \to \cat{E}$
exists. Moreover, for any $P \in \cat{P}$ the morphism  
\[ \eta_P :  (\mrm{L} F \circ Q)(P) \to F(P)  \]
in $\cat{E}$ is an isomorphism.
\end{thm}

\subsection{K-projectives etc}

\begin{dfn} \label{dfn:101}
Let $\cat{M}$ be an abelian category.

\begin{enumerate}
\item A complex $P \in \dcat{K}(\cat{M})$ is called {\em K-projective} if for
every acyclic $N \in \dcat{K}(\cat{M})$, the complex 
$\opn{Hom}_{\cat{M}}(P, N)$ is also acyclic.

\item Let $M \in \dcat{K}(\cat{M})$. A {\em K-projective resolution} of $M$ is a
quasi-isomorphism $P \to M$ in $\dcat{K}(\cat{M})$, where $P$ is K-projective.

\item We say that $\dcat{K}(\cat{M})$ {\em has enough K-projectives} if every 
$M \in \dcat{K}(\cat{M})$ has some K-projective resolution.
\end{enumerate}
\end{dfn}


\begin{exa}
Let $\cat{M} := \cat{Mod}\, A$, for some ring $A$.
Then $\dcat{K}(\cat{M})$ has enough K-projectives. 
We will prove this later. (?)
\end{exa}

\begin{exa}
Let $\cat{M}$ be an abelian category. Any bounded above
complex of projectives is K-projective.

Now assume that $\cat{M}$ has enough projectives. Let 
$\cat{K}^-(\cat{M})$ be the category of bounded above complexes. 
Any $M \in \cat{K}^-(\cat{M})$ admits a quasi-isomorphism 
$P \to M$, with $M$ a bounded above complex of projectives.
This generalizes the ``old-fashioned'' projective resolution 
\[  \cdots \to P^{-1} \to P^0 \to M \to 0 \]
for $M \in \cat{M}$. 
Thus $\cat{K}^-(\cat{M})$  has enough K-projectives. 
\end{exa}

The dual versions of Proposition \ref{prop:101}, Corollary \ref{cor:101},
Proposition \ref{prop:102} and Example \ref{exa:102} hold. 

Later we will also discuss {\em K-flat complexes} over a ring $A$.
These are very useful for constructing certain left derived functors.

\cleardoublepage
\section{Outline: Derived Bifunctors}

\subsection{Biadditive Bifunctors}
The most important functors for us are these: 
\begin{equation} \label{eqn:9}
\opn{Hom}_{\cat{M}} (-,-) : \cat{M}^{\mrm{op}} \times \cat{M} \to  
\cat{Mod}\, \K , 
\end{equation}
where $\cat{M}$ is a $\K$-linear abelian category; and 
\begin{equation} \label{eqn:8}
- \otimes_A - : \cat{Mod}\, A^{\mrm{op}} \times \cat{Mod}\, A \to
\cat{Mod}\, \K
\end{equation}
where $A$ is a $\K$-algebra. 
We view these as functors on the product categories, or as bifunctors (see
Subsection \ref{subsec:bifunc}). 
They have this additional property:

\begin{dfn}
Let $\cat{L}, \cat{M}$ and $\cat{N}$ be linear categories. A {\em biadditive
bifunctor} 
\[ F : \cat{L} \times \cat{M} \to \cat{N}  \]
is a bifunctor such that for every 
$L_0, L_1 \in \cat{L}$ and $M_0, M_1 \in \cat{M}$ the function 
\[ F : \opn{Hom}_{\cat{L}}(L_0, L_1)
\times \opn{Hom}_{\cat{M}}(M_0, M_1) \to 
\opn{Hom}_{\cat{N}} \big( F(L_0, M_0),  F(L_1, M_1) \big) \]
is bilinear.
\end{dfn}

In general, given additive categories $\cat{L}, \cat{M}$ and $\cat{N}$, and a
biadditive bifunctor 
\[ F : \cat{L} \times \cat{M} \to \cat{N} , \]
there are two different ways to get a biadditive bifunctor on complexes,
depending on which {\em totalization} we use. Let $L \in \dcat{C}(\cat{L})$ and 
$M \in \dcat{C}(\cat{M})$. The first option is the direct sum totalization, that
gives a complex 
$\cat{C}_{\oplus}(F)(L, M) \in \dcat{C}(\cat{N})$ with degree $i$ component
\begin{equation}
\cat{C}_{\oplus}(F)(L, M)^i := \bigoplus_{p + q = i} 
F(L^p, M^{q}) . 
\end{equation}
For this to work we must assume that $\cat{N}$ has countable direct sums. 
The differential $\d$ of this complex is this: its restriction $\d^{p, q}$ to
the component  $F(L^p, M^{q})$ is
\begin{equation} \label{eqn:7}
\d^{p, q} := F(\d_L, 1_M) + (-1)^p F(1_L, \d_M) . 
\end{equation}
 
The second option is to use the product totalization
\begin{equation}
\cat{C}_{\Pi}(F)(L, M)^i := \prod_{p + q = i} F(L^p, M^{q}) . 
\end{equation}
 For this to work we must assume that $\cat{N}$ has countable products. 
The result is a complex $\cat{C}_{\Pi}(F)(L, M)$, with differential
(\ref{eqn:7}).

Thus we obtain biadditive bifunctors 
\[ \cat{C}_{\oplus}(F), \cat{C}_{\Pi}(F) : 
\dcat{C}(\cat{L}) \times \dcat{C}(\cat{M}) \to \dcat{C}(\cat{N}) . \]

\begin{exa} \label{exa:2}
For the tensor functor (\ref{eqn:8}) the direct sum totalization is used. 
For the Hom functor (\ref{eqn:9})  the product totalization is used; and the
complex \lb
$\cat{C}_{\Pi}(\opn{Hom}_{\cat{M}})(L, M)$, for $L, M \in \dcat{C}(\cat{M})$,
coincides with the complex $\opn{Hom}_{\cat{M}}(L, M)$ from Definition
\ref{dfn:3}. 
To get the degrees and signs right, one must observe that the isomorphism of
categories $\dcat{C}(\cat{M})^{\mrm{op}} \iso \dcat{C}(\cat{M}^{\mrm{op}})$
inverts the degree. 
\end{exa}

\begin{exer}
Verify that the example above is correct; namely that the signs in Hom
complexes agree.
\end{exer}

Both bifunctors $\cat{C}_{\oplus}(F)$ and $\cat{C}_{\Pi}(F)$  respect the
homotopy relation (this is easy to check), so there are induced biadditive
bifunctors 
\begin{equation}
\cat{K}_{\oplus}(F), \cat{K}_{\Pi}(F) : 
\dcat{K}(\cat{L}) \times \dcat{K}(\cat{M}) \to \dcat{K}(\cat{N}) .
\end{equation}
It is clear from the constructions that these functors commute with the
translations (shifts) in an obvious sense, so they are {\em bi-T-additive
bifunctors}.

\subsection{Bitriangulated bifunctors}
The functors $\cat{K}_{\oplus}(F)$ and $\cat{K}_{\Pi}(F)$ above also send
distinguished triangles (in each argument) to distinguished triangles. We call
such functors {\em bitriangulated bifunctors}. A full definition will be given
later (?). 

Switching notation, suppose we are given a bitriangulated bifunctor 
\[ F : \dcat{K}(\cat{M}) \times \dcat{K}(\cat{N}) \to \cat{E} . \]
Here $\cat{M}$ and $\cat{N}$ are abelian categories, and $\cat{E}$ is
triangulated. A {\em right derived bifunctor} of $F$ is a 
bitriangulated bifunctor
\[ \mrm{R} F : \dcat{D}(\cat{M}) \times \dcat{D}(\cat{N}) \to \cat{E} , \]
with a morphism of bitriangulated bifunctors
\[ \eta : F \to \mrm{R} F \circ (Q \times Q) , \]
that has the same universal property as in Definition \ref{dfn:8}. 
There is the same sort of uniqueness as in Proposition \ref{prop:3}.

The left derived bifunctor $\mrm{L} F$ is defined similarly.

Here is an existence result, similar to Theorem \ref{thm:101}.

\begin{thm} \label{thm:131}
Let $\cat{M}$ and $\cat{N}$ be abelian categories,  $\cat{E}$ a
triangulated category, and 
\[ F : \dcat{K}(\cat{M}) \times \dcat{K}(\cat{N}) \to \cat{E} . \]
a  bitriangulated bifunctor. 
Suppose there is a full triangulated subcategory 
$\cat{J} \subset \dcat{K}(\cat{M})$ with these two properties:
\begin{enumerate}
\rmitem{i}  If $\phi : I \to I'$ is a quasi-isomorphism in $\cat{J}$, and 
$\psi : N \to N'$ is a quasi-isomorphism in $\dcat{K}(\cat{N})$, then 
\[ F(\psi, \phi) : F(N, I) \to F(N', I') \]
is an isomorphism in $\cat{E}$.

\rmitem{ii} Every $M \in \dcat{K}(\cat{M})$ admits a quasi-isomorphism 
$M \to I$ with $I \in \cat{J}$.
\end{enumerate}
Then the derived functor $\mrm{R} F$ exists, and moreover
\[ \eta_{N, I} :  F(N, I) \to \mrm{R} F(N, I) \]
is an isomorphism for any $I \in \cat{J}$.
\end{thm}

Note the symmetry between $\cat{M}$ and $\cat{N}$ in this theorem.

\begin{exa}
$\cat{M}$ is any abelian category.
Consider the biadditive bifunctor 
\[ F : \cat{M}^{\mrm{op}} \times \cat{M} \to \cat{Ab} , \]
\[ F(M, N) := \opn{Hom}_{\cat{M}}(M, N) . \]
Using the product totalization (Example \ref{exa:2})
we get a derived functor $\mrm{R} F$. The isomorphisms of
triangulated categories 
$\dcat{K}(\cat{M}^{\mrm{op}}) \cong \dcat{K}(\cat{M})^{\mrm{op}}$
and
$\dcat{D}(\cat{M}^{\mrm{op}}) \cong \dcat{D}(\cat{M})^{\mrm{op}}$
allows us to write  $\mrm{R} F$ like this:
\[ \opn{RHom}_{\cat{M}} : 
\dcat{D}(\cat{M})^{\mrm{op}} \times \dcat{D}(\cat{M}) \to 
\dcat{D}(\cat{Ab}) . \]
(If $\cat{M}$ happens to be a $\K$-linear category, for some commutative ring
$\K$, then $\opn{RHom}_{\cat{M}}$ takes values in $\dcat{D}(\cat{Mod}\, \K)$.)
 
In case $\dcat{K}(\cat{M})$ has enough K-injectives, then the full subcategory 
$\cat{J} := \dcat{K}(\cat{M})_{\mrm{inj}} \subset \dcat{K}(\cat{M})$ 
on the K-injectives satisfies properties (i-ii) of the theorem. 
If $\dcat{K}(\cat{M})$ has enough K-projectives, then we can take the full
subcategory 
$\cat{P} := \dcat{K}(\cat{M})_{\mrm{proj}} \subset \dcat{K}(\cat{M})$ 
on the K-projectives. 
Then $\cat{P}^{\mrm{op}} \subset \dcat{K}(\cat{M}^{\mrm{op}})$
 satisfies properties (i-ii) of the theorem. 
\end{exa}


\cleardoublepage
\section{Triangulated Categories} \label{sec:triangulated}

\subsection{Triangulated Categories}
We now give the full definition of triangulated category. 

Recall the notion of T-additive category $(\cat{K}, T)$   (Definition
\ref{dfn:20}). 
We often refer to the translation automorphism $T$ as the shift, and write 
$M[k] := T^k(M)$ for $k \in \Z$. 

Let  $(\cat{K}, T)$ be a T-additive category. 
A {\em triangle} in   $\cat{K}$ is a diagram 
\begin{equation} 
L \xar{\al} M \xar{\be} N \xar{\ga} T(L) . 
\end{equation}
The objects $L, M , N$ are called the vertices of the triangle. 

Suppose $L' \xar{\al'} M' \xar{\be'} N' \xar{\ga'} T(L')$
is another triangle in  $\cat{K}$. A {\em morphism of triangles} between them is
a commutative diagram 
\begin{equation} \label{eqn:10}
\UseTips  \xymatrix @C=5ex @R=5ex { 
L 
\ar[r]^{\al}
\ar[d]_{\phi}
&
M
\ar[r]^{\be}
\ar[d]_{\psi}
& 
N
\ar[r]^{\ga}
\ar[d]_{\chi}
&
T(L)
\ar[d]_{T(\phi)}
\\
L' 
\ar[r]^{\al'}
&
M'
\ar[r]^{\be'}
& 
N'
\ar[r]^{\ga'}
&
T(L') \ .
} 
\end{equation}
The morphism of triangles (\ref{eqn:10}) is called an isomorphism if 
$\phi, \psi$ and $\chi$ are all isomorphisms. 

\begin{dfn}
A {\em triangulated category} is a T-additive category $(\cat{K}, T)$, equipped
with a set of triangles called  {\em distinguished triangles}. 
The following axioms have to be satisfied:

\begin{enumerate}
\item[(TR1)] 
\begin{itemize}
\item Any triangle that's isomorphic to a distinguished
triangle is also a distinguished triangle. 

\item For every morphism $\al : L \to M$ there is a
distinguished triangle 
\[ L \xar{\al} M \xar{} N \xar{} T(L) . \] 

\item  For every object $M$ the triangle 
\[ M \xar{1_M} M \to 0 \to T(M) \] 
is distinguished.
\end{itemize}

\item[(TR2)] A triangle 
\[ L \xar{\al} M \xar{\be} N \xar{\ga} T(L) \]
is distinguished iff the triangle 
\[ M \xar{\be} N \xar{\ga} T(L) \xar{- T(\al)} T(M) \] 
is distinguished.

\item[(TR3)]  Suppose 
\[ L \xar{\al} M \xar{\be} N \xar{\ga} T(L) \]
and
\[ L' \xar{\al'} M' \xar{\be'} N' \xar{\ga'} T(L') \]
are distinguished triangles, and 
$\phi : L \to L'$ and $\psi : M \to M'$ are morphisms that satisfy 
$\psi \circ \al = \al' \circ \phi$. 
Then there is a morphism $\chi : N \to N'$ such that 
\[ \UseTips  \xymatrix @C=5ex @R=5ex { 
L 
\ar[r]^{\al}
\ar[d]_{\phi}
&
M
\ar[r]^{\be}
\ar[d]_{\psi}
& 
N
\ar[r]^{\ga}
\ar@{-->}[d]_{\chi}
&
T(L)
\ar[d]_{T(\phi)}
\\
L' 
\ar[r]^{\al'}
&
M'
\ar[r]^{\be'}
& 
N'
\ar[r]^{\ga'}
&
T(L') \ .
} \]
is a morphism of triangles.

\item[(TR4)] Suppose 
\[ L \xar{\al} M \xar{\be} N' \xar{} T(L) , \]
\[ M \xar{\ga} N \xar{\de} L' \xar{} T(M) \]
and
\[ L \xar{\ga \circ \al} N \xar{\ep} M' \xar{} T(L) , \]
are distinguished triangles. Then there is a distinguished triangle 
\[  N' \xar{\phi} M' \xar{\psi} L' \xar{} T(N')  \]
making the diagram 
\[ \UseTips  \xymatrix @C=5ex @R=5ex { 
L 
\ar[r]^{\al}
\ar[d]_{1}
&
M
\ar[r]^{\be}
\ar[d]_{\ga}
& 
N'
\ar[r]^{}
\ar@{-->}[d]_{\phi}
&
T(L)
\ar[d]_{1}
\\
L
\ar[r]^{\ga \circ \al}
\ar[d]_{\al}
&
N
\ar[r]^{\ep}
\ar[d]_{1}
& 
M'
\ar[r]^{}
\ar@{-->}[d]_{\psi}
&
T(L)
\ar[d]_{T(\al)}
\\
M
\ar[r]^{\ga}
\ar[d]_{\be}
&
N
\ar[r]^{\de}
\ar[d]_{\ep}
& 
L'
\ar[r]^{}
\ar[d]_{1}
&
T(M)
\ar[d]_{T(\be)}
\\
N'
\ar@{-->}[r]^{\phi}
&
M'
\ar@{-->}[r]^{\psi}
& 
L'
\ar@{-->}[r]^{}
&
T(N')
} \]

commutative.
\end{enumerate}
\end{dfn}

Here are a few remarks on this definition. 
The object $N$ in axiom (TR1) is referred to as a {\em cone} on 
$\al : L \to M$. We should think of the cone as something combining 
``the cokernel'' and ``the kernel'' of $\al$.

Axiom (TR2) says that if we
``turn'' a distinguished triangle we remain with a distinguished triangle.

Axiom (TR3) says that a commutative square induces a morphism on the cones
of the horizontal morphisms, that fits into a morphism of distinguished
triangles. Note however that the new morphism
$\chi$ is {\em not unique}; in other words,  {\em cones are not
functorial}.  This fact has some deep consequences in many
applications. 

Axiom (TR4) is called the {\em octahedron axiom}. Frankly I never understood
its role... Let's see if, during our course, we can discover why we need this
axiom. 

Note that the numbering of the axioms in \cite{Sc, KS1, KS2} is different.

\subsection{Triangulated Functors}
Suppose $\cat{K}$ and $\cat{L}$ are T-additive categories. 
The notion of  T-additive functor $F : \cat{K} \to \cat{L}$
was defined in Definition \ref{dfn:20}. In that definition we also introduced
the notion of morphism 
$\eta : F \to G$ between such T-additive functors.

\begin{dfn}
Let $\cat{K}$ and $\cat{L}$ be triangulated categories. 
\begin{enumerate}
\item A {\em triangulated functor} from $\cat{K}$ to $\cat{L}$
is a T-additive functor 
\[ (F, \xi) : \cat{K} \to \cat{L} \] 
that sends distinguished triangles to
distinguished triangles. Namely for any distinguished triangle 
\[ L \xar{\al} M \xar{\be} N \xar{\ga} T(L) \]
in $\cat{K}$, the triangle 
\[ F(L) \xar{F(\al)} F(M)  \xar{F(\be)} F(N)  \xar{\xi \circ F(\ga)} 
T(F(L)) \]
is a distinguished triangle in $\cat{L}$.

\item Suppose $(G, \nu) : \cat{K} \to \cat{L}$ is another triangulated functor.
A {\em morphism of triangulated functors} $\eta : (F, \xi) \to (G , \nu)$ is 
a morphism of T-additive functors.
\end{enumerate}
\end{dfn}

We usually keep the isomorphism $\xi$ implicit, and refer to $F$ as a
triangulated functor. 

For a category $\cat{K}$ there is a canonical contravariant functor 
$\opn{op} : \cat{K} \to \cat{K}^{\opn{op}}$,
that is the identity on objects, and reverses the arrows. In fact $\opn{op}$ is
an anti-isomorphism of categories.

\begin{prop}
Let $\cat{K}$ be a triangulated category, and let 
$\opn{op} : \cat{K} \to \cat{K}^{\mrm{op}}$ be the canonical contravariant
isomorphism from $\cat{K}$ to its opposite category. 
Define a translation $T^{\mrm{op}}$ on $\cat{K}^{\mrm{op}}$ by the formula 
$T^{\mrm{op}} := \opn{op} \circ\, T^{-1} \circ \opn{op}^{-1}$. 
The distinguished triangles in $\cat{K}^{\mrm{op}}$ are defined to be the
triangles 
\[ N \xar{\opn{op}(\be)} M \xar{\opn{op}(\al)} L 
\xar{\opn{op}(- T^{-1}(\ga))}  T^{\mrm{op}}(N) , \]
where 
$L \xar{\al} M \xar{\be} N \xar{\ga} T(L)$
is any distinguished triangle in $\cat{K}$.
Then $(\cat{K}^{\mrm{op}}, T^{\mrm{op}})$ is a triangulated category. 
\end{prop}

\begin{proof}
This is an exercise. Please check that I got the formulas right!
(Hint: use the proof of the next proposition.)
\end{proof}

\begin{dfn}
Let $\cat{K}$ be a triangulated category, and let $\cat{M}$ be an abelian
category. A {\em cohomological functor} $F : \cat{K} \to \cat{M}$ is an
additive functor, such that for every distinguished triangle 
$L \xar{\al} M \xar{\be} N \xar{\ga} T(L)$ in $\cat{K}$, the sequence 
\[ F(L) \xar{F(\al)} F(M) \xar{F(\be)} F(N) \] 
is exact.
\end{dfn}

\begin{prop}
Let  $F : \cat{K} \to \cat{M}$ be a cohomological functor, and let 
$L \xar{\al} M \xar{\be} N \xar{\ga} T(L)$
be a distinguished triangle in $\cat{K}$. Then the sequence 
\[ \begin{aligned}
&    \cdots \to F(L[i]) \xar{F(\al[i])} F(M[i]) \xar{F(\be[i])} F(N[i])
\\
& \hspace{12ex}
\xar{F(\ga[i])} F(L[i+1])   \xar{F(\al[i+1])} F(M[i+1]) \to \cdots 
\end{aligned} \]
is exact.
\end{prop}

\begin{proof}
By axiom (TR2) we have distinguished triangles 
\[ L[i] \xar{(-1)^{i} \al[i]} M[i] \xar{(-1)^{i} \be[i]} N[i] 
\xar{(-1)^{i} \ga[i]} L[i+1] , \]
\[ M[i] \xar{(-1)^{i} \be[i]} N[i] \xar{(-1)^{i} \ga[i]} L[i+1] 
\xar{(-1)^{i+1} \al[i+1]} M[i+1] \]
and
\[ N[i] \xar{(-1)^{i} \ga[i]} L[i+1] 
\xar{(-1)^{i+1} \al[i+1]} M[i+1] \xar{(-1)^{i+1} \be[i+1]} N[i+1] . 
\]
Now use the definition, noting that multiplying morphisms in an exact sequence
preserves exactness.
\end{proof}


\begin{prop} \label{prop:4}
Let $\cat{K}$ be a triangulated category.
\begin{enumerate}
\item If $L \xar{\al} M \xar{\be} N \xar{\ga} T(L)$ is a distinguished triangle
in $\cat{K}$, then $\be \circ \al = 0$. 

\item For any $P \in \cat{K}$ the functors 
\[ \opn{Hom}_{\cat{K}}(-, P) : \cat{K}^{\mrm{op}} \to \cat{Ab} \]
and 
\[ \opn{Hom}_{\cat{K}}(P, -) : \cat{K} \to \cat{Ab} \]
 are cohomological functors. 
\end{enumerate}
\end{prop}

\begin{proof}
(1) By axioms (TR1) and (TR3) we have a commutative diagram 
\[ \UseTips  \xymatrix @C=5ex @R=5ex { 
L
\ar[r]^{1}
\ar[d]_{1}
&
L
\ar[r]^{}
\ar[d]_{\al}
& 
0
\ar[r]^{}
\ar[d]_{}
&
T(L)
\ar[d]_{1}
\\
L
\ar[r]^{\al}
&
M
\ar[r]^{\be}
& 
N
\ar[r]^{\ga}
&
T(L) \ .
}  \]
We see that $\be \circ \al$ factors through $0$. 

\medskip \noindent
(2) We will prove the covariant statement; the contravariant statement is
immediate consequence, since 
\[ \opn{Hom}_{\cat{K}}(M, P) = \opn{Hom}_{\cat{K}^{\mrm{op}}}(P, M) , \]
and $\cat{K}^{\mrm{op}}$ is triangulated (with the correct triangulated
structure to make this true).

Consider a distinguished triangle 
$L \xar{\al} M \xar{\be} N \xar{\ga} T(L)$.
We have to prove that the sequence 
\[ \opn{Hom}_{\cat{K}}(P, L) \xar{\al \circ} 
\opn{Hom}_{\cat{K}}(P, M) \xar{\be \circ} 
\opn{Hom}_{\cat{K}}(P, N)  \]
is exact. In view of part (1), all we need to show is that for any 
$\psi : P \to M$ s.t.\ $\be \circ \psi = 0$, there is some  
$\phi : P \to L$ s.t.\ $\psi = \al \circ \phi$.  In a picture, we must show
that the diagram below (solid arrows) 
\[ \UseTips  \xymatrix @C=5ex @R=5ex { 
P
\ar[r]^{1}
\ar@{-->}[d]_{\phi}
&
P
\ar[r]^{}
\ar[d]_{\psi}
& 
0
\ar[r]^{}
\ar[d]_{}
&
T(P)
\ar@{-->}[d]_{\phi}
\\
L
\ar[r]^{\al}
&
M
\ar[r]^{\be}
& 
N
\ar[r]^{\ga}
&
T(L) \ .
}  \]
can be completed (dashed arrow). 
This is true by (TR) (=turning) and and (TR3) (=extending). 
\end{proof}

I need to remind you of the Yoneda Lemma. Let $\cat{C}$ be any category. There
is a related category 
$\cat{Fun}(\cat{C}^{\mrm{op}}, \cat{Set})$
whose objects are the functors $F : \cat{C}^{\mrm{op}} \to \cat{Set}$, and
whose morphisms are the natural transformations. (From the set theory point of
view this construction requires enlarging the universe.)
Any object $C \in \cat{Set}$ gives rise to an object 
$G_C \in \cat{Fun}(\cat{C}^{\mrm{op}}, \cat{Set})$,
namely the functor 
$G_C := \opn{Hom}_{\cat{C}}(-, C)$.

\begin{prop}[Yoneda Lemma]
The functor 
\[ G : \cat{C} \to \cat{Fun}(\cat{C}^{\mrm{op}}, \cat{Set}) \]
is fully faithful.
\end{prop}

See \cite{KS2} for a proof. 
We get an embedding of $\cat{C}$ into 
$\cat{Fun}(\cat{C}^{\mrm{op}}, \cat{Set})$ as a full subcategory.
A functor in the essential image of $G$ is called a {\em representable functor}.

\begin{prop}
Let $\cat{K}$ be a triangulated category, and let 
\[ \UseTips  \xymatrix @C=5ex @R=5ex { 
L 
\ar[r]^{\al}
\ar[d]_{\phi}
&
M
\ar[r]^{\be}
\ar[d]_{\psi}
& 
N
\ar[r]^{\ga}
\ar[d]_{\chi}
&
T(L)
\ar[d]_{T(\phi)}
\\
L' 
\ar[r]^{\al'}
&
M'
\ar[r]^{\be'}
& 
N'
\ar[r]^{\ga'}
&
T(L') \ .
}  \]
be a morphism of distinguished triangles. If $\phi$ and $\psi$ are isomorphisms,
then $\chi$ is also an isomorphism.
\end{prop}

\begin{proof}
Take an arbitrary $P \in \cat{K}$, and let 
$F := \opn{Hom}_{\cat{K}}(P, -)$. We get a 
commutative diagram 
\[ \UseTips  \xymatrix @C=8ex @R=8ex { 
F(L) 
\ar[r]^{F(\al)}
\ar[d]_{F(\phi)}
&
F(M)
\ar[r]^{F(\be)}
\ar[d]_{F(\psi)}
& 
F(N)
\ar[r]^{F(\ga)}
\ar[d]_{F(\chi)}
&
F(T(L))
\ar[d]_{F(T(\phi))}
\ar[r]^{F(T(\al))}
&
F(T(M)) 
\ar[d]_{F(T(\psi))}
\\
F(L') 
\ar[r]^{F(\al')}
&
F(M')
\ar[r]^{F(\be')}
& 
F(N')
\ar[r]^{F(\ga')}
&
F(T(L'))
\ar[r]^{F(T(\al'))}
&
F(T(M')) 
} \]
in $\cat{Ab}$. By Proposition \ref{prop:4}(2)  the rows in the diagram 
are exact sequences.. 
Since the other vertical arrows are isomorphisms, it follows that 
\[ F(\chi) : \opn{Hom}_{\cat{K}}(P, N) \to 
\opn{Hom}_{\cat{K}}(P, N') \]
is an isomorphism. 

Let us write $G_N := \opn{Hom}_{\cat{K}}(-, N)$
and $G_{N'} := \opn{Hom}_{\cat{K}}(-, N')$. 
The calculation above shows that the
morphism $G_{\chi} : G_{N} \to G_{N'}$ between these  representable functors
is an isomorphism. Using the Yoneda Lemma we conclude that $\chi$ is an
isomorphism.
\end{proof}

\cleardoublepage
\section{The Homotopy Category is Triangulated}

\subsection{Standard Triangles}
Let $\cat{M}$ be an abelian category. We already know that 
$\dcat{K}(\cat{M})$ is a T-additive category, where the shift $T(M) = M[1]$ is 
$M[1]^i = M^{i+1}$ and $\d_{M[1]} = - \d_M$.

By {\em graded object} of $\cat{M}$ we mean a collection 
$M = \{ M^i \}_{i \in \Z}$ of objects. Thus a complex in $\cat{M}$ is a graded
object with a differential.

In what follows we view a direct sum $M_1 \oplus M_2$ 
of objects of $\cat{M}$, or of graded objects of $\cat{M}$, as a column
$\bmat{M_1 \\ M_2}$. Given homomorphisms $\phi_{i, j} : M_j \to N_i$,
the combined morphism 
\[ \phi : M_1 \oplus M_2 \to N_1 \oplus N_2 \]
(compatible with the embeddings and the projections)
has a matrix representation
$\phi = \bmat{ \phi_{1,1} &  \phi_{1,2} \\  \phi_{2,1} &  \phi_{2,2} }$.
Namely 
\[ \phi( \bmat{m_1 \\ m_2}) = 
\bmat{ \phi_{1,1} &  \phi_{1,2} \\  \phi_{2,1} &  \phi_{2,2} } \cdot 
\bmat{m_1 \\ m_2} = 
\bmat{ \phi_{1,1}(m_1) + \phi_{1,2}(m_2) \\
 \phi_{2,1}(m_1) + \phi_{2,2}(m_2)  } . \]
Likewise for other finite direct sums. 

Consider a morphism $\al : L \to M$ in $\dcat{C}(\cat{M})$. Define a complex
$N$ like this: as graded object of $\cat{M}$ we take
$N := L[1] \oplus M$, namely 
\[ N^i := L^{i+1} \oplus M^i = \bmat{L^{i+1} \\ M^i} . \]
The differential
\[ \d_N^i : N^i  = L^{i+1} \oplus M^i \to  N^{i+1}  =  L^{i+2} \oplus M^{i+1} \]
is given by the matrix  
\[ \d^i_N := 
\bmat{ - \d_L^{i+1} & 0 \\[0.5em] \al^{i+1} &  \d_M^{i} } .  \]
We call $N$ the {\em mapping cone} of $\al$, and denote it by 
$\opn{cone}(\al)$.

There are morphisms 
\[ M \xar{\be} N = \bmat{L[1] \\ M} \xar{\ga} L[1] \] 
in $\dcat{C}(\cat{M})$ given in matrix notation by
\[ \be := \bmat{ 0 \\ 1_M  } \ , \ 
\ga := \bmat{ 1_{L[1]} & 0 } . \]

I leave it to the reader to check that $\d_N$ is indeed a differential, and
that $\be$ and $\ga$ are morphisms in $\dcat{C}(\cat{M})$.
We get a triangle 
\begin{equation} \label{eqn:12}
L \xar{\al} M \xar{\be} N \xar{\ga} L[1] , 
\end{equation}
in $\dcat{C}(\cat{M})$, which we call  the the {\em standard
triangle} associated to $\al$. 

Passing to $\dcat{K}(\cat{M})$, we get a triangle 
\begin{equation} \label{eqn:11}
L \xar{\bar{\al}} M \xar{\bar{\be}} N \xar{\bar{\ga}} L[1] , 
\end{equation}
where $\bar{\al}, \bar{\be}, \bar{\ga}$ are the morphisms in 
$\dcat{K}(\cat{M})$ represented by $\al, \be, \ga$ respectively (their homotopy
classes). 
We call (\ref{eqn:11}) the {\em standard triangle in $\dcat{K}(\cat{M})$}
associated to $\al$.

Note that the standard triangle (\ref{eqn:11}) depends functorially on the
morphism $\al$ in $\dcat{C}(\cat{M})$;
but it is not functorial in the morphism $\bar{\al}$ in $\dcat{K}(\cat{M})$.

\subsection{Distinguished triangles in 
\texorpdfstring{$\dcat{K}(\cat{M})$}{K(M)}}

\begin{dfn} \label{dfn:140}
A triangle in $\dcat{K}(\cat{M})$ is called a {\em distinguished triangle} if it
isomorphic, in $\dcat{K}(\cat{M})$, to a standard triangle as in (\ref{eqn:11}).
\end{dfn}

\begin{thm}
The T-additive category $\dcat{K}(\cat{M})$, with the set of distinguished
triangles defined above, is triangulated.
\end{thm}

The proof below is worked out using a hint in \cite{RD}, and a lemma in 
 \cite{KS1} -- I hope it is correct!

\begin{lem}  \label{lem:2}
Let $M \in \dcat{C}(\cat{M})$, and consider the mapping cone 
$N := \opn{cone}(1_M)$. Then the complex $N$ is null-homotopic, i.e.\ 
$0 \to N$ is an isomorphism in $\dcat{K}(\cat{M})$.
\end{lem}

\begin{proof}
We shall exhibit a homotopy $\th$ from $0_N$ to $1_N$. 
Define 
\[ \th^i : N^i = M^{i+1} \oplus M^{i} \to N^{i-1} = M^{i} \oplus M^{i-1} \]
to be the matrix 
\[ \th^i := \bmat{ 0 & 1_{M^i} \\[0.5em] 0 & 0  } . \]
We have
\[ \d_N^{i-1} \circ \th^i + \th^{i+1} \circ \d_N^i =
\bmat{ 1_{M^{i+1}} & 0 \\ 0 & 1_{M^{i}}  }  =  1_{N^i} . \]
\end{proof}


\begin{lem}[{\cite[Lemma 1.4.2]{KS1}}] \label{lem:1}
Consider a morphism $\al : L \to M$ in $\dcat{C}(\cat{M})$, the associated
 standard  triangle
\[ L \xar{\al} M \xar{\be} N \xar{\ga} L[1] , \]
and  the associated  standard  triangle
\[ M \xar{\be} N \xar{\phi} P \xar{\psi} M[1]  \]
in  $\dcat{C}(\cat{M})$. Here $N = \opn{cone}(\al)$ and $P = \opn{cone}(\be)$.
There is a morphism 
$\rho : L[1] \to P$ in $\dcat{C}(\cat{M})$ s.t.\ 
$\bar{\rho}$ is an isomorphism in $\dcat{K}(\cat{M})$, and the diagram 
\[ \UseTips \xymatrix @C=5ex @R=5ex { 
M
\ar[r]^{\bar{\be}}
\ar[d]_{\bar{1}_M}
&
N
\ar[r]^{\bar{\ga}}
\ar[d]_{\bar{1}_N}
& 
L[1]
\ar[r]^{- \bar{\al}[1]}
\ar[d]_{\bar{\rho}}
&
M[1]
\ar[d]_{\bar{1}_{M[1]}}
\\
M
\ar[r]^{\bar{\be}}
&
N
\ar[r]^{\bar{\phi}}
& 
P
\ar[r]^{\bar{\psi}}
&
M[1] 
} \]
commutes in $\dcat{K}(\cat{M})$.
\end{lem}

\begin{proof}
Note that 
$P^i = M^{i+1} \oplus L^{i+1} \oplus M^i$
and $L[1]^i = L^{i+1}$. 
Define morphisms 
$\rho^i : L^{i+1} \to P^i$ 
and
$\chi^i : P^i \to L^{i+1}$ in $\cat{M}$ by the matrix representations
\[ \rho^i := \bmat{ -\al^{i+1} \\ 1_{L^{i+1}} \\ 0 } \ , \
\chi^i := \bmat{ 0 & 1_{L^{i+1}} & 0 } .  \]
We get morphisms of graded objects 
$\rho : L[1] \to P$ and $\chi : P \to L[1]$. 
Direct calculations (please verify!) show that:
\begin{itemize}
\item $\rho$ and $\chi$ are morphisms in $\dcat{C}(\cat{M})$.
\item $\chi \circ \rho = 1_{L[1]}$.
\item $\chi \circ \phi = \ga$.
\item $\psi \circ \rho = -\al[1]$.
\end{itemize}

It remains to prove that $\rho \circ \chi$ is homotopic to $1_P$. 
Define a morphism $\th^i : P^i \to P^{i-1}$ by the matrix
\[ \th^i := \bmat{ 0 & 0 & 1_{M^i} \\ 0 & 0 & 0 \\ 0 & 0 & 0} . \]
We get a morphism $\th : P \to P[-1]$ of graded objects, and 
\[ 1_P - \rho \circ \chi = \th \circ \d_P + \d_P \circ \th . \]
\end{proof}

\begin{proof}[Proof of the Theorem] 
(TR1): The only nontrivial thing to show is that 
\[ M \xar{1_M} M \to 0 \to M[1] \] 
is a distinguished triangle. But this follows from Lemma \ref{lem:2}.

\medskip \noindent
(TR2): This is an immediate consequence of Lemma \ref{lem:1}, since the bottom
triangle there (the one with $P$) is standard.

\medskip \noindent
(TR3): 
Consider a commutative diagram (solid arrows) in $\dcat{K}(\cat{M})$:
\[ \UseTips \xymatrix @C=5ex @R=5ex { 
L 
\ar[r]^{\bar{\al}}
\ar[d]_{\bar{\phi}}
&
M
\ar[r]^{\bar{\be}}
\ar[d]_{\bar{\psi}}
& 
N
\ar[r]^{\bar{\ga}}
\ar@{-->}[d]_{\bar{\chi}}
&
T(L)
\ar[d]_{T(\bar{\phi})}
\\
L' 
\ar[r]^{\bar{\al}'}
&
M'
\ar[r]^{\bar{\be}'}
& 
N'
\ar[r]^{\bar{\ga}'}
&
T(L') 
} \]
where the horizontal triangles are distinguished. 
By definition this diagram is isomorphic to a diagram 
in $\dcat{K}(\cat{M})$, that comes from a diagram 
\[ \UseTips \xymatrix @C=5ex @R=5ex { 
L 
\ar[r]^{\al}
\ar[d]_{\phi}
&
M
\ar[r]^{\be}
\ar[d]_{\psi}
& 
N
\ar[r]^{\ga}
\ar@{-->}[d]_{\chi}
&
T(L)
\ar[d]_{T(\phi)}
\\
L' 
\ar[r]^{\al'}
&
M'
\ar[r]^{\be'}
& 
N'
\ar[r]^{\ga'}
&
T(L') 
} \]
(solid arrows) in $\dcat{C}(\cat{M})$,
in which $N = \opn{cone}(\al)$, $N' = \opn{cone}(\al')$, and the horizontal
triangles are the standard ones. 
However this diagram in $\dcat{C}(\cat{M})$ in only commutative up to homotopy. 
This means that there is a degree $-1$ homomorphism 
$\th : L \to M'$ s.t.\ 
\[ \al' \circ \phi = \psi \circ \al + \d(\th) . \]
For every $i$ define the morphism 
\[ \chi^i : N^i = \bmat{ L^{i+1} \\ M^i } \to 
N'^{\, i} = \bmat{ L'^{\, i+1} \\ M'^{\, i} }
\]
to be left multiplication with the matrix
$\bmat{ \phi^{i+1} & 0 \\ - \th^{i+1} & \psi^{i} }$.
A matrix calculation shows that 
$\chi : N \to N'$ is a morphism in
$\dcat{C}(\cat{M})$. It is easy to see that  
$\chi \circ \be = \be' \circ \psi$ and 
$\phi[1] \circ \ga = \ga' \circ \chi$ in $\dcat{C}(\cat{M})$.
Hence passing to $\dcat{K}(\cat{M})$ we have a morphism of triangles.

\medskip \noindent
(TR4): I will not prove this axiom, since it looks as if we won't need it. 
\end{proof}

\cleardoublepage
\section{Localization of Categories}

\subsection{Definition of localization}
I will start with the general concept of localization of a category.
Later we will talk about localization of linear categories. 
It turns out that it is easier to prove localization results with arrows!

The emphasis will be on morphisms rather than on objects. 
Let $\cat{A}$ be a category.
It will be convenient to write 
\[ \cat{A}(M, N) := \opn{Hom}_{\cat{A}}(M, N)  \]
for $M, N \in \opn{Ob}(\cat{A})$. 

We use sometimes use the notation $a \in \cat{A}$, leaving the objects
implicit. When we write $b \circ a$ for $a, b \in \cat{A}$, we mean that they
are composable.

\begin{dfn}
Let $\cat{A}$ be a category. A {\em multiplicatively closed set} 
$\cat{S}$ in $\cat{A}$ is the data of a subset 
$\cat{S}(M, N) \subset \cat{A}(M, N)$ for any pair of objects
$M, N \in \cat{A}$, such that 
$1_M \in \cat{S}(M, M)$, and such that for any
$s \in \cat{S}(L, M)$ and
$t \in \cat{S}(M, N)$ the composition
$t \circ s \in \cat{S}(L, N)$. 
\end{dfn}

Using our shorthand, we can write the definition like this: $1_M \in \cat{S}$,
and $s, t \in \cat{S}$ implies $t \circ s \in \cat{S}$. 

\begin{dfn}
A {\em weak localization} of $\cat{A}$ with respect to $\cat{S}$ 
is a pair $(\cat{A}_{\cat{S}}, Q)$, consisting of a  category 
$\cat{A}_{\cat{S}}$ and a functor 
$Q : \cat{A} \to \cat{A}_{\cat{S}}$, having the following properties:
\begin{itemize}
\rmitem{i} For every $s \in \cat{S}$, the morphism $Q(s) \in \cat{A}_{\cat{S}}$
is invertible (i.e.\ an isomorphism). 

\rmitem{ii} Suppose $\cat{B}$ is a category, and 
$F : \cat{A} \to \cat{B}$
is a functor such that $F(s)$ is an isomorphism for every 
$s \in \cat{S}$. 
Then there is a pair $(F_{\cat{S}}, \eta)$, consisting of a functor 
$F_{\cat{S}} : \cat{A}_{\cat{S}} \to \cat{B}$
and an isomorphism 
$\eta : F \iso F_{\cat{S}} \circ Q$ of functors
$\cat{A} \to \cat{B}$. 
Furthermore, the pair $(F_{\cat{S}}, \eta)$ is unique up to a unique
isomorphism. 
\end{itemize}
\end{dfn}

The last sentence in the universal property means that if 
$(F_{\cat{S}}', \eta')$ is another such pair, then there is a unique
isomorphism of functors $\zeta : F_{\cat{S}} \iso  F_{\cat{S}}'$
s.t.\ $\eta' = \zeta 
\circ \eta$. 

We refer to $Q$ as the localization functor. 

In a diagram:
\[ \UseTips  \xymatrix @C=10ex @R=8ex  {
\cat{S}
\ar[r]^{\mrm{inc}}
&
\cat{A}
\ar[r]^F _{}="q" 
\ar[d]_{Q}
& 
\cat{B}
\\
&
\cat{A}_{\cat{S}}
\ar[ur]_{F_{\cat{S}}} ^(0.5){}="f"
\ar@{=>}  "q";"f" _{\eta}
} \]
This is is commutative via $\eta$. 
For any other $(F_{\cat{S}}', \eta')$ there is a unique $\ze$ s.t.\ 
$\eta' = \ze \circ \eta$:
\[ \UseTips  \xymatrix @C=14ex @R=12ex {
\cat{A}
\ar[r]^F _{}="q"  ^(0.6){}="q2" 
\ar[d]_{Q}
& 
\cat{B} 
\\
\cat{A}_{\cat{S}}
\ar[ur] ^(0.5){}="f" ^(0.5){}="f1" 
\ar@{=>}  "q";"f" _{\eta}
\ar@(r,d)[ru]_{F_{\cat{S}}'} ^(0.55){}="g"  ^(0.48){}="g1"
\ar@{=>}  "q2";"g" ^(0.7){\eta'}
\ar@{=>}  "f";"g1" _(0.5){\ze}
} \]
. 

\begin{prop}
A weak localization $(\cat{A}_{\cat{S}}, Q)$ is unique up to an equivalence, 
and this equivalence is unique up to a unique isomorphism. 
Namely if $(\cat{A}'_{\cat{S}}, Q')$ is another localization of $\cat{A}$
w.r.t.\ $\cat{S}$, then there is a pair $(G, \eta)$, consisting of an
equivalence 
$G : \cat{A}_{\cat{S}} \to \cat{A}_{\cat{S}}'$
and an isomorphism of functors 
$\eta : Q' \iso G \circ Q$.
Moreover, the pair $(G, \eta)$ is unique up to a unique isomorphism.
\end{prop}

In a diagram:
\[ \UseTips  \xymatrix @C=10ex @R=8ex  {
\cat{S}
\ar[r]^{\mrm{inc}}
&
\cat{A}
\ar[r]^{Q'} _{}="q" 
\ar[d]_{Q}
& 
\cat{A}_{\cat{S}}'
\\
&
\cat{A}_{\cat{S}}
\ar[ur]_{G} ^(0.5){}="f"
\ar@{=>}  "q";"f" _{\eta}
} \]

\begin{proof}
Exercise. 
\end{proof}

The general concept of localization is quite messy. We prefer:

\begin{dfn}
A {\em strict localization} of $\cat{A}$ with respect to $\cat{S}$ 
is a pair $(\cat{A}_{\cat{S}}, Q)$, consisting of a category 
$\cat{A}_{\cat{S}}$ and a functor 
$Q : \cat{A} \to \cat{A}_{\cat{S}}$, having the following properties:
\begin{itemize}
\rmitem{i} For every $s \in \cat{S}$, the morphism $Q(s) \in \cat{A}_{\cat{S}}$
is invertible. 

\rmitem{ii} There is equality 
$\opn{Ob}(\cat{A}_{\cat{S}}) = \opn{Ob}(\cat{A})$, and $Q$ is the identity on
objects.

\rmitem{iii} Suppose $\cat{B}$ is a category, and 
$F : \cat{A} \to \cat{B}$
is a functor such that $F(s)$ is invertible for every 
$s \in \cat{S}$. 
Then there is a unique functor 
$F_{\cat{S}} : \cat{A}_{\cat{S}} \to \cat{B}$
such that 
$F_{\cat{S}} \circ Q = F$ 
as functors $\cat{A} \to \cat{B}$.
\end{itemize}
\end{dfn}

\begin{prop} \label{prop:5}
\begin{enumerate}
\item A strict localization is a weak localization. 

\item A strict localization is unique up to a unique isomorphism of categories.
Namely if
$(\cat{A}'_{\cat{S}}, Q')$ is another strict localization, then there is a
unique functor $G : \cat{A}_{\cat{S}} \to \cat{A}'_{\cat{S}}$
which is the identity on objects, bijective on morphisms, and 
$G \circ Q = Q'$. 
\end{enumerate}
\end{prop}

\begin{proof}
Exercise.
\end{proof}

\subsection{Ore localization}
There is even a better notion of localization. 
The references here are \cite{RD, Wei, Ste, Row}. 

\begin{dfn}
A {\em right Ore localization} of $\cat{A}$ with respect to $\cat{S}$ 
is a pair $(\cat{A}_{\cat{S}}, Q)$, consisting of a category 
$\cat{A}_{\cat{S}}$ and a functor 
$Q : \cat{A} \to \cat{A}_{\cat{S}}$, having the following properties:
\begin{itemize}
\item[(L1)] There is equality 
$\opn{Ob}(\cat{A}_{\cat{S}}) = \opn{Ob}(\cat{A})$, and $Q$ is the identity on
objects.

\item[(L2)]  For every $s \in \cat{S}$, the morphism $Q(s) \in
\cat{A}_{\cat{S}}$
is invertible. 

\item[(L3)]  Every morphism $q \in \cat{A}_{\cat{S}}$ can be written as
$q = Q(a) \circ Q(s)^{-1}$ for some 
$a \in \cat{A}$ and $s \in \cat{S}$.

\item[(L4)]  Suppose $a, b \in \cat{A}$ satisfy 
$Q(a) = Q(b)$. Then $a \o s = b \o s$ for some $s \in \cat{S}$. 
\end{itemize}
\end{dfn}


\begin{lem} \label{lem:3}
Let $(\cat{A}_{\cat{S}}, Q)$ be an Ore localization, 
let $a_1, a_2 \in \cat{A}$ and $s_1, s_2 \in \cat{S}$.
TFAE:
\begin{enumerate}
\rmitem{i} 
$Q(a_1) \o Q(s_1)^{-1} =  Q(a_2) \o Q(s_2)^{-1}$
in $\cat{A}_{\cat{S}}$. 

\rmitem{ii} There are $b_1, b_2 \in \cat{A}$ s.t.\ 
$a_1 \o b_1 = a_2 \o b_2$,
and $s_1 \o b_1  = s_2 \o b_2 \in \cat{S}$.
\end{enumerate}
\end{lem}

\begin{proof}
(ii) $\Rightarrow$ (i):
Since $Q(s_i)$ and $Q(s_i \o b_i)$ are invertible, it follows that 
$Q(b_i)$ are invertible. So 
\[ \begin{aligned}
& Q(a_1) \o Q(s_1)^{-1} = Q(a_1) \o Q(b_1) \o Q(b_1)^{-1} \o Q(s_1)^{-1} 
\\
& \quad = 
Q(a_2) \o Q(b_2) \o Q(b_2)^{-1} \o Q(s_2)^{-1} = 
 Q(a_2) \o Q(s_2)^{-1} .
\end{aligned} \]

\medskip \noindent
(i) $\Rightarrow$ (ii):
This is almost from \cite{Ste}. 
By property (L3) there are $c \in \cat{A}$ and $u \in \cat{S}$ s.t.\
\[ Q(s_2)^{-1} \circ Q(s_1) = Q(c) \o Q(u)^{-1} . \]
This gives 
\[ Q(s_1) \o Q(u) = Q(s_2) \o Q(c) . \]
Hence  
\[ Q(a_1) = Q(a_2) \o Q(s_2)^{-1} \o Q(s_1) = Q(a_2) \o Q(c) \o Q(u)^{-1}  \]
and  
\[ Q(a_1) \o Q(u) = Q(a_2) \o Q(c) . \]

Because 
$Q(a_1 \o u) = Q(a_2 \o c)$, by property (L4) there is 
$v \in \cat{S}$ s.t.\ 
\[ a _1 \o u \o v = a_2 \o c \o v . \]
Likewise $Q(s_1 \o u) = Q(s_2 \o c)$, so there is $v' \in \cat{S}$
s.t.\ 
\[ s_1 \o u \o v' = s_2 \o c \o v' . \] 

Again using property (L3), there are $d \in \cat{A}$ and $w \in \cat{S}$ s.t.\
\[ Q(v)^{-1} \circ Q(v') = Q(d) \o Q(w)^{-1} . \]
Rearranging we get 
\[ Q(v' \o w) = Q(v \o d) . \]
 By property (L4) there is $w' \in \cat{S}$ s.t.\ 
\[ v' \o w \o w' = v \o d \o w' . \]
Define 
\[ b_1 := u \o v \o  d \o w' \ , \ b_2 := c \o v \o d \o w' \ . \]
Then 
\[ \begin{aligned}
& s_1 \o b_1 = s_1 \o u \o v \o  d \o w' = 
s_1 \o u \o v' \o w \o w' 
\\
& \quad = s_2 \o c \o v' \o w \o w' = s_2 \o b_2 , 
\end{aligned} \]
and it is in $\cat{S}$. 
Also 
\[ a_1 \o b_1 = a_1 \o u \o v \o  d \o w' =
a_2 \o c \o v  \o  d \o w' = a_2 \o b_2 . \]
\end{proof}

\begin{prop} \label{prop:103}
A right Ore localization is a strict localization. 
\end{prop}

\begin{proof}
Say  $\cat{B}$ is a category, and 
$F : \cat{A} \to \cat{B}$
is a functor such that $F(s)$ is an isomorphism for every 
$s \in \cat{S}$. 

The uniqueness of a functor
$F_{\cat{S}} : \cat{A}_{\cat{S}} \to \cat{B}$
satisfying $F_{\cat{S}} \circ Q = F$  is clear from property (L3).
We have to prove existence.

Define $F_{\cat{S}}$
to be $F$ on objects, and 
\[ F_{\cat{S}}(q) := F(a_1) \circ F(s_1)^{-1} \]
for 
\[ q = Q(a_1) \circ Q(s_1)^{-1} \in \cat{A}_{\cat{S}} , \
a_1 \in \cat{A},  \ s_1 \in \cat{S} . \] 
We have to prove this is well defined.
So suppose that 
$q = Q(a_2) \circ Q(s_2)^{-1}$
is another presentation of $q$.
Let $b_1, b_2 \in \cat{A}$ be as in the lemma. 
Since $F(s_i)$ and $F(s_i \o b_i)$ are invertible, then so is $F(b_i)$. 
We get
\[ F(a_2) = F(a_1) \o F(b_1) \o F(b_2)^{-1} \]
and 
\[ F(s_2) = F(s_1) \o F(b_1) \o F(b_2)^{-1} . \]
Hence 
\[ F(a_2) \circ F(s_2)^{-1}  = F(a_1) \circ F(s_1)^{-1}  . \]
\end{proof}

There are corresponding ``left'' versions of the definitions and the results. 

\begin{cor}
\begin{enumerate}
\item A right Ore localization is unique up to a unique isomorphism.

\item If a right Ore localization $(\cat{A}_{\cat{S}}, Q)$ and a left Ore 
localization $(\cat{A}'_{\cat{S}}, Q')$ both exist, then they are uniquely
isomorphic. 
\end{enumerate}
\end{cor}

\begin{proof}
Both assertions are immediate consequences of Propositions \ref{prop:103} and 
\ref{prop:5}.
\end{proof}

\begin{dfn} \label{dfn:21}
Let $\cat{S}$ be multiplicatively closed set in a  category $\cat{A}$. We
say that $\cat{S}$ is a {\em right denominator set} if it satisfies these two
conditions:
\begin{enumerate}
\item[(D1)]
(Right Ore condition) Given $a \in \cat{A}$ and $s \in \cat{S}$,
there exist $b \in \cat{A}$ and $t \in \cat{S}$ s.t.\ 
$a \o t = s \o b$.

\item[(D2)] 
(Right cancellation condition) Given $a, b \in \cat{A}$
and  $s \in \cat{S}$
s.t.\ $s \o a = s \o b$, there exists $t \in \cat{S}$ s.t.\ 
$a \o t = b \o t$.
\end{enumerate}
\end{dfn}

In diagrams:

\[ 
\UseTips  \xymatrix @C=5ex @R=5ex {
K
\ar@{-->}[d]_{t}
\\
M
\ar@(dl,ul)[d]_{a}
\ar@(dr,ur)[d]^{b}
\\
N
\ar[d]_{s}
\\
L
} 
\qquad \qquad
\UseTips  \xymatrix @C=5ex @R=5ex  {
& 
K
\ar@{-->}[dl]_{t}
\ar@{-->}[dr]^{b}
\\
M 
\ar[dr]_{a}
& &
N
\ar[dl]^{s}
\\
&
L 
} \]

\begin{thm} \label{thm:103}
The following conditions are equivalent for a category $\cat{A}$ 
and a multiplicatively closed set $\cat{S} \subset \cat{A}$. 
\begin{enumerate}
\rmitem{i} The right Ore localization $(\cat{A}_{\cat{S}}, Q)$ exists.

\rmitem{ii} $\cat{S}$ is a right denominator set. 
\end{enumerate}
\end{thm}

The proof is a bit later. Of course Definition \ref{dfn:21} and Theorem
\ref{thm:103} have ``left'' versions.

Let's assume that $\cat{S}$ is a right denominator set. 
For any $M, N \in \opn{Ob}(A)$ consider the set 
\[ (\cat{A} \times \cat{S})(M, N) :=  
\coprod_{L \in \opn{Ob}(A)} 
\cat{A}(L, N) \times \cat{S}(L, M) . \]
This could be a very big set... So an element 
$(a, s) \in (\cat{A} \times \cat{S})(M, N)$ can be pictured as a diagram
\[ \UseTips  \xymatrix @C=5ex @R=5ex  {
& L 
\ar[dl]_{s}
\ar[dr]^{a}
\\
M 
& &
N
} \]
in $\cat{A}$.

We define a relation $\sim$ on the set $\cat{A} \times \cat{S}$
like this:
\[ (a_1, s_1) \sim (a_2, s_2) \]
if there exist $b_1, b_2 \in \cat{A}$ s.t.\ 
\[ a_1 \o b_1 = a_2 \o b_2 \tup{ and } 
s_1 \o b_1 = s_2 \o b_2 \in \cat{S} . \]

Note that this relation imposes condition (ii) of Lemma \ref{lem:3}.

In a commutative diagram:
\[ \UseTips  \xymatrix @C=6ex @R=6ex  {
&
K
\ar[dl]_{b_1}
\ar[dr]^{b_2}
\\
L_1
\ar[d]_{s_1}
\ar[drr]_(0.75){a_1}
& &
L_2
\ar[dll]_(0.75){s_2}
\ar[d]^{a_2}
\\
M 
& &
N
} \]
The arrows ending at $M$ are in $\cat{S}$. 

\begin{lem} \label{lem:4}
If the right Ore condition holds then the relation $\sim$ is an equivalence.
\end{lem}

\begin{proof}
Reflexivity: take $K := L$ and $b_i := 1_L : L \to L$. 
Symmetry is trivial. 

We shall use the right Ore condition (D1) to prove transitivity. 
Suppose we are given 
$(a_1, s_1) \sim (a_2, s_2)$ and $(a_2, s_2) \sim (a_3, s_3)$.
So we have the first solid commutative diagram in  Figure \ref{fig:5}.

\begin{figure}
\[ \UseTips  \xymatrix @C=8ex @R=8ex  {
&
H
\ar@{-->}[d]^{u}
\\
&
I
\ar@{-->}[dl]_{t}
\ar@{-->}[dr]^{d}
\\
K
\ar[d]_{b_1}
\ar[dr]^{b_2}
& & 
J
\ar[dl]_{c_2}
\ar[d]^{c_3}
\\
L_1
\ar[d]_{s_1}
\ar[drr]_(0.65){a_1}
&
L_2
\ar[dl]_(0.65){s_2}
\ar[dr]^(0.65){a_2}
&
L_3
\ar[dll]^(0.65){s_3}
\ar[d]^{a_3}
\\
M 
& &
N
} 
\qquad \qquad 
\UseTips  \xymatrix @C=8ex @R=8ex  {
&
H
\ar[dddl]_{b_1 \o t \o u}
\ar[dddr]^{a_3 \o d \o u}
\\
\\
\\
L_1
\ar[d]_{s_1}
\ar[drr]_(0.65){a_1}
&
&
L_3
\ar[dll]^(0.65){s_3}
\ar[d]^{a_3}
\\
M 
& &
N
} \]

\caption{} 
\label{fig:5}
\end{figure}
The arrows ending at $M$ are in $\cat{S}$. 
By the right Ore condition applied to $K \to M \leftarrow J$ there are 
$t \in \cat{S}$ and $d \in \cat{A}$ s.t.\ 
\[ (s_3 \o c_3) \o d = (s_1 \o b_1) \o t . \]
Now 
\[ s_2 \o (b_2 \o t) = s_1 \o (b_1 \o t) = s_2 \o (c_2 \o d) . \]
By (D2) there is $u \in \cat{S}$ s.t.\ 
\[ (b_2 \o t) \o u = (c_2 \o d) \o u . \]
So all paths from $H \to M$ and $H \to N$ commute, and all paths ending at $M$
are in $\cat{S}$. Now delete $L_2$ and the arrows through it. Then delete 
$I, J, K$,  but keep the paths through them. 
We get the second diagram in Figure \ref{fig:5}. It is commutative, and all
arrows ending at $M$ are in $\cat{S}$. We have evidence for 
$(a_1, s_1) \sim (a_3, s_3)$.
\end{proof}


\begin{proof}[Proof of Theorem \tup{\ref{thm:103}}]
\mbox{}

\medskip \noindent
{\bf Step 1.} We prove (i) $\Rightarrow$ (ii). 
Take $a \in \cat{A}$ and $s \in \cat{S}$. 
Consider $q := Q(s)^{-1} \o Q(a)$. By (L3) there are 
$b \in \cat{A}$ and $t \in \cat{S}$ s.t.\ 
$q = Q(b) \o Q(t)^{-1}$. 
So
\[ Q(s \o b) = Q(a \o t) . \]
By (L4) there is $u \in \cat{S}$ s.t.\ 
\[ s \o b \o u = a \o t \o u . \]
We read this as 
\[ s \o (b \o u) = a \o (t \o u) , \]
and note that $t \o u \in \cat{S}$. So (D1) holds.

Next $a, b \in \cat{A}$ and $s \in \cat{S}$ s.t.\ 
$s \o a = s \o b$.  Then 
$Q(s \o a) = Q(s \o b)$. But $Q(s)$ is invertible, so 
$Q(a) = Q(b)$. By (L4) there is $t \in \cat{S}$ s.t.\ 
$a \o t = b \o t$. We have proved (D2).

\medskip \noindent
{\bf Step 2.} Now we assume that condition (ii) holds, and we define 
the sets $\cat{A}_{\cat{S}}(M, N)$, composition between them, and the identity
morphisms. 

For any $M, N \in \opn{Ob}(\cat{A})$ let
\[ \cat{A}_{\cat{S}}(M, N) :=
\frac{(\cat{A} \times \cat{S})(M, N)}{\sim} , \]
where $\sim$ is the equivalence relation from Lemma \ref{lem:4}.

We define composition like this. Given 
$q_1 \in \cat{A}_{\cat{S}}(M_0, M_1)$
and 
$q_2 \in \cat{A}_{\cat{S}}(M_1, M_2)$,
choose representatives 
$(a_i, s_i) \in (\cat{A} \times \cat{S})(M_{i-1}, M_i))$.
We use the notation 
$q_i = \ol{ (a_i, s_i) }$ to indicate this. 
By (D1) there are $c \in \cat{A}$ and $u \in \cat{S}$ s.t.\ 
$s_2 \o c = a_1 \o u$. 
The composition 
\[ q_2 \o q_1 \in \cat{A}_{\cat{S}}(M_0, M_2) \]
is defined to be
\[ q_2 \o q_1 := \ol{ (a_2 \o c, s_1 \o u) } 
 \in (\cat{A} \times \cat{S})(M_{0}, M_2)) . \]

The algebraic idea behind the formula is this: we want
\[ \begin{aligned}
& q_2 \o q_1 = 
Q(a_2) \o Q(s_2)^{-1} \o Q(a_1) \o Q(s_1)^{-1} 
\\
& \quad 
= Q(a_2) \o Q(c) \o Q(u)^{-1} \o Q(s_1)^{-1} =
Q(a_2 \o c) \o Q(s_1 \o u)^{-1} . 
\end{aligned} \]
See diagram in Figure \ref{fig:2}. 

\begin{figure}
\[ \UseTips  \xymatrix @C=6ex @R=6ex  {
& & 
K
\ar[dl]_{c}
\ar[dr]^{u}
\\
&
L_1
\ar[dl]_{s_1}
\ar[dr]_(0.5){a_1}
& & 
L_2
\ar[dl]_(0.5){s_2}
\ar[dr]^{a_2}
\\
M_0 
& &
M_1
& &
M_2
} \]
\caption{} 
\label{fig:2}
\end{figure}

We have to verify that this definition is independent of the representatives.
So suppose we take other representatives 
$q_i = \ol{ (a'_i, s'_i) }$, and we choose $u', c'$ to
construct a composition (see diagram in Figure \ref{fig:3}). We must prove that 
\[ \ol{ (a_2 \o c, s_1 \o u) }  = \ol{ (a'_2 \o c', s'_1 \o u') } . \]

\begin{figure}
\[ \UseTips  \xymatrix @C=8ex @R=8ex  {
H''
\ar@{-->}[r]^{w''}
& 
H'
\ar@{-->}[r]^{w'}
& 
H
\ar@{-->}[dl]_{w}
\ar@{-->}[dr]^{e}
\\
&
I_1
\ar@{-->}[dl]_{v_1}
\ar@{-->}[d]^{d_1}
& & 
I_2
\ar@{-->}[d]_{v_2}
\ar@{-->}[dr]^{d_2}
\\
J_1
\ar[d]_{b_1}
\ar[dr]^(0.7){b'_1}
&
K
\ar[dl]^(0.7){u}
\ar[drr]^(0.7){c}
& & 
K'
\ar[dll]_(0.7){u'}
\ar[dr]_(0.7){c'}
& 
J_2
\ar[dl]_(0.7){b_2}
\ar[d]^{b'_2}
\\
L_1
\ar[d]_{s_1}
\ar[drr]_(0.65){a_1}
&
L'_1
\ar[dl]^(0.65){s'_1}
\ar[dr]^{a'_1}
& & 
L_2
\ar[dl]_{s_2}
\ar[dr]_(0.65){a_2}
& 
L'_2
\ar[dll]^(0.65){s'_2}
\ar[d]^{a'_2}
\\
M_0 
& &
M_1
& &
M_2
} \]
\caption{} 
\label{fig:3}
\end{figure}

\begin{figure}
\[ \UseTips  \xymatrix @C=8ex @R=8ex  {
& & 
H''
\ar[dl]_{w \o w' \o w''}
\ar[dr]^{e \o w' \o w''}
\\
&
I_1
\ar[d]^{d_1}
& & 
I_2
\ar[d]_{v_2}
\\
&
K
\ar[dl]^(0.7){u}
\ar[drr]^(0.7){c}
& & 
K'
\ar[dll]_(0.7){u'}
\ar[dr]_(0.7){c'}
& 
\\
L_1
\ar[d]_{s_1}
&
L'_1
\ar[dl]^(0.65){s'_1}
& & 
L_2
\ar[dr]_(0.65){a_2}
& 
L'_2
\ar[d]^{a'_2}
\\
M_0 
& &
& &
M_2
} \]
\caption{} 
\label{fig:4}
\end{figure}

There are morphisms $b_i, b'_i$ the are evidence for 
$(a_i, s_i) \sim (a'_i, s'_i)$.
The solid diagram in Figure \ref{fig:3} is commutative. 
All morphisms ending at $M_0$ (i.e.\ oriented paths ending at $M_0$)  are
in $\cat{S}$.  Also the morphism $J_2 \to M_1$ is in $\cat{S}$.

Choose $v_1 \in \cat{S}$ and $d_1 \in \cat{A}$ s.t.\ the diagram above $L_1$
is commutative. This can be done by (D1). 


Consider the solid diagram on the left side of Figure \ref{fig:7}.
Since $J_2 \to M_1$ is in $\cat{S}$, by (D1) there are 
$\til{v}_1 \in \cat{S}$ and $\til{d}_1 \in \cat{A}$ s.t.\ the two paths
$\til{I}_2 \to M_1$ are equal.
By (D2) there is $\til{v} \in \cat{S}$ s.t.\ the two paths 
$I_2 \to L'_2$ are equal. 
We get the commutative diagram on the right side of Figure \ref{fig:7},
with 
$d_2 := \til{d}_2 \o \til{v}$ and 
$v_2 := \til{v}_2 \o \til{v} \in \cat{S}$.

\begin{figure}
\[ \UseTips  \xymatrix @C=5ex @R=5ex {
&
I_2
\ar@{-->}[d]^{\til{v}}
\\
& 
\til{I}_2
\ar@{-->}[dl]_{\til{v}_2}
\ar@{-->}[dr]^{\til{d}_2}
\\
K'
\ar[dr]_{c'}
& &
J_2
\ar[dl]^{b'_2}
\\
&
L'_2 
\ar[d]_{s'_2}
\\
&
M_1
} 
\qquad \qquad 
\xymatrix @C=5ex @R=5ex {
\\
&
I_2
\ar[dl]_{v_2}
\ar[dr]^{d_2}
\\
K'
\ar[dr]_{c'}
& &
J_2
\ar[dl]^{b'_2}
\\
&
L'_2
} \]
\caption{} 
\label{fig:7}
\end{figure}

We now embed the diagram on the right side of Figure \ref{fig:7}
into the diagram in Figure \ref{fig:3}.
At this stage the whole diagram in Figure \ref{fig:3} is commutative.

Choose $w \in \cat{S}$ and $e \in \cat{A}$ to fill the diagram
$I_2 \to M_0 \leftarrow I_2$, using (D1). The path $H \to I_1 \to M_0$ is
in $\cat{S}$.  But we could have failure of commutativity 
in the paths $H \to L'_1$ and $H \to L_2$. 

The two paths  $H \to L'_1$ satisfy 
\[ s'_1 \o (b'_1 \o v \o w) = s'_1 \o (u' \o v_2 \o e) . \]
Therefore there is $w' \in \cat{S}$ s.t.\ 
\[  (b'_1 \o v \o w) \o w' = (u' \o v_2 \o e) \o w' . \]
Next, the two paths  $H' \to L_2$ satisfy 
\[ s_2 \o (c \o d_1 \o w \o w') = s_2 \o (b_2 \o d_2 \o e \o w') . \]
Therefore there is $w'' \in \cat{S}$ s.t.\ 
\[ (c \o d_1 \o w \o w') \o w'' = (b_2 \o d_2 \o e \o w') \o w'' . \]
Now all paths $H'' \to M_2$ are equal. All paths $H'' \to M_0$ are equal and
are in $\cat{S}$.

Erase $M_1, J_1, J_2$ and all arrows touching them. 
Then erase $H, H'$, but keep the paths through them. 
So we have the commutative diagram of Figure \ref{fig:4}).
This is evidence for 
\[  (a_2 \o c, s_1 \o u) \sim (a'_2 \o c', s'_1 \o u') . \]
The proof that composition is well-defined is done.

The identity morphism $1_M$ of an object $M$ is $\ol{(1_M, 1_M)}$.  

\medskip \noindent
{\bf Step 3.} 
We have to verify the associativity and the identity properties of 
composition in $\cat{A}_{\cat{S}}$. Namely that $\cat{A}_{\cat{S}}$ is a
category. This seems to be not too hard, given Step
2, and we leave it as an exercise!

\medskip \noindent
{\bf Step 4.}  
The functor $Q : \cat{A} \to \cat{A}_{\cat{S}}$ is defined to be 
$Q(M) := M$, and $Q(a) := \ol{(a, 1_M)}$ for $a : M \to N$ in
$\cat{A}$. We have to verify this is a functor... Again, an exercise.

\medskip \noindent
{\bf Step 5.}  
Finally we verify properties (L1)-(L4). (L1) is clear. The inverse of
$Q(s)$ is $\ol{(1, s)}$, so (L2) holds. 

It is not hard to see that 
\[ \ol{(a, s)} = \ol{(a, 1)} \o \ol{(1, s)}  ; \]
this is (L3). 

If $Q(a_1) = Q(a_2)$, then $(a_1, 1_M) \sim (a_2, 1_M)$; so there
are
$b_1, b_2 \in \cat{A}$  s.t.\ 
$a_1 \o b_1 = a_2 \o b_2$ and $1 \o b_1 = 1 \o b_2 \in \cat{S}$. 
Writing $s := b_1 \in \cat{S}$, we get $a_1 \o s = a_2 \o s$. This proves (L4). 
\end{proof}


\begin{rem} \label{rem:1}
Suppose $A$ is a ring and $S$ is a right denominator set in it. 
Then the right Ore localization $A_S$ is {\em flat} as left $A$-module. 
See \cite[Theorem 3.1.20]{Row}.
I have no idea if something like this is true for linear categories with more
than one object. 
\end{rem}

\cleardoublepage
\section{The Derived Category}

\subsection{Localization of linear categories}

\begin{prop} \label{prop:107}
Let $\cat{A}$ be a category, let $\cat{S}$ be a right
denominator set in $\cat{A}$, and let $(\cat{A}_{\cat{S}}, Q)$ be the right Ore
localization.
For any two morphisms $q_1, q_2 : M \to N$ in $\cat{A}_{\cat{S}}$ there
is a common denominator. Namely we can write 
\[ q_i = Q(a_i) \o Q(s)^{-1} \]
for suitable $a_i \in \cat{A}$ and $s \in \cat{S}$. 
\end{prop}

\begin{proof}
Choose representatives
$q_i = Q(a'_i) \o Q(s'_1)^{-1}$. 
By (D1) applied to $L_1 \to M \leftarrow L_2$, 
there are $b \in \cat{A}$ and $t \in \cat{S}$ 
s.t.\ the diagram above $M$ commutes:
\[ \UseTips  \xymatrix @C=6ex @R=6ex  {
&
L
\ar[dl]_{t}
\ar[dr]^{b}
\\
L_1
\ar[d]_{s'_1}
\ar[drr]_(0.75){a'_1}
& &
L_2
\ar[dll]_(0.75){s'_2}
\ar[d]^{a'_2}
\\
M 
& &
N
} \]
Write
$s := s'_1 \o t = s'_2 \o b$,  $a_1 := a'_1 \o t$ and 
$a_1 := a'_2 \o b$. 
By Lemma \ref{lem:3} we get 
$q_i = Q(a_i) \o Q(s)^{-1}$.
\end{proof}

\begin{thm} \label{thm:104}
Let $\cat{A}$ be a $\K$-linear category, let $\cat{S}$ be a right
denominator set in $\cat{A}$, and let $(\cat{A}_{\cat{S}}, Q)$ be the right Ore
localization.

\begin{enumerate}
\item  The category $\cat{A}_{\cat{S}}$ has a unique $\K$-linear structure
such that $Q$ becomes a $\K$-linear functor. 

\item Suppose $\cat{B}$ is another $\K$-linear category, and 
$F : \cat{A} \to \cat{B}$ is a $\K$-linear functor s.t.\
$F(s)$ is invertible for every $s \in \cat{S}$. Let 
$F_{\cat{S}} : \cat{A}_{\cat{S}} \to \cat{B}$ be the localization of $F$. Then 
$F_{\cat{S}}$ is a $\K$-linear functor. 
\end{enumerate}
\end{thm}

\begin{proof}
(1) Let $q_i : M \to N$ be morphisms in $\cat{A}_{\cat{S}}$.
Choose common denominator presentations
$q_i = Q(a_i) \o Q(s)^{-1}$. Since $Q$ must be an additive functor, we have
to define 
\[ Q(a_1) + Q(a_2) := Q(a_1 + a_2) . \]
By the distributive law (bilinearity of composition) we must define 
\[ q_1 + q_2 := \bigl( Q(a_1) + Q(a_2) \bigr) \o Q(s)^{-1} = 
Q(a_1 + a_2) \o Q(s)^{-1}  . \]
For $\la \in \K$ we must define 
\[ \la \cdot q_i := Q(\la a_i) \o Q(s)^{-1}  . \]
The usual tricks are then used to prove independence of representatives. So 
$\cat{A}_{\cat{S}}$ is a $\K$-linear category, and $Q$ is a $\K$-linear
functor. 

\medskip \noindent (2)
The only option for $F_{\cat{S}}$ is 
$F_{\cat{S}}(q_i) := F(a_i) \o F(s)^{-1}$. 
The usual tricks are used to prove independence of representatives.
\end{proof}

This includes the case $\K = \Z$ of course. There are ``left'' versions of
these results. 

\begin{exa}
Let $A$ be a ring, which we can think of as a one object linear category
$\cat{A}$. In this context, Theorem
\ref{thm:104} is one of the most important results in ring theory. See
\cite{Row, Ste}. 
\end{exa}

\begin{prop} \label{prop:105}
Let $\cat{A}$ be an additive category, let $\cat{S}$ be a right
denominator set in $\cat{A}$, and let $(\cat{A}_{\cat{S}}, Q)$ be the right Ore
localization. Then the category $\cat{A}_{\cat{S}}$ has finite direct sums. 
So $\cat{A}_{\cat{S}}$ is an additive category. 
\end{prop}

\begin{proof}
Clear from Propositions \ref{prop:2} and \ref{prop:109}.
\end{proof}

\subsection{Localization of triangulated categories}

Let $\cat{K}$ be a triangulated category, with translation $T$.

\begin{prop} \label{prop:106}
Suppose 
$H : \cat{K} \to \cat{M}$ is a cohomological functor, where $\cat{M}$ is some
abelian category. Let 
\[ \cat{S} := \{ s \in \cat{K} \mid H( T^i(s)) \tup{ is invertible for all }
i \in \Z  \} . \] 

Then $\cat{S}$ is a left and right denominator set in $\cat{K}$.
\end{prop}

\begin{proof}
It is clear that $\cat{S}$ is closed under composition and contains the
identity morphisms. So it is a multiplicatively closed set. 

Let's prove that the right Ore condition (D1) holds. Suppose we are given 
$L \xar{a} N \xleftarrow{s} M$. 
Consider the solid commutative diagram 
\[ \UseTips  \xymatrix @C=5ex @R=5ex { 
K
\ar[r]^{t}
\ar@{-->}[d]_{b}
&
L
\ar[r]^{c \o a}
\ar[d]_{a}
& 
P
\ar[r]^{}
\ar[d]_{=}
&
K[1]
\ar@{-->}[d]_{b[1]}
\\
M
\ar[r]^{s}
&
N
\ar[r]^{c}
& 
P
\ar[r]^{}
&
M[1]
}  \]
where the bottom row is a distinguished triangle built on $M \xar{s} N$,
and and the top row is a distinguished triangle built on 
$L \xar{c \o a} P$, turned $120^\circ$ to the right. By axiom (TR3) there is a
morphism $b$ making the diagram commutative. Since $\mrm{H}(s[i])$ are
invertible for all $i \in \Z$, it follows that $\mrm{H}(P[i]) = 0$. But then 
$\mrm{H}(t[i])$ are invertible for all $i \in \Z$, so $t \in \cat{S}$. 

Finally we prove (D2). Say $a \in \cat{K}$ and $s \in \cat{S}$ satisfy 
$a \o s = 0$. Let 
\[ P \xar{b} M \xar{s} N \xar{} P[1] \]
be a distinguished triangle built on $s$. We get an exact sequence 
\[ \opn{Hom}_{\cat{K}}(L, P) \xar{b \o} 
\opn{Hom}_{\cat{K}}(L, M) \xar{s \o} 
\opn{Hom}_{\cat{K}}(L, N) . \]
Since $a : L \to M$ satisfies $s \o a = 0$, there is $c : L \to P$ s.t.\ 
$a = b \o c$. 
Now look at the distinguished triangle 
\[ K \xar{t} L \xar{c} P \xar{} K[1] \]
built on $c$. We know that 
$c \o t = 0$; hence 
$a \o t = b \o c \o t = 0$. 
But ($s \in \cat{S}$) $\Rightarrow$ ($\mrm{H}(P[i]) = 0$ for all $i$)
$\Rightarrow$ ($t \in \cat{S}$).

The left versions of (D1) and (D2) are proved the same way. 
\end{proof}


\begin{thm} \label{thm:105}
Let $\cat{S}$ be the denominator set in $\cat{K}$ associated to a cohomological
functor, as in Proposition \tup{\ref{prop:106}}, 
and let  $(\cat{K}_{\cat{S}}, Q)$ be the Ore localization.
The additive category $\cat{K}_{\cat{S}}$ has a unique triangulated structure
such that these two properties hold:

\begin{enumerate}
\rmitem{i} The functor  $Q : \cat{K} \to \cat{K}_{\cat{S}}$ is
triangulated. 

\rmitem{ii} Suppose $\cat{E}$ is another triangulated category, and 
$F : \cat{K} \to \cat{E}$ is a triangulated functor s.t.\
$F(s)$ is invertible for every $s \in \cat{S}$. Let 
$F_{\cat{S}} : \cat{K}_{\cat{S}} \to \cat{E}$ be the localization of $F$. Then 
$F_{\cat{S}}$ is a triangulated functor. 
\end{enumerate}
\end{thm}

\begin{proof} 
Step 1. 
We define the translation functor $T_{\cat{S}}$ on
$\cat{K}_{\cat{S}}$. For objects we take 
$T_{\cat{S}}(M) := T(M)$ of course. And for a morphism 
$q \in \cat{K}_{\cat{S}}$ we choose a presentation 
$q = Q(a) \o Q(s)^{-1}$, and define
$T_{\cat{S}}(q) := Q(T(a)) \o Q((T(s))^{-1}$.
We must prove independence of presentation; but this is standard. 

\medskip \noindent
Step 2.
The distinguished triangles in $\cat{K}_{\cat{S}}$ are defined to be those
triangles that are isomorphic to the images under $Q$ of distinguished triangles
in $\cat{K}$. 
Let us verify the axioms of triangulated category. 

\medskip \noindent
(TR1). 
It is trivial that every
triangle that's isomorphic to a distinguished triangle is distinguished; and
that the triangle 
\[ M \xar{1_M} M \to 0 \to M[1] \] 
is distinguished.

Suppose we are given $\al : L \to M$ in $\cat{K}_{\cat{S}}$. We have to build a
distinguished triangle on it. Choose a presentation $\al = Q(a) \o Q(s)^{-1}$.
Use (D1)$_{\tup{left}}$ to find $b \in \cat{K}$ and $t \in \cat{S}$ such that
$t \o a = b \o s$. Consider the solid commutative diagram below, where
the rows are  distinguished triangles  built on $a$  and $b$ respectively.
\[ \UseTips  \xymatrix @C=5ex @R=5ex { 
K
\ar[r]^{a}
\ar[d]_{s}
&
M
\ar[r]^{e}
\ar[d]_{t}
& 
N
\ar[r]^{c}
\ar@{-->}[d]_{u}
&
K[1]
\ar[d]_{s[1]}
\\
L
\ar[r]^{b}
&
\til{L}
\ar[r]^{}
& 
P
\ar[r]^{d}
&
L[1]
}  \]
By (TR3) there is a morphism $u$ that makes the whole diagram commutative.
Since $s, t \in \cat{S}$ and $H$ is a cohomological functor, it follows that 
$u \in \cat{S}$. 
We get a commutative diagram
\[ \UseTips  \xymatrix @C=10ex @R=7ex { 
K
\ar[r]^{Q(a)}
\ar[d]_{Q(s)}
&
M
\ar[r]^{Q(e)}
\ar[d]_{Q(1_M)}
& 
N
\ar[r]^{Q(c)}
\ar[d]_{Q(u)}
&
K[1]
\ar[d]_{Q(s)[1]}
\\
L
\ar[r]^{\al}
&
M
\ar[r]^{Q(u \o e)}
& 
P
\ar[r]^{Q(d)}
&
L[1]
}  \]
in $\cat{K}_{\cat{S}}$. The top row is a distinguished triangle, and the
vertical arrows are isomorphisms. So the bottom row is a distinguished
triangle, and it is built on $\al$. 

\medskip \noindent 
(TR2). Turning: this is trivial. 

\medskip \noindent 
(TR3). We are given the solid commutative diagram in $\cat{K}_{\cat{S}}$,
where the rows are distinguished triangles:
\begin{equation} \label{eqn:104}
\UseTips \xymatrix @C=5ex @R=5ex { 
L 
\ar[r]^{\al}
\ar[d]_{\phi}
&
M
\ar[r]^{\be}
\ar[d]_{\psi}
& 
N
\ar[r]^{\ga}
\ar@{-->}[d]_{\chi}
&
L[1]
\ar[d]_{\phi[1]}
\\
L' 
\ar[r]^{\al'}
&
M'
\ar[r]^{\be'}
& 
N'
\ar[r]^{\ga'}
&
L'[1] 
} 
\end{equation}
and we have to find $\chi$ to complete the diagram. 

By replacing the rows with isomorphic triangles, we can assume they come from
$\cat{K}$.
Thus we can replace (\ref{eqn:104}) with this diagram:
\begin{equation} \label{eqn:110}
\UseTips \xymatrix @C=7ex @R=7ex { 
L 
\ar[r]^{Q(\al)}
\ar[d]_{\phi}
&
M
\ar[r]^{Q(\be)}
\ar[d]_{\psi}
& 
N
\ar[r]^{Q(\ga)}
\ar@{-->}[d]_{\chi}
&
L[1]
\ar[d]_{\phi[1]}
\\
L' 
\ar[r]^{Q(\al')}
&
M'
\ar[r]^{Q(\be')}
& 
N'
\ar[r]^{Q(\ga')}
&
L'[1] 
} 
\end{equation}
Let us choose presentations 
$\phi = Q(a) \o Q(s)^{-1}$ and $\psi = Q(b) \o Q(t)^{-1}$. 
Then the solid diagram (\ref{eqn:110}) comes from applying $Q$ to the diagram 
\begin{equation} \label{eqn:103}
\UseTips \xymatrix @C=5ex @R=5ex { 
L 
\ar[r]^{\al}
&
M
\ar[r]^{\be}
& 
N
\ar[r]^{\ga}
&
L[1]
\\
\til{L} 
\ar[d]_{a}
\ar[u]^{s}
&
\til{M}
\ar[d]_{b}
\ar[u]^{t}
& 
&
\til{L}[1]
\ar[d]_{a[1]}
\ar[u]^{s[1]}
\\
L' 
\ar[r]^{\al'}
&
M'
\ar[r]^{\be'}
& 
N'
\ar[r]^{\ga'}
&
L'[1] 
} \end{equation}
in $\cat{K}$. Here the rows are distinguished triangles.

By (L3) we can find $c \in \cat{K}$ and $u \in \cat{S}$ s.t.\ 
\[ Q(t)^{-1} \o Q(\al) \o Q(s) = Q(c) \o Q(u)^{-1} . \]
This is the solid diagram:
\[ \UseTips \xymatrix @C=5ex @R=5ex { 
& &
L 
\ar[r]^{\al}
&
M
\\
\til{L}''
\ar@{-->}[r]^{u'}
&
\til{L}'
\ar[r]^{u}
\ar@(dr,dl)[rr]_(0.2){c}
&
\til{L} 
\ar[d]_(0.3){a}
\ar[u]^{s}
&
\til{M}
\ar[d]_{b}
\ar[u]^{t}
\\
& &
L' 
\ar[r]^{\al'}
&
M'
} \]
Thus 
\[ Q(\al \o s \o u) = Q(t \o c) . \]
Take $u' \in \cat{S}$ s.t.\ 
\[ (\al \o s \o u) \o u' = (t \o c) \o u' . \]
This is possible by (L4). We get
\[ \phi = Q(a) \o Q(s)^{-1} = Q(a \o u \o u') \o Q(s \o u \o u')^{-1} . \]
Thus, after substituting $\til{L} :=  \til{L}''$,
 $s := s \o u \o u'$,
$a := a \o u \o u'$
and $c := c \o u'$, we get a new diagram
\begin{equation} \label{eqn:105}
\UseTips \xymatrix @C=5ex @R=5ex { 
L 
\ar[r]^{\al}
&
M
\ar[r]^{\be}
& 
N
\ar[r]^{\ga}
&
L[1]
\\
\til{L} 
\ar[r]^{c}
\ar[d]_{a}
\ar[u]^{s}
&
\til{M}
\ar[d]_{b}
\ar[u]^{t}
& 
&
\til{L}[1]
\ar[d]_{a[1]}
\ar[u]^{s[1]}
\\
L' 
\ar[r]^{\al'}
&
M'
\ar[r]^{\be'}
& 
N'
\ar[r]^{\ga'}
&
L'[1] 
} 
\end{equation}
in $\cat{K}$. Here the top left square is commutative; but maybe the bottom
left square is not commutative. 

When we apply $Q$ to the diagram (\ref{eqn:105}), the  whole diagram, including
the bottom left square, becomes commutative, since (\ref{eqn:110}) is
commutative. Again using condition (L4), there is 
$v \in \cat{S}$ s.t.\ 
\[ (\al' \o a) \o v = (b \o c) \o v . \]
In a diagram:
\[ \UseTips \xymatrix @C=5ex @R=5ex { 
& 
L 
\ar[r]^{\al}
&
M
\\
\til{L}'
\ar[r]^{v}
&
\til{L} 
\ar[d]_{a}
\ar[u]^{s}
\ar[r]^{c}
&
\til{M}
\ar[d]_{b}
\ar[u]^{t}
\\
& 
L' 
\ar[r]^{\al'}
&
M'
} \]
Performing the replacements
$\til{L} :=  \til{L}'$,
$s := s \o v$,
$c := c \o v$ and 
$a := a \o v$ we now have a commutative square also at the bottom
left of (\ref{eqn:105}). Since $\ga \o \be = 0 = \ga' \o \be'$, 
in fact the whole diagram (\ref{eqn:105}) in $\cat{K}$ is now commutative.

Now by (TR1) we can embed $c$ in a distinguished
triangle. We get the solid diagram
\begin{equation} \label{eqn:106}
\UseTips \xymatrix @C=5ex @R=5ex { 
L 
\ar[r]^{\al}
&
M
\ar[r]^{\be}
& 
N
\ar[r]^{\ga}
&
L[1]
\\
\til{L} 
\ar[r]^{c}
\ar[d]_{a}
\ar[u]^{s}
&
\til{M}
\ar[r]^{\til{\be}}
\ar[d]_{b}
\ar[u]^{t}
& 
\til{N}
\ar[r]^{\til{\ga}}
\ar@{-->}[d]_{d}
\ar@{-->}[u]^{w}
&
\til{L}[1]
\ar[d]_{a[1]}
\ar[u]^{s[1]}
\\
L' 
\ar[r]^{\al'}
&
M'
\ar[r]^{\be'}
& 
N'
\ar[r]^{\ga'}
&
L'[1] 
} 
\end{equation}
in $\cat{K}$. The rows are distinguished triangles. 
Since $\til{\ga} \o \til{\be} = 0$, the solid diagram is commutative.  
By (TR3) there are morphisms $w$ and $d$ that make the whole diagram
commutative. Now the morphism $w \in \cat{S}$ by the usual long exact sequence
argument. The morphism 
\[ \chi := Q(d) \o Q(w)^{-1} :  N \to N' \]
solves the problem.

\medskip \noindent 
(TR4). As always, we neglect this axiom. 

\medskip \noindent
Step 3. Now $\cat{K}_{\cat{S}}$ is a triangulated category. 
Conditions (i)-(ii) are clear. They imply the uniqueness of the
triangulated structure that we imposed on $\cat{K}_{\cat{S}}$. 
Indeed, since $Q$ must be a triangulated functor, we can't have less
distinguished triangles than those we declared. We can't have more
distinguished triangles, because of condition (ii). 
\end{proof}

\subsection{The derived category}

\begin{prop}
Let $\cat{M}$ be an abelian category. The functor 
\[ \mrm{H}^0 : \dcat{K}(\cat{M}) \to \cat{M} \]
is cohomological. 
\end{prop}

\begin{proof}
Exercise. (Hint: enough to check for standard triangles.)
\end{proof}

The set $\dcat{S}(\cat{M})$ of quasi-isomorphisms in $\dcat{K}(\cat{M})$
satisfies 
\[ \dcat{S}(\cat{M}) = 
\{ s \in \dcat{K}(\cat{M}) \mid \mrm{H}^i(s) 
\tup{ is an isomorphism for all } i \}  . \]
Therefore Theorem \ref{thm:105} applies, and the next definition makes sense. 

\begin{dfn}
Let $\cat{M}$ be a $\K$-linear abelian category.  The {\em derived category} of
$\cat{M}$ is the $\K$-linear triangulated category 
$\dcat{D}(\cat{M}) := \dcat{K}(\cat{M})_{\dcat{S}(\cat{M})}$.
\end{dfn}



\cleardoublepage
\section{Full Subcategories of the Derived Category}

\subsection{General facts}
In this section $\cat{M}$ is an abelian category. Recall that 
$\dcat{C}(\cat{M})$ is the category of complexes in $\cat{M}$, 
$\dcat{K}(\cat{M})$ is the homotopy
category of complexes, $\dcat{S}(\cat{M})$ is the set of
quasi-isomorphisms in $\dcat{K}(\cat{M})$, and 
$\dcat{D}(\cat{M}) = \dcat{K}(\cat{M})_{\dcat{S}(\cat{M})}$ is the derived
category.

Let $\cat{L}$ be a full triangulated subcategory of $\dcat{K}(\cat{M})$,
and let $\cat{S} := \cat{L} \cap \dcat{S}(\cat{M})$. 
Note that $\cat{S}$  satisfies 
\[ \cat{S} = \{ s \in  \cat{L}
\mid \mrm{H}^i(s) \tup{ is an isomorphism for all } i \}  . \]
Hence Theorem \ref{thm:105} applies, and the Ore localization 
$(\cat{L}_{\cat{S}}, Q)$ exists. By the universal properties, there is a unique
triangulated functor 
$\cat{L}_{\cat{S}} \to \dcat{D}(\cat{M})$
that extends the inclusion $\cat{L} \to \dcat{K}(\cat{M})$.
We are interested in sufficient conditions for the functor 
$\cat{L}_{\cat{S}} \to \dcat{D}(\cat{M})$ to be fully faithful.

\begin{rem}
I could not find a counterexample -- can anyone think up / locate one?
\end{rem}

\begin{lem} \label{lem:101}
Let $\cat{L} \subset \til{\cat{L}}$ be full triangulated subcategories of
$\dcat{K}(\cat{M})$, let $\cat{S} := \cat{L} \cap \dcat{S}(\cat{M})$,
and let $\til{\cat{S}} := \til{\cat{L}} \cap \dcat{S}(\cat{M})$. 
Assume either of these conditions holds:
\begin{itemize}
\item[(r)] Given any morphism $s : M \to L$ in $\til{\cat{S}}$ such that 
$L \in \opn{Ob}(\cat{L})$, there 
exists a morphism $t : L' \to M$  in $\til{\cat{S}}$ such that 
$L' \in \opn{Ob}(\cat{L})$.

\item[(l)] The same, but with arrows reversed.
\end{itemize}
Then the functor $\cat{L}_{\cat{S}} \to \til{\cat{L}}_{\til{\cat{S}}}$ is fully
faithful.
\end{lem}

\begin{proof}
We will prove the case (r); the other case is done the same way. 
Let $L_1, L_2 \in \opn{Ob}(\cat{L})$, and let $q : L_1 \to L_2$ be a
morphism in $\til{\cat{L}}_{\til{\cat{S}}}$. Choose a presentation 
$q = Q(a) \o Q(s)^{-1}$ with
$s : M \to L_1$ a morphism in $\til{\cat{S}}$
and $a : M \to L_2$  a morphism in $\til{\cat{L}}$. 
\[ \UseTips  \xymatrix @C=5ex @R=5ex  {
&
L'
\ar@{-->}[d]_{t}
\\
& M 
\ar[dl]_{s}
\ar[dr]^{a}
\\
L_1
\ar[rr]^{q}
& &
L_2
} \]
By condition (r) we can find a morphism $t : L' \to M$ in $\til{\cat{S}}$ with 
$L' \in \opn{Ob}(\cat{L})$. It follows that 
$q = Q(a \o t) \o Q(s \o t)^{-1}$; so $q$ is in the image of the functor 
$F : \cat{L}_{\cat{S}} \to \til{\cat{L}}_{\til{\cat{S}}}$. Thus $F$ is full. 

Now let $q : L_1 \to L_2$ be a morphism in $\cat{L}_{\cat{S}}$
s.t.\ $F(q) = 0$.  Choose a presentation 
$q = Q(a) \o Q(s)^{-1}$ with
$s : L' \to L_1$ a morphism in $\cat{S}$
and $a : L' \to L_2$  a morphism in $\cat{L}$. 
\[ \UseTips  \xymatrix @C=5ex @R=5ex  {
&
L''
\ar@{-->}[d]_{t'}
\\
& M 
\ar@{-->}[d]_{t}
\\
&
L'
\ar[dl]_{s}
\ar[dr]^{a}
\\
L_1
\ar[rr]^{q}
& &
L_2
} \]
Because $F(q) = 0$, and using Lemma \ref{lem:3}, there is a morphism 
$t : M \to L'$ in $\til{\cat{L}}$ 
 s.t.\ $a \o t = 0$ and $s \o t \in \til{\cat{S}}$.
Note that $t \in \til{\cat{S}}$. 
By condition (r), 
applied to $t : M \to L'$, there is $t' : L'' \to M$ in $\til{\cat{S}}$
s.t.\ $L'' \in \opn{Ob}(\cat{L})$. Then 
$q = Q(a \o t \o t') \o Q(s \o t \o t')^{-1} = 0$. 
This proves that $F$ is faithful.
\end{proof}

\subsection{Bounded complexes}
A graded object $M = \{ M^i \}_{i \in \Z}$ of $\cat{M}$ is said to be {\em
bounded above} if the set 
$\{ i \mid M^i \neq 0 \}$ is bounded above. 
Likewise we define {\em bounded below} and {\em bounded} graded objects. 

\begin{dfn}
We define $\dcat{C}^-(\cat{M})$, $\dcat{C}^+(\cat{M})$ and 
$\dcat{C}^{\mrm{b}}(\cat{M})$ to be  full subcategories of 
$\dcat{C}(\cat{M})$ consisting of bounded above, bounded below and bounded
complexes respectively. 

Likewise we define $\dcat{K}^-(\cat{M})$, $\dcat{K}^+(\cat{M})$ and 
$\dcat{K}^{\mrm{b}}(\cat{M})$ to be the corresponding full subcategories of 
$\dcat{K}(\cat{M})$. 
\end{dfn}

Of course 
$\dcat{K}^{\mrm{b}}(\cat{M}) = \dcat{K}^-(\cat{M})
\cap \dcat{K}^+(\cat{M})$.
For $\star \in \{ -, +, \mrm{b} \}$, the category 
$\dcat{K}^{\star}(\cat{M})$ is a full triangulated subcategory of 
$\dcat{K}^{\star}(\cat{M})$; this is because the mapping cone preserves the
various boundedness conditions. 
Let 
\[ \dcat{S}^{\star}(\cat{M}) := \dcat{K}^{\star}(\cat{M}) \cap
\dcat{S}(\cat{M}) , \]
the set of quasi-isomorphisms in $\dcat{K}^{\star}(\cat{M})$. 

\begin{dfn}
For $\star \in \{ -, +, \mrm{b} \}$ let 
\[  \dcat{D}^{\star}(\cat{M}) := 
\dcat{K}^{\star}(\cat{M})_{\dcat{S}^{\star}(\cat{M})} , \]
the Ore localization of $\dcat{K}^{\star}(\cat{M})$ with respect to 
$\dcat{S}^{\star}(\cat{M})$.
\end{dfn}

\begin{prop}
For $\star \in \{ -, +, \mrm{b} \}$ the canonical functor 
$\dcat{D}^{\star}(\cat{M}) \to \dcat{D}(\cat{M})$ 
is fully faithful.
\end{prop}

\begin{proof}
Let $s : M \to L$ be a quasi-isomorphism with $L \in \dcat{K}^-(\cat{M})$.
Say $L$ is concentrated in degrees $\leq i$.
Then $\opn{H}^j(M) = \opn{H}^j(L) = 0$ for all $j > i$. 
Consider the {\em smart truncation} of $M$ at $i$:
\begin{equation} \label{eqn:112}
\opn{smt}^{\leq i}(M) := 
\bigl( \cdots \to M^{i-2} \xar{\d} M^{i-1} \xar{\d} \opn{Z}^i(M) \to 0 \to
\cdots \bigr)
\end{equation}
where 
$\opn{Z}^i(M) := \opn{Ker}(\d : M^i \to M^{i+1})$,
the object of $i$-cocycles, is in degree $i$. 
Then $\opn{smt}^{\leq i}(M)$ is a subcomplex of $M$, 
$\opn{smt}^{\leq i}(M) \in \dcat{K}^-(\cat{M})$, 
and the monomorphism
$t : \opn{smt}^{\leq i}(M) \to M$ is a quasi-isomorphism. 
According to Lemma \ref{lem:101}, with 
$\cat{L} := \dcat{K}^-(\cat{M})$,
$\til{\cat{L}} := \dcat{K}(\cat{M})$ and 
with condition (r), we see that 
$\dcat{D}^{-}(\cat{M}) \to \dcat{D}(\cat{M})$ 
is fully faithful.

Next let $s : L \to M$ be a quasi-isomorphism with $L \in \dcat{K}^+(\cat{M})$.
Say $L$ is concentrated in degrees $\geq i$. 
Then $\opn{H}^j(M) = \opn{H}^j(L) = 0$ for all $j < i$. 
Consider the other smart truncation of $M$ at $i$:
\begin{equation} \label{eqn:111}
\opn{smt}^{\geq i}(M) := 
\bigl( \cdots \to 0 \to \opn{Y}^{i}(M) \xar{\d} M^{i+1} \xar{\d} M^{i+2}
\to \cdots \bigr) 
\end{equation}
where 
\begin{equation} \label{eqn:130}
\opn{Y}^{i}(M) := \opn{Coker}(\d : M^{i-1} \to M^{i})
\end{equation}
is in degree $i$. 
Then $\opn{smt}^{\geq i}(M)$ is a quotient complex of $M$, 
$\opn{smt}^{\geq i}(M) \in \dcat{K}^+(\cat{M})$, 
and the epimorphism
$t : M \to \opn{smt}^{\geq i}(M)$ is a quasi-isomorphism. 
According to Lemma \ref{lem:101}, with condition (l), we see that 
$\dcat{D}^{-}(\cat{M}) \to \dcat{D}(\cat{M})$ 
is fully faithful.

Finally we note that 
\[ \dcat{K}^{\mrm{b}}(\cat{M}) =  \dcat{K}^{-}(\cat{M}) \cap
\dcat{K}^{+}(\cat{M}) . \]
This implies, as in the proof for 
$\dcat{D}^{-}(\cat{M}) \to \dcat{D}(\cat{M})$, that
$\dcat{D}^{\mrm{b}}(\cat{M}) \to \dcat{D}^{+}(\cat{M})$
is fully faithful. Hence 
$\dcat{D}^{\mrm{b}}(\cat{M}) \to \dcat{D}(\cat{M})$
is fully faithful. 
\end{proof}

\newpage
\subsection{Thick subcategories of \texorpdfstring{$\cat{M}$}{M}}

Let $\cat{M}$ be an abelian category. A {\em thick abelian subcategory} of
$\cat{M}$ is a full abelian subcategory $\cat{N}$ that is closed under
extensions. Namely if 
\[ 0 \to M' \to M \to M'' \to 0 \]
is a short exact sequence in $\cat{M}$ with $M', M'' \in \cat{N}$, then 
$M \in \cat{N}$ too. 

Let $\dcat{D}_{\cat{N}}(\cat{M})$ be the full subcategory of $\dcat{D}(\cat{M})$
consisting of complexes $M$ such that 
$\opn{H}^i(M) \in \cat{N}$ for every $i$. 

Full triangulated subcategories were defined in Definition \ref{dfn:22}.

\begin{prop}
If $\cat{N}$ is a thick abelian subcategory of $\cat{M}$ then
$\dcat{D}_{\cat{N}}(\cat{M})$ is a triangulated subcategory of 
$\dcat{D}(\cat{M})$.
\end{prop}

\begin{proof}
Clearly $\dcat{D}_{\cat{N}}(\cat{M})$ is closed under translations. 
Now suppose 
\[ M' \to M \to M'' \to M[1] \]
is a distinguished triangle in $\dcat{D}(\cat{M})$ such that two of its
vertices is in $\dcat{D}_{\cat{N}}(\cat{M})$; we have to show that the third
vertex is also in $\dcat{D}_{\cat{N}}(\cat{M})$.
Since we can turn the triangle, we may as well assume that 
$M', M''\in \dcat{D}_{\cat{N}}(\cat{M})$. Consider the  exact sequence 
\[  \opn{H}^{i-1}(M'') \to \opn{H}^i(M') \to \opn{H}^i(M) \to
\opn{H}^i(M'') \to \opn{H}^{i+1}(M') . \]
The four outer objects belong to $\cat{N}$. 
Since $\cat{N}$ is a thick abelian subcategory of $\cat{M}$ it follows that 
$\opn{H}^i(M) \in \cat{N}$.
\end{proof}

\begin{exa}
Let $A$ be a noetherian commutative ring. 
The category $\cat{Mod}_{\mrm{f}}\, A$ of finitely generated
modules is a thick abelian subcategory of $\cat{Mod}\, A$.
\end{exa}

\begin{exa}
Consider $\cat{Mod}\, \Z = \cat{Ab}$. 
As above we have the thick abelian subcategory 
$\cat{Ab}_{\mrm{fg}} = \cat{Mod}_{\mrm{f}}\, \Z$ of finitely generated
abelian groups.
There is also the thick abelian subcategory 
$\cat{Ab}_{\mrm{tors}}$ of torsion abelian groups (every element has a finite
order). The intersection of $\cat{Ab}_{\mrm{tors}}$ and 
$\cat{Ab}_{\mrm{fg}}$ is the category $\cat{Ab}_{\mrm{fin}}$ of finite abelian
groups. This is also thick. 
\end{exa}

\begin{exa}
Let $X$ be a noetherian scheme (or an algebraic variety over an algebraically
closed field of you prefer). Consider the abelian category 
$\cat{Mod}\, \mcal{O}_X$ of $\mcal{O}_X$-modules. In it there the thick abelian
subcategory $\cat{QCoh}\, \mcal{O}_X$ of quasi-coherent sheaves, and in that 
there the thick abelian
subcategory $\cat{Coh}\, \mcal{O}_X$ of coherent sheaves.
\end{exa}


For a left noetherian ring $A$ we write 
\[ \dcat{D}_{\mrm{f}}(\cat{Mod}\, A) :=
\dcat{D}_{\cat{Mod}_{\mrm{f}}\, A}(\cat{Mod}\, A) . \]

\begin{prop} \label{prop:200}
Let $A$ be a left noetherian ring and $\star \in \{ -, \mrm{b} \}$. Then the
canonical functor 
\[ \dcat{D}^{\star}(\cat{Mod}_{\mrm{f}}\, A) \to 
\dcat{D}^{\star}_{\mrm{f}}(\cat{Mod}\, A) \]
is fully faithful.
\end{prop}

\begin{proof}
A bit later (Theorem \ref{thm:125}) I will prove that any 
$M \in \dcat{D}^{-}_{\mrm{f}}(\cat{Mod}\, A)$
admits a free resolution $P \to M$, where $P$ is a
bounded above complex of finitely generated free modules. 
Thus we get a quasi-isomorphism $P \to M$ with 
$P \in \dcat{D}^{-}(\cat{Mod}_{\mrm{f}}\, A)$.
By Lemma \ref{lem:101} with condition (r) we conclude that 
\[ \dcat{D}^{-}(\cat{Mod}_{\mrm{f}}\, A) \to 
\dcat{D}^{-}_{\mrm{f}}(\cat{Mod}\, A) \]
is fully faithful.

Now if moreover $M \in \dcat{D}^{\mrm{b}}_{\mrm{f}}(\cat{Mod}\, A)$,
then take $i$ s.t.\ $M$ is concentrated in degrees $\geq i$. 
Then $\opn{H}^i(P) = 0$ for $j \leq i$, and we get a quasi-isomorphism 
$\opn{smt}^{\geq i}(P) \to M$ (see (\ref{eqn:111})). 
But $\opn{smt}^{\geq i}(P) \in \dcat{D}^{\mrm{b}}(\cat{Mod}_{\mrm{f}}\, A)$.
\end{proof}

\begin{rem} \label{rem:200}
The proposition holds more generally for any {\em locally noetherian abelian
category} $\cat{M}$. I need to explain this. The abelian category $\cat{M}$
is a {\em Grothendieck abelian category} if it has infinite direct limits,
and they are exact. An object $M \in \cat{M}$ is called a {\em noetherian
object} if it satisfies the ascending chain condition on subobjects. 
The set of noetherian objects is denoted by $\cat{M}_{\mrm{f}}$. 
This is a thick abelian subcategory. We say that $\cat{M}$ is locally
noetherian if it is generated by its noetherian objects; namely for any 
$M \in \cat{M}$ there is an epimorphism 
$\bigoplus_{i \in I} M_i \to M$ with $M_i \in \cat{M}_{\mrm{f}}$. 

The categories $\cat{Mod}\, A$, for $A$ a left noetherian ring, and
$\cat{QCoh}\, \mcal{O}_X$ for a noetherian scheme $X$, are locally noetherian.
\end{rem}

\subsection{The embedding of \texorpdfstring{$\cat{M}$}{M} in 
\texorpdfstring{$\dcat{D}(\cat{M})$}{M}}

For $M, N \in \cat{M}$ there is no difference between 
$\opn{Hom}_{\cat{M}}(M, N)$, $\opn{Hom}_{\dcat{C}(\cat{M})}(M, N)$ and 
$\opn{Hom}_{\dcat{K}(\cat{M})}(M, N)$. Thus the canonical functors 
$\cat{M} \to \dcat{C}(\cat{M})$ and 
$\cat{M} \to \dcat{K}(\cat{M})$ are fully faithful. 
The same is true for $\dcat{D}(\cat{M})$, but this requires a proof.

Let $\dcat{D}(\cat{M})^0$ be the full subcategory of 
$\dcat{D}(\cat{M})$ consisting of complexes whose cohomology is 
concentrated in degree $0$. This is an additive subcategory of 
$\dcat{D}(\cat{M})$.

\begin{prop} \label{prop:121}
The canonical functor $\cat{M} \to \dcat{D}(\cat{M})^0$ 
is an equivalence.
\end{prop}

\begin{proof}
There is a functor $\opn{H}^0 : \dcat{D}(\cat{M}) \to \cat{M}$
satisfying $\opn{H}^0 \o\, Q = \bsym{1}_{\cat{M}}$. This implies that $Q$ is
faithful.

Now take any $M, N \in\cat{M}$ and a morphism $q : M \to N$ in 
$\dcat{D}(\cat{M})$. So 
$q = Q(a) \o Q(s)^{-1}$ for some
morphisms $a : L \to N$ and $s : L \to M$ in $\dcat{D}(\cat{M})$, with $s$
a quasi-isomorphism. Let $L' := \opn{smt}^{\leq 0}(L)$, as in (\ref{eqn:112});
so there is a
quasi-isomorphism $u : L' \to L$. Writing 
$a' := a \o u$ and $s' := s \o u$, we see that $s'$ is a quasi-isomorphism, and
$q = Q(a') \o Q(s')^{-1}$. 
Next let $L'' := \opn{smt}^{\geq 0}(L')$, as in (\ref{eqn:130}); so there are 
morphisms $v : L' \to L''$, $a'' : L'' \to N$ and $s'' : L'' \to M$ s.t.\ 
$v$ is a quasi-isomorphism, $a' = a'' \o v$ and $s' = s'' \o v$. 
But $L'' \in \cat{M}$, so $s'' : L'' \to M$ is an isomorphism in 
$\cat{M}$. (To be precise, $s''$ is an isomorphism in the image of $\cat{M}$
inside $\dcat{K}(\cat{M})$, but we know that
$\cat{M} \to \dcat{K}(\cat{M})$ is fully faithful.) We get 
$q = Q(a'') \o Q(s'')^{-1} = Q(a'' \o (s'')^{-1})$.

Finally we have to prove that any $L \in \dcat{D}(\cat{M})^0$ is isomorphic, in 
$\dcat{D}(\cat{M})$, to a complex $L''$ that's concentrated in degree $0$. But
we already showed it in the previous paragraph.
\end{proof}

\begin{rem}
Let's write $\dcat{D}(\cat{M})^{\geq 0}$ 
for the full subcategory of $\dcat{D}(\cat{M})$ on the objects $M$ s.t.\ 
$\opn{H}^i(M) = 0$ for $i < 0$. Like we define 
$\dcat{D}(\cat{M})^{\leq 0}$.
The pair 
$\bigl( \dcat{D}(\cat{M})^{\geq 0}, \dcat{D}(\cat{M})^{\leq 0} \bigr)$
is called the {\em standard t-structure} on $\dcat{D}(\cat{M})$. The abelian
category 
$\dcat{D}(\cat{M})^{0} = \dcat{D}(\cat{M})^{\geq 0} \cap 
\dcat{D}(\cat{M})^{\leq 0}$ 
is called the {\em heart} of this t-structure. 
\end{rem}

This last result doesn't belong here properly -- it should be in Section 
\ref{sec:triangulated}.

\begin{prop} \label{prop:120}
Let $\cat{K}$ be a triangulated category, and let 
\[ M \xar{\phi} N \to L \to M[1] \]
be a distinguished triangle in it. Then $L \cong 0$ iff $\phi$ is
an isomorphism.
\end{prop}

\begin{proof}
Assume $L \cong 0$. 
Consider the commutative diagram 
\begin{equation} \label{eqn:120}
\UseTips \xymatrix  @C=12ex @R=6ex { 
\opn{Hom}_{\cat{K}}(M, M)
\ar[r]^{\opn{Hom}_{\cat{K}}(1_M, \phi)}
&
\opn{Hom}_{\cat{K}}(M, N)
\\
\opn{Hom}_{\cat{K}}(N, M)
\ar[r]^{\opn{Hom}_{\cat{K}}(1_N, \phi)}
\ar[u]^{\opn{Hom}_{\cat{K}}(\phi, 1_M)}
&
\opn{Hom}_{\cat{K}}(N, N)
\ar[u]_{\opn{Hom}_{\cat{K}}(\phi, 1_N)}
} 
\end{equation}
in $\cat{Ab}$. 

Applying the cohomological functor $\opn{Hom}_{\cat{K}}(M, -)$ to the given
triangle we get an exact sequence 
\[ \opn{Hom}_{\cat{K}}(M, L[-1]) \to \opn{Hom}_{\cat{K}}(M, M)
\xar{\opn{Hom}_{\cat{K}}(1_M, \phi)}
\opn{Hom}_{\cat{K}}(M, N) \to \opn{Hom}_{\cat{K}}(M, L) \]
in $\cat{Ab}$. Since
\[ \opn{Hom}_{\cat{K}}(M, L[-1]) = 0 = \opn{Hom}_{\cat{K}}(M, L) \]
we see that $\opn{Hom}_{\cat{K}}(1_M, \phi)$ is an isomorphism. 
Likewise all other morphisms in the diagram (\ref{eqn:120}) are shown to be
isomorphisms. Now $\phi$ in the top right corner 
corresponds to unique morphisms in the other three positions; and they are 
$1_M$ in the top left, $1_N$ in the bottom right, and some morphism
$\psi : N \to M$ at the bottom left. So 
$\phi \o \psi = 1_N$ and $\psi \o \phi = 1_M$.

Conversely, if $\phi$ is an isomorphism, then from the exact sequence 
\[ \opn{Hom}_{\cat{K}}(L, M) \xar{\cong} \opn{Hom}_{\cat{K}}(L, N)
\to \opn{Hom}_{\cat{K}}(L, L) \to \opn{Hom}_{\cat{K}}(L, M[1]) 
\xar{\cong} \opn{Hom}_{\cat{K}}(L, N[1])  \]
we see that $\opn{Hom}_{\cat{K}}(L, L) =  0$. So $1_L = 0$, and hence 
$L \cong 0$.
\end{proof}

\cleardoublepage
\section{Resolutions}

In this section I will prove that under certain conditions, K-injective and
K-projective resolutions exist. See Definitions \ref{dfn:100} and \ref{dfn:101}.
We will see some more important properties of these resolutions. 

\subsection{More on K-injectives}
Let $\cat{M}$ be an abelian category.

Let's recall what K-injectives are (Definition \ref{dfn:100}).
A complex $I \in \dcat{K}(\cat{M})$ is called {\em K-injective} if for
every acyclic $N \in \dcat{K}(\cat{M})$, the complex 
$\opn{Hom}_{\cat{M}}(N, I)$ is also acyclic.

We use the symbol ``\textvisiblespace'' to denote the
empty string (!?).

\begin{dfn} Let $\star \in \{ \mrm{b}, -, + , \tup{\textvisiblespace} \}$.
\begin{enumerate}
\item We denote by $\dcat{K}^{\star}(\cat{M})_{\tup{K-inj}}$  the full
subcategory of 
$\dcat{K}^{\star}(\cat{M})$ on the K-injective complexes. 

\item We say that $\dcat{K}^{\star}(\cat{M})$ has {\em enough K-injectives} if
every $M \in \dcat{K}^{\star}(\cat{M})$ admits a quasi-isomorphism 
$M \to I$ with $I \in \dcat{K}^{\star}(\cat{M})_{\tup{K-inj}}$.
\end{enumerate}
\end{dfn}

The category $\dcat{K}^{\star}(\cat{M})_{\tup{K-inj}}$  usually will be
interesting only for 
$\star \in \{ +, \text{\textvisiblespace} \}$. 
The next lemma says that being K-injective is intrinsic to 
$\dcat{K}^{\star}(\cat{M})$.

\begin{lem}
Let $I \in \dcat{K}^{\star}(\cat{M})$, where 
$\star \in \{ \mrm{b}, -, + \}$.
The following conditions are equivalent:
\begin{enumerate}
\rmitem{i} $I$ is K-injective.

\rmitem{ii} The complex $\opn{Hom}_{\cat{M}}(M, I)$ is acyclic for every
acyclic $M \in \dcat{K}^{\star}(\cat{M})$. 
\end{enumerate}
\end{lem}

\begin{proof}
The implication (i) $\Rightarrow$ (ii) is trivial. We will prove the reverse
direction only for $\star = +$; the rest is similar.
Say $M$ is an unbounded acyclic complex, and $I$ is concentrated in degrees 
$\geq k_0$. Take any $k \in \Z$. We want to prove that
$\opn{H}^k \bigl( \opn{Hom}_{\cat{M}}(M, I) \bigr) = 0$.
Now 
$\opn{H}^k \bigl( \opn{Hom}_{\cat{M}}(M, I) \bigr)$
only depends on the quotient complex 
$M' := \opn{smt}^{\geq k_0 - k - 2}(M)$ of $M$;
see (\ref{eqn:111}). Namely 
\[ \opn{H}^k \bigl( \opn{Hom}_{\cat{M}}(M, I) \bigr) \cong 
\opn{H}^k \bigl( \opn{Hom}_{\cat{M}}(M', I) \bigr) . \]
But $M'$ is a bounded below acyclic complex, so by
assumption 
$\opn{H}^k \bigl( \opn{Hom}_{\cat{M}}(M', I) \bigr) = 0$.
\end{proof}

\begin{lem} \label{lem:120}
A complex $I \in \dcat{K}^{\star}(\cat{M})$ is K-injective iff 
$\opn{Hom}_{\dcat{K}(\cat{M})}(N, I) = 0$ for every acyclic complex 
$N \in \dcat{K}^{\star}(\cat{M})$.
\end{lem}

\begin{proof}
This is because 
\[ \opn{Hom}_{\dcat{K}(\cat{M})}(N[-p], I) \cong
\opn{H}^0 \bigl( \opn{Hom}_{\cat{M}}(N[-p], I) \bigr) \cong 
\opn{H}^p \bigl( \opn{Hom}_{\cat{M}}(N, I) \bigr) , \]
and $N$ is acyclic iff $N[-p]$ is acyclic.
\end{proof}

\begin{prop}  \label{prop:130}
$\dcat{K}^{\star}(\cat{M})_{\tup{K-inj}}$ is a triangulated subcategory of 
$\dcat{K}^{\star}(\cat{M})$.
\end{prop}

\begin{proof}
It is clear that $\dcat{K}^{\star}(\cat{M})_{\tup{K-inj}}$ is closed under
shifts. 

Suppose 
\[ I \to J \to K \to I[1] \]
is a distinguished triangle in $\dcat{K}^{\star}(\cat{M})$, with 
$I, J$ being K-injective complexes. We have to show that $K$ is also
K-injective. Take any acyclic complex $N \in \dcat{K}^{\star}(\cat{M})$. There
are exact sequences
\[ \opn{Hom}_{\dcat{K}(\cat{M})}(N, J[p]) \to 
\opn{Hom}_{\dcat{K}(\cat{M})}(N, K[p])
\to \opn{Hom}_{\dcat{K}(\cat{M})}(N, I[p+1])  \]
in $\cat{Ab}$ for all $p$. 
By Lemma \ref{lem:120} we have
\[ \opn{Hom}_{\dcat{K}(\cat{M})}(N, J[p]) = 0 = 
\opn{Hom}_{\dcat{K}(\cat{M})}(N, I[p+1]) , \]
and therefore 
$\opn{Hom}_{\dcat{K}(\cat{M})}(N, K[p]) = 0$. 
Thus $K$ is K-injective.
\end{proof}

\begin{lem} \label{lem:121}
Let $\phi : I \to M$ be a quasi-isomorphism in $\dcat{K}^{}(\cat{M})$, and
assume that $I$ is K-injective. Then $\phi$ is a split monomorphism: there
exists a morphism $\psi : M \to I$ in $\dcat{K}^{}(\cat{M})$ s.t.\ 
$\psi \o \phi = 1_I$.
\end{lem}

\begin{proof}
Consider a distinguished triangle built on $\phi$:
\[ M \xar{\phi} I \to L \to M[1] . \]
Since $\phi$ is a quasi-isomorphism, the complex $L$ is acyclic. But $I$ is
K-injective, and therefore  
$\opn{Hom}_{\dcat{K}(\cat{M})}(L[p], I) = 0$ for all $p$.
Look at the exact sequence of abelian groups:
\[ \opn{Hom}_{\dcat{K}(\cat{M})}(L, I) \to 
\opn{Hom}_{\dcat{K}(\cat{M})}(I, I) 
\xar{\opn{Hom}_{\cat{K}}(\phi, 1_I)}
\opn{Hom}_{\dcat{K}(\cat{M})}(M, I)
\to  \opn{Hom}_{\dcat{K}(\cat{M})}(L[-1], I) .  \]
The end terms are zero, implying that 
$\opn{Hom}_{\cat{K}}(\phi, 1_I)$ is bijective. 
The morphism $\psi \in \opn{Hom}_{\dcat{K}(\cat{M})}(M, I)$
corresponding to $1_I$ is what we want.
\end{proof}
%
%
%
%


\begin{lem} \label{lem:122}
Let $I \in \dcat{K}^{\star}(\cat{M})_{\tup{K-inj}}$ and 
$M \in \dcat{K}^{\star}(\cat{M})$. Then 
\[ Q : \opn{Hom}_{\dcat{K}^{\star}(\cat{M})}(M, I) \to 
\opn{Hom}_{\dcat{D}^{\star}(\cat{M})}(M, I) \]
is bijective.
\end{lem}

\begin{proof}
Recall that $Q : \dcat{K}^{\star}(\cat{M}) \to \dcat{D}^{\star}(\cat{M})$ is a
left Ore localization. 

Suppose and $q : M \to I$ is a morphism in $\dcat{D}^{\star}(\cat{M})$. 
By axiom (L3) we have
$q = Q(s)^{-1} \o Q(a)$ for some 
$a : M \to N$ and some quasi-isomorphism $s: I \to N$ in 
$\dcat{K}^{\star}(\cat{M})$.
Lemma \ref{lem:121} says that there exists $b : N \to I$ in 
$\dcat{K}^{\star}(\cat{M})$ s.t.\ 
$b \o s = 1_I$. But then we get
\[ q = Q(s)^{-1} \o Q(a) =  Q(b \o s) \o  Q(s)^{-1} \o Q(a) = 
Q(b \o a) .  \]
This shows $Q$ is surjective.

Now let $a : M \to I$ be a morphism in $\dcat{K}^{\star}(\cat{M})$ s.t.\ 
$Q(a) = 0$. By axiom (L4)  there is a quasi-isomorphism 
$s : I \to N$  in $\dcat{K}^{\star}(\cat{M})$ s.t.\ $s \o a = 0$. 
Lemma \ref{lem:121} says that there exists $b : N \to I$ in 
$\dcat{K}^{\star}(\cat{M})$ s.t.\ 
$b \o s = 1_I$. Then
\[ a = (b \o s) \o a = b \o (s \o a) = 0 . \]
We proved that $Q$ is injective.
\end{proof}

\begin{thm} \label{thm:120}
The localization functor
\[ Q : \dcat{K}^{\star}(\cat{M})_{\tup{K-inj}} \to
\dcat{D}^{\star}(\cat{M}) \]
is fully faithful.
\end{thm}

\begin{proof}
This is clear from Lemma \ref{lem:122}: 
\[ Q : \opn{Hom}_{\dcat{K}^{\star}(\cat{M})}(I, J) \to 
\opn{Hom}_{\dcat{D}^{\star}(\cat{M})}(I, J) \]
is bijective for all $I, J \in \dcat{K}^{\star}(\cat{M})_{\tup{K-inj}}$.
\end{proof}

\begin{cor} \label{cor:120}
If $\dcat{K}^{\star}(\cat{M})$ has enough K-injectives, then the localization
functor
\[ Q : \dcat{K}^{\star}(\cat{M})_{\tup{K-inj}} \to
\dcat{D}^{\star}(\cat{M}) \]
is an equivalence.
\end{cor}

\begin{proof}
We already know that $Q$ is fully faithful. The extra condition says that it is
essentially surjective on objects. 
\end{proof}

A {\em generator} of $\cat{M}$ is an object $G \in \cat{M}$ s.t.\ every object
$M \in \cat{M}$ admits an epimorphism 
$G^{\oplus X} \surj M$ for some index set $X$ (in the universe $\cat{U}$). 

\begin{dfn}
A {\em Grothendieck abelian category} is an abelian category $\cat{M}$ that has
infinite direct sums, exact direct limits, and a generator.
\end{dfn}

\begin{exa}
Let $A$ be a ring. Then the free left module $A$ is a generator of 
$\cat{Mod}\, A$.

Let $(X, \mcal{A})$ be a ringed space. For any open set $U$ of $X$ let 
$g_U : U \to X$ be the inclusion, and let 
${g_U}_! (\mcal{A}|_U)$ be the extension by zero of the sheaf $\mcal{A}|_U$.
Then the sheaf 
$\bigoplus_{U} {g_U}_! (\mcal{A}|_U)$ is a generator of 
$\cat{Mod}\, \mcal{A}$.
\end{exa}

Here is the best existence result for K-injectives. 

\begin{thm}
Assume $\cat{M}$ is a Grothendieck abelian category. Then $\dcat{K}(\cat{M})$
has enough K-injectives. 
\end{thm}

See \cite[Theorem 14.3.1]{KS2} for a proof. This is due essentially to
Spaltenstein, Bockstedt-Neeman and Keller (around 1990). 
I will prove a weaker result, that will be sufficient for us.

\subsection{Bounded below injective resolutions}
Again $\cat{M}$ is some abelian category.

\begin{thm} \label{thm:110}
Let $I$ be a bounded below complex of injective objects of $\cat{M}$. Then 
$I$ is a K-injective complex. 
\end{thm}

\begin{proof}
Let $M$ be an acyclic complex and 
$\phi = \{ \phi^l \} \in \opn{Z}^k \bigl( \opn{Hom}_{\cat{M}}(M, I) \bigr)$,
i.e.\ $\phi$ is a $k$-cocycle. 
We have to prove that 
$\phi \in \opn{B}^k \bigl( \opn{Hom}_{\cat{M}}(M, I) \bigr)$,
i.e.\ $\phi$ is a coboundary. 
By shifting we can assume that $I$ is concentrated in degrees $\geq 0$, and
that $\phi$ has degree $0$. 

So the morphisms $\phi^l : M^l \to I^l$ 
satisfy 
\[ \d_I^l \o \phi^l  = \phi^{l+1} \o \d_M^l . \]
We have to find $\th^l : M^l \to I^{l-1}$ that satisfy 
\begin{equation} \label{eqn:113}
 \phi^l  = \th^{l+1} \o \d_M^l + \d_I^{l-1} \o \th^l .
\end{equation}

The proof is by induction. 
For $l \leq 0$ we take $\th^l := 0$ of course. Let $k \geq 0$, 
and assume that we have $\th^l$ defined for every $l \leq k$ s.t.\ 
(\ref{eqn:113}) holds for all $l < k$. 
We will construct $\th^{k+1}$  s.t.\ 
(\ref{eqn:113}) holds for all $l \leq k$. 
\[ \UseTips \xymatrix @C=9ex @R=10ex { 
&
\cdots
\ar[r]^{}
&
M^{k-1}
\ar[r]^{\d_M^{k-1}}
\ar[d]_(0.65){\phi^{k-1}}
\ar[dl]_{\th^{k-1}}
& 
M^{k}
\ar[r]^{\d_M^{k}}
\ar[d]_(0.65){\phi^{k}}
\ar[dl]_{\th^{k}}
&
M^{k+1}
\ar[r]^{}
\ar[d]_(0.65){\phi^{k+1}}
\ar@{-->}[dl]_{\th^{k+1}}
&
\cdots
\\
\cdots
\ar[r]^{}
&
I^{k-2}
\ar[r]_{\d_I^{k-2}}
&
I^{k-1}
\ar[r]_{\d_I^{k-1}}
& 
I^{k}
\ar[r]_{\d_I^{k}}
&
I^{k+1}
\ar[r]^{}
&
\cdots
} \]

According to (\ref{eqn:113}) with $l := k-1$ we get
\[ \phi^k \o \d_M^{k-1} = \d_I^{k-1} \o \phi^{k-1} = 
\d_I^{k-1} \o (\th^k \o \d_M^{k-1} + \d_I^{k-2} \o \th^{k-1}) 
= \d_I^{k-1} \o \th^k \o \d_M^{k-1}  . \]
Thus 
\begin{equation} \label{eqn:116}
(\phi^k - \d_I^{k-1} \o \th^k) \o \d_M^{k-1} = 0 . 
\end{equation}

Let 
\begin{equation} \label{eqn:114}
\opn{Y}^k(M) := \opn{Coker} \bigl( \d_M^{k-1} : M^{k-1} \to  M^{k} \bigr) . 
\end{equation}
There is an induced morphism 
\[ \bar{\d}_M^{k} : \opn{Y}^k(M) \to M^{k+1} , \]
and a canonical isomorphism
\begin{equation} \label{eqn:115}
\opn{H}^k(M) \cong \opn{Ker} \bigl( \bar{\d}_M^{k} : \opn{Y}^k(M) \to M^{k+1}
\bigr) . 
\end{equation}
Since $M$ is acyclic we conclude that $\bar{\d}_M^{k}$ is a monomorphism.
By (\ref{eqn:116}) we get an induced morphism 
$\psi^k : \opn{Y}^k(M) \to I^{k}$
s.t.\ the solid diagram 
\[ \UseTips \xymatrix @C=8ex @R=8ex { 
M^k 
\ar@{->>}[r]
\ar[dr]_{\phi^k - \d_I^{k-1} \o\, \th^k}
&
\opn{Y}^k(M)
\ar[d]^{\psi^k}
\ar@{^{(}->}[r]^{\bar{\d}_M^{k}}
&
M^{k+1}
\ar@{-->}[dl]^{\th^{k+1}}
\\
&
I^k
} \]
commutes. But $I^k$ is an injective object, so a morphism $\th^{k+1}$ exists to
make the whole diagram commutative. 
It is easy to check that (\ref{eqn:113}) holds for $k$. 
\end{proof}

Recall that $\cat{M}$ is said to have {\em enough injectives} is every 
$M \in \cat{M}$ admits an embedding (a monomorphism) $M \to I$ for some
injective object $I$. 

We know that in this case any $M \in \cat{M}$ admits an injective resolution 
$M \to I$. Let's remember how $I$ is constructed. We start by embedding 
$M$ in an injective object $I^0$. Then we embed $\opn{Coker}(M \to I^0)$ in an
injective object $I^1$. And so on. 
This construction works also for a bounded below complex $M$, as we now prove.

\begin{thm} \label{thm:111}
Assume $\cat{M}$ has enough injectives. Then every $M \in \dcat{K}^+(\cat{M})$
admits a quasi-isomorphism $M \to I$, where $I$ is a bounded below complex 
of injectives in $\cat{M}$. 
\end{thm}

\begin{proof}
After shifting is necessary, we can assume $M$ is concentrated in degrees 
$\geq 0$.
We will find injective objects $I^p$, and morphisms
$\d^p_I : I^p \to I^{p+1}$, $\zeta^p : M^p \to I^p$
that form a complex $I$ and a quasi-isomorphism $\zeta : M \to I$. 
This is done by induction on $p$. For $p < 0$ we take $I^p := 0$, and of
course $\zeta^p = \d_I^{p-1} := 0$. 
Now let $p \geq 0$, and suppose that we have found the objects 
$I^{p'}$, and the related morphisms $\zeta^{p'}$ and
$\d_I^{p'-1}$, for all $p' < p$. Consider the object 
\[ N := \opn{Coker} \Bigl( (\d_M^{p-1}, \zeta^{p-1}) : M^{p-1} 
\to M^p \oplus \opn{Y}^{p-1}(I) \Bigr) . \]
See (\ref{eqn:114}). 

\[ \UseTips \xymatrix @C=7ex @R=5ex {
& &
M^{p-1}
\ar[r]^{\d_M^{p-1}}
\ar[dl]_{\ze^{p-1}}
\ar[d]
& 
M^p
\ar[d]
\ar@{-->}[dr]^{\ze^{p}}
\\
I^{p-2}
\ar[r]^{\d_I^{p-2}}
&
I^{p-1}
\ar@{->>}[r]^{}
\ar@(dr,dl)@{-->}[rrr]_{\d_I^{p-1}}
&
\opn{Y}^{p-1}(I)
\ar[r]^{}
&
N 
\ar@{^{(}-->}[r]^{}
&
I^p
} \]

Choose an embedding $N \inj I^{p}$ for some injective object $I^p$. 
There are induced morphisms 
$\d^{p-1}_I : I^{p-1} \to I^{p}$ and $\zeta^p : M^p \to I^p$, and these satisfy 
\[ \d^{p-1}_I \o \zeta^{p-1} = \zeta^p \o \d_M^{p-1} \]
and 
\[ \d^{p-1}_I \o \d^{p-2}_I = 0 . \]
Thus, after doing this for all $p$, we get a complex $I$ and 
a morphism of complexes $\zeta : M \to I$.

It remains to prove that $\zeta$ is a quasi-isomorphism. This is a
straightforward calculation that is left as an exercise. To facilitate the
calculation, notice first that we can work with ``elements'' of the objects.
This is justified by the Freyd-Mitchell Theorem. Alternatively we can use
``points'' as in algebraic geometry; i.e.\ for any $N \in \cat{M}$ its
``points'' are the morphisms $x :L \to N$ in $\cat{M}$. (Actually this is the
idea behind the  Freyd-Mitchell Theorem.) So $\phi : N \to N'$ is an epimorphism
(resp.\ monomorphism) iff it is surjective (resp.\ injective) on points.

Also notice that the way we introduced $I^p$ guarantees that 
\[ \opn{H}^{p-1}(\ze) : \opn{H}^{p-1}(M) \to \opn{H}^{p-1}(I) \]
is an epimorphism, and 
\[ \opn{H}^{p}(\ze) : \opn{H}^{p}(M) \to \opn{H}^{p}(I) \]
is a monomorphism.  
\end{proof}

\begin{exer}
Theorem \ref{thm:111} can be improved as follows: instead of assuming that $M$
is bounded below, let's only assume that its cohomology $\opn{H}(M)$ is bounded
below. Say 
$\opn{H}^p(M) = 0$ for $p < p_0$. Then there is a quasi-isomorphism 
$\ze : M \to I$, where $I$ is a complex of injectives concentrated in degrees 
$\geq p_0$. 
\end{exer}


\begin{cor} \label{cor:130}
If $\cat{M}$ has enough injectives, then $\dcat{K}^{+}(\cat{M})$ has enough
K-injectives. 
\end{cor}

\begin{proof}
Combine Theorems \ref{thm:111} and \ref{thm:110}.
\end{proof}

We don't know much about the structure of K-injective complexes. All we can say
is:

\begin{prop}
Assume $\cat{M}$ has enough injectives. Let 
$I \in \dcat{K}^{+}(\cat{M})_{\tup{K-inj}}$. Then 
there is an isomorphism $I \cong J$ in $\dcat{K}^{}(\cat{M})$
with $J$ a bounded below complex of injective objects of $\cat{M}$.
\end{prop}

\begin{proof}
By Theorem \ref{thm:111} there is a quasi-isomorphism $\phi : I \to J$
in $\dcat{K}^{}(\cat{M})$, where $J$
is a bounded below complex of injective objects of $\cat{M}$. Now $Q(\phi)$ is
an isomorphism in $\dcat{D}^{+}(\cat{M})$. According to Theorem
\ref{thm:120} the morphism $\phi$ is already an isomorphism.
\end{proof}

\begin{exer}
Find a K-injective complex $I \in \dcat{K}^{+}(\cat{Ab})$
which is {\em not} made up of injective objects of $\cat{Ab}  = \cat{Mod}\, \Z$.
\end{exer}

\subsection{More on K-projectives}
Let's recall what K-projectives are.
A complex $P \in \dcat{K}(\cat{M})$ is called {\em K-projective} if for
every acyclic $N \in \dcat{K}(\cat{M})$, the complex 
$\opn{Hom}_{\cat{M}}(P, N)$ is also acyclic.

There are similar results here. We omit the proofs. The results can either be
proved by modifying the proofs of the corresponding results in the previous
subsections, or by using:

\begin{lem}
There is a contravariant isomorphism of triangulated categories 
$\opn{op} : \dcat{K}(\cat{M}) \to \dcat{K}(\cat{M}^{\mrm{op}})$, such
that $\opn{op}(\dcat{K}^{\star}(\cat{M})) =
\dcat{K}^{\opn{op}(\star)}(\cat{M}^{\mrm{op}})$
for $\star \in \{ \mrm{b}, - , + , \tup{\textvisiblespace} \}$.
A complex $P \in \dcat{K}(\cat{M})$ is K-projective iff 
$\opn{op}(P) \in \dcat{K}(\cat{M}^{\mrm{op}})$ is 
K-injective.

Similarly there are contravariant isomorphisms of triangulated categories 
$\opn{op} : \dcat{D}^{\star}(\cat{M}) \to 
\dcat{D}^{\opn{op}(\star)}(\cat{M}^{\mrm{op}})$,
that commute with the localization functors $Q$.
\end{lem}

The meaning of $\opn{op}(\star)$ should be clear...
The proof is an exercise. 

We denote by $\dcat{K}^{\star}(\cat{M})_{\tup{K-proj}}$  the full subcategory
of $\dcat{K}^{\star}(\cat{M})$ on the K-projective complexes. 
This will be usually interesting only for 
$\star \in \{ -, \text{\textvisiblespace} \}$.

\begin{prop} 
$\dcat{K}^{\star}(\cat{M})_{\tup{K-proj}}$ is a triangulated subcategory of 
$\dcat{K}^{\star}(\cat{M})$.
\end{prop}

\begin{proof}
This is like Proposition \ref{prop:130}.
\end{proof}

\begin{thm} 
The localization functor
\[ Q : \dcat{K}^{\star}(\cat{M})_{\tup{K-proj}} \to
\dcat{D}^{\star}(\cat{M}) \]
is fully faithful.
\end{thm}

\begin{proof}
This is like Theorem \ref{thm:120}.
\end{proof}

We say that $\dcat{K}^{\star}(\cat{M})$ has {\em enough K-projectives} if every 
$M \in \dcat{K}^{\star}(\cat{M})$ admits a quasi-isomorphism 
$P \to M$ with $P \in \dcat{K}^{\star}(\cat{M})_{\tup{K-proj}}$.

\begin{cor} 
If $\dcat{K}^{\star}(\cat{M})$ has enough K-projectives, then the localization
functor
\[ Q : \dcat{K}^{\star}(\cat{M})_{\tup{K-proj}} \to
\dcat{D}^{\star}(\cat{M}) \]
is an equivalence.
\end{cor}

\begin{proof}
This is like Corollary \ref{cor:120}.
\end{proof}

The best existence result is this:

\begin{thm} \label{thm:132}
For a ring $A$ the category $\dcat{K}(\cat{Mod}\, A)$ has enough K-projectives.
\end{thm}

This is not very hard to prove -- see \cite[Theorem 14.4.3]{KS2}. I will
prove something weaker below. 

\begin{thm}  \label{thm:133}
Let $P$ be a bounded above complex of projective objects of $\cat{M}$. Then 
$P$ is a K-projective complex. 
\end{thm}

\begin{proof}
Like Theorem \ref{thm:110}.
\end{proof}

\begin{thm}  \label{thm:124}
Assume $\cat{M}$ has enough projectives. Then every $M \in \dcat{K}^-(\cat{M})$
admits a quasi-isomorphism $P \to M$, where $P$ is a bounded above
complex  of projectives in $\cat{M}$. 
\end{thm}

\begin{proof}
Like Theorem \ref{thm:111}.
\end{proof}

\begin{cor}
If $\cat{M}$ has enough projectives, then $\dcat{K}^{-}(\cat{M})$ has enough
K-projectives. 
\end{cor}

\begin{proof}
Like Corollary \ref{cor:130}.
\end{proof}

When $A$ is a ring we can also find free resolutions. This is better than
Theorem \ref{thm:124}.

\begin{thm} \label{thm:125}
Let $A$ be a ring. Then every $M \in \dcat{K}^-(\cat{Mod}\, A)$
admits a quasi-isomorphism $P \to M$, where $P$ is a bounded above
complex  of free $A$-modules. 

If $A$ is left noetherian, and each module $\opn{H}^i(M)$ is finitely
generated, then we can choose $P$ s.t.\ each $P^i$ is a finitely
generated free module. 
\end{thm}

\begin{proof}
Proceed as in the proof of Theorem \ref{thm:111}, but in the opposite
direction. To construct $P^i$, let
\[ N := \opn{Ker} \Bigl( (\d_M^{i}, \zeta^{i+1}) : 
M^i \times \opn{Z}^{i+1}(P) \to M^{i+1} \Bigr) . \]
Now choose any epimorphism $\phi : P^i \surj N$ where $P^i$ is a free module. 

\[ \UseTips \xymatrix @C=7ex @R=5ex {
P^{i}
\ar@{->>}[r]^{\phi}
\ar@{-->}[dr]_{\ze^{i}}
\ar@(ur,ul)@{-->}[rrr]^{\d_P^{i}}
&
N
\ar[r]^{}
\ar[d]
&
\opn{Z}^{i+1}(P)
\ar@{^{(}->}[r]^{}
\ar[d]
&
P^{i+1}
\ar[r]^{\d_P^{i+1}}
\ar[dl]^{\ze^{i+1}}
&
P^{i+2}
\\
&
M^{i}
\ar[r]^{\d_M^{i}}
& 
M^{i+1}
} \]

We obtain a morphism of complexes $\ze : P \to M$. 
The way we introduced $P^i$ guarantees that 
\[ \opn{H}^{i+1}(\ze) : \opn{H}^{i+1}(P) \to \opn{H}^{i+1}(M) \]
is a monomorphism, and 
\[ \opn{H}^{i}(\ze) : \opn{H}^{i}(P) \to \opn{H}^{i}(M) \]
is an epimorphism.  

In the finitely generated and noetherian case we also assume inductively that
$P^{i'}$ are finitely generated for all $i' > i$. Since $\opn{Z}^{i+1}(P)$ is a
finitely generated module, a calculation shows that
there are finitely many elements in $N$ that are
responsible for $\opn{H}^{i+1}(\ze)$ to be a monomorphism.
Similarly, since $\opn{H}^{i}(M)$ is finitely generated, there are finitely
many elements in $N$ that insure that $\opn{H}^{i}(\ze)$ is an epimorphism.
We can find a finitely generated free module $P^i$ whose
image in $N$ contains all these special elements. 
\end{proof}

\cleardoublepage
\section{Derived Functors}

\subsection{Right derived functors} \label{subsec:130}
Here is a refined version of Definition \ref{dfn:8}.

\begin{dfn}  \label{dfn:131}
Let $\cat{E}$ be a triangulated category, and  
$F : \dcat{K}^{\star}(\cat{M}) \to \cat{E}$ a triangulated functor. 
A {\em right derived functor} of $F$ 
is a triangulated functor 
\[ \mrm{R}^{\star} F : \dcat{D}(\cat{M}) \to \cat{E} , \]
together with a morphism 
\[ \eta : F \to \mrm{R}^{\star} F \circ Q \]
of triangulated functors $\dcat{K}^{\star}(\cat{M}) \to \cat{E}$. The pair 
$(\mrm{R}^{\star} F, \eta)$ must have this universal property:
\begin{itemize}
\item[($*$)] Given any triangulated functor 
$G : \dcat{D}^{\star}(\cat{M}) \to \cat{E}$,
and a morphism of triangulated functors $\eta' : F \to G \circ Q$, there is a
unique morphism of triangulated functors $\th : \mrm{R}^{\star} F \to G$
s.t.\ $\eta' = \th \circ \eta$. 
\end{itemize}
\end{dfn}

Just like Proposition \ref{prop:3}, we have:

\begin{prop} 
If a right derived functor $(\mrm{R}^{\star} F, \eta)$ exists, then it is
unique, up to a unique isomorphism of triangulated functors.
\end{prop}

The variant of Theorem \ref{thm:101} is:

\begin{thm}  \label{thm:121}
Let $\cat{E}$ be triangulated category, and 
$F : \dcat{K}^{\star}(\cat{M}) \to \cat{E}$ a triangulated functor. 
Assume that $\dcat{K}^{\star}(\cat{M})$ has enough K-injectives. 
Then the right derived functor 
$\mrm{R}^{\star} F : \dcat{D}(\cat{M}) \to \cat{E}$
exists. Moreover, for any $I \in \dcat{K}^{\star}(\cat{M})_{\tup{K-inj}}$ the
morphism  
\[ \eta_I : F(I) \to (\mrm{R}^{\star} F \circ Q)(I) \]
in $\cat{E}$ is an isomorphism.
\end{thm}

\begin{proof}
This is similar to the proof of Theorem \ref{thm:101}, that we only sketched.
Actually here the proof is easier, since we don't have to localize 
$\dcat{K}^{\star}(\cat{M})_{\mrm{inj}}$ with respect to quasi-isomorphisms.

First, for every $M \in \dcat{K}^{\star}(\cat{M})$ we choose a 
quasi-isomorphism $\ze_M : M \to I(M)$, with 
$I(M) \in \dcat{K}^{\star}(\cat{M})_{\mrm{inj}}$. 
We choose $I(M)$ and $\ze_M$ s.t.\ $I(M[p]) = I(M)[p]$ for all $p \in \Z$, and
also $I(M) = M$ if $M \in \dcat{K}^{\star}(\cat{M})_{\mrm{inj}}$. 
Then to any morphism $\phi : M \to N$ in $\dcat{K}^{\star}(\cat{M})$
we get a morphism 
\[ I(\phi) := Q(\ze_N) \o Q(\phi) \o Q(\ze_M)^{-1} :
I(M) \to I(N) \]
in $\dcat{D}^{\star}(\cat{M})$.
But because 
$Q : \dcat{K}^{\star}(\cat{M})_{\mrm{inj}} \to
\dcat{K}^{\star}(\cat{M})$
is fully faithful, $I(\phi)$ lifts uniquely to a morphism in
$\dcat{K}^{\star}(\cat{M})_{\mrm{inj}}$. 
Thus we have a functor 
\[ I : \dcat{K}^{\star}(\cat{M}) \to \dcat{K}^{\star}(\cat{M})_{\mrm{inj}}  \]
that splits the inclusion 
$\dcat{K}^{\star}(\cat{M})_{\mrm{inj}} \to \dcat{K}^{\star}(\cat{M})$,
and a morphism 
\[ \ze : \bsym{1} \to I \]
of functors $\dcat{K}^{\star}(\cat{M}) \to \dcat{K}^{\star}(\cat{M})$.
It is not hard to check that $I$ is a triangulated functor, and $\ze$ is a
morphism of triangulated functors. 
Also $I$ sends quasi-isomorphism to isomorphisms, and therefore it extends to a
triangulated functor 
\[ I : \dcat{D}^{\star}(\cat{M}) \to \dcat{K}^{\star}(\cat{M})_{\mrm{inj}} . \]

Now we define 
$\mrm{R}^{\star} F : \dcat{D}(\cat{M}) \to \cat{E}$
to be
\[ \mrm{R}^{\star} F := F|_{\dcat{K}^{\star}(\cat{M})_{\mrm{inj}}} \o I . \]
And we define 
\[ \eta_M : F(M) \to (\mrm{R}^{\star} F \circ Q)(M) = (F \o I)(M) \]
 to be 
\[ \eta_M := F(\ze_M) . \]
I leave it as an exercise to verify that $(\mrm{R}^{\star} F, \eta)$ is indeed
a right derived functor of $F$. 
\footnote{More details for the proof can be found here:
\url{http://www.math.bgu.ac.il/~amyekut%
/publications/course-der-cats/explanation_14-1-3.pdf}.}
\end{proof}

Let $\dagger \in \{ \mrm{b}, -, +  \}$
be s.t.\ $\dcat{K}^{\dagger}(\cat{M}) \subset \dcat{K}^{\star}(\cat{M})$,
and s.t.\ 
the derived functors $\mrm{R}^{\star} F$ and 
$\mrm{R}^{\dagger} F := \mrm{R}^{\dagger} (F|_{\dcat{K}^{\dagger}(\cat{M})})$
both exist. The universal property gives rise to a canonical morphism 
$\mrm{R}^{\dagger} F \to (\mrm{R}^{\star} F)|_{\dcat{D}^{\dagger}(\cat{M})}$
of triangulated functors 
$\dcat{D}^{\dagger}(\cat{M}) \to \cat{E}$. 
In general this might not be an isomorphism (but I don't know an example!).
However:

\begin{prop} \label{prop:132}
In the situation of Theorem \tup{\ref{thm:121}}, assume that 
$\dcat{K}^{\dagger}(\cat{M})$ also has enough K-injectives. Then the morphism 
$\mrm{R}^{\dagger} F \to (\mrm{R}^{\star} F)|_{\dcat{D}^{\dagger}(\cat{M})}$
is an isomorphism.
\end{prop}

\begin{proof}
This is clear since 
$\mrm{R}^{\dagger} F (I) \to F(I)$
and $\mrm{R}^{\star} F (I) \to F(I)$
are isomorphisms when 
$I \in \dcat{K}^{\dagger}(\cat{M})_{\mrm{inj}} \subset 
\dcat{K}^{\star}(\cat{M})_{\mrm{inj}}$.
\end{proof}

When the conclusion of this proposition holds, we will simply write 
$\mrm{R}^{} F$ for the right derived functor.

\subsection{Left derived functor}
This is like Subsection \ref{subsec:130}, with sides reversed. 
No need to elaborate.

\subsection{Derived bifunctors}
\mbox{}

\begin{dfn} \label{dfn:132}
Let $\cat{M}$ and $\cat{N}$ be abelian categories, let 
$\cat{E}$ be a triangulated category, and let 
$\star, \dagger \in \{ \mrm{b}, - , + , \tup{\textvisiblespace} \}$.
Suppose we are given a bitriangulated bifunctor 
\[ F : \dcat{K}^{\dagger}(\cat{M}) \times \dcat{K}^{\star}(\cat{N}) \to \cat{E}
. \]
A {\em right derived bifunctor} of $F$
is a bitriangulated bifunctor
\[ \mrm{R}^{\dagger, \star} F : \dcat{D}^{\dagger}(\cat{M}) \times
\dcat{D}^{\star}(\cat{N}) \to \cat{E} , \]
with a morphism of bitriangulated bifunctors
\[ \eta : F \to \mrm{R}^{\dagger, \star} F \circ (Q \times Q) , \]
that has the same universal property as in Definition \ref{dfn:131}. 
\end{dfn}

There is a (by now obvious) uniqueness of $(\mrm{R}^{\dagger, \star} F, \eta)$.
Here is a refinement of Theorem \ref{thm:131}.

\begin{thm}
In the situation of Definition \tup{\ref{dfn:132}}, suppose there is a full
triangulated subcategory 
$\cat{J} \subset \dcat{K}^{\star}(\cat{N})$ with these two properties:
\begin{enumerate}
\rmitem{i}  If $\psi : I \to I'$ is a quasi-isomorphism in $\cat{J}$, and 
$\phi : M \to M'$ is a quasi-isomorphism in $\dcat{K}^{\dagger}(\cat{M})$, then 
\[ F(\phi, \psi) : F(M, I) \to F(M', I') \]
is an isomorphism in $\cat{E}$.

\rmitem{ii} Every $N \in \dcat{K}^{\star}(\cat{N})$ admits a quasi-isomorphism 
$N \to I$ with $I \in \cat{J}$.
\end{enumerate}
Then the right derived functor $(\mrm{R}^{\dagger, \star} F, \eta)$ exists, and
moreover
\[ \eta_{M, I} :  F(M, I) \to \mrm{R}^{\dagger, \star} F(M, I) \]
is an isomorphism for any $I \in \cat{J}$.
\end{thm}

\begin{proof}
This is an exercise. Cf.\ proofs of Theorems \ref{thm:101} and \ref{thm:121}.
\end{proof}

For subcategories 
$\dcat{K}^{\til{\dagger}}(\cat{M}) \subset \dcat{K}^{\dagger}(\cat{M})$
and
$\dcat{K}^{\til{\star}}(\cat{N}) \subset \dcat{K}^{\star}(\cat{N})$
there is a result similar to Proposition \ref{prop:132}. 
Note also the symmetry between $\cat{M}$ and $\cat{N}$. 

There is a similar definition of left derived bifunctor 
$(\mrm{L}^{\dagger, \star} F, \eta)$, and a similar existence result.

%
%


\begin{rem}
Here is something I realized last week: there is an ambiguity in the
triangulated structure of $\dcat{K}^{\star}(\cat{M})$, for 
$\star \in \{ \mrm{b}, -, + \}$. Consider a triangle $T$ in 
$\dcat{K}^{\star}(\cat{M})$. One possibility is to declare $T$ to be
distinguished if it is a distinguished triangle in 
$\dcat{K}^{}(\cat{M})$; namely $T \cong T'$ for some standard
triangle $T'$ in $\dcat{K}^{}(\cat{M})$. See Definition \ref{dfn:140}. 
But there is another option: declare $T$ to be
distinguished if  $T \cong T'$ for some standard
triangle $T'$ in $\dcat{K}^{\star}(\cat{M})$.

It turns out that these conditions are equivalent. This is a nice exercise.
\end{rem}

Let $\dcat{D}(\cat{M})^{\mrm{b}}$ be the full subcategory of 
$\dcat{D}(\cat{M})$ consisting of complexes with bounded cohomology. 
We already saw that the inclusion 
$\dcat{D}^{\mrm{b}}(\cat{M}) \to \dcat{D}(\cat{M})^{\mrm{b}}$
is an equivalence of triangulated categories. We shall often want to identify
these categories, and to do this we choose (implicitly) a quasi-inverse 
$\dcat{D}(\cat{M})^{\mrm{b}} \to \dcat{D}^{\mrm{b}}(\cat{M})$, for instance 
\[ M \mapsto (\opn{smt}^{\geq i_0(M)} \circ \opn{smt}^{\leq i_1(M)})(M) , \]
where $i_0(M) \ll 0 \ll i_1(M)$ are integers depending on 
$M$. 

\subsection{The left derived tensor functor}
Consider a commutative ring $A$. (This also work for noncommutative
rings, but we have to tensor a left module with a right module, and the
results is a $\Z$-module.) The biadditive bifunctor 
\[ F := - \ot_A - : \cat{Mod}\, A \times \cat{Mod}\, A \to 
\cat{Mod}\, A \]
extends to a biadditive bifunctor 
\begin{equation} \label{eqn:142}
F := - \ot_A - : \dcat{C}^{\star}(\cat{Mod}\, A) \times 
\dcat{C}^{\star}(\cat{Mod}\, A) \to  \dcat{C}^{\star}(\cat{Mod}\, A) , 
\end{equation}
for $\star \in \{ \mrm{b}, - , + , \tup{\textvisiblespace} \}$,
using the direct sum totalization. Namely 
\[ F(M, N)^i = (M \ot_A N)^i = \bigoplus_{j \in \Z}\ (M^j \ot_A N^{i-j})
\in \cat{Mod}\, A . \]

\begin{prop} \label{prop:142}
The bifunctor \tup{(\ref{eqn:142})} induces a bitriangulated bifunctor
\[ F := - \ot_A - : \dcat{K}^{\star}(\cat{Mod}\, A) \times 
\dcat{K}^{\star}(\cat{Mod}\, A) \to  \dcat{K}^{\star}(\cat{Mod}\, A) . \]
\end{prop}

\begin{proof}
We first note that the tensor product of complexes respects homotopies, and
therefore there is an induced bifunctor between the homotopy categories. 
Clearly this bifunctor respects shifts. As for distinguished triangles,
consider a standard triangle 
\[ T = \bigl( L \xar{\al} M \to N \to L[1] \bigr) \]
in $\dcat{C}^{}(\cat{Mod}\, A)$, i.e.\ $N = \opn{cone}(\al)$. Then for any
complex $P$, the complex
$P \ot_A N$ is isomorphic to the cone on 
$1_P \ot \al$, and moreover $P \ot_A T$ is isomorphic, in  
$\dcat{C}^{}(\cat{Mod}\, A)$, to the corresponding standard triangle.
This implies that for any distinguished triangle $T$ in 
$\dcat{K}^{}(\cat{Mod}\, A)$, the triangle $P \ot_A T$ is distinguished.
\end{proof}

\begin{dfn}
A complex $P \in \dcat{K}^{\star}(\cat{Mod}\, A)$ is called a {\em K-flat
complex} if for any acyclic complex 
$M \in \dcat{K}^{\star}(\cat{Mod}\, A)$ the complex 
$M \ot_A P$ is also acyclic. 
\end{dfn}

We denote by $\dcat{K}^{\star}(\cat{Mod}\, A)_{\mrm{flat}}$ the full 
subcategory of $\dcat{K}^{\star}(\cat{Mod}\, A)$ consisting of K-flat
complexes. 

\begin{lem} \mbox{} \label{lem:140}
\begin{enumerate}
\item The subcategory $\dcat{K}^{\star}(\cat{Mod}\, A)_{\mrm{flat}}$ is
triangulated.

\item If $P$ is K-projective then it is K-flat. Thus 
$\dcat{K}^{}(\cat{Mod}\, A)$ and 
$\dcat{K}^{-}(\cat{Mod}\, A)$ have enough K-flats.

\item Suppose $\phi : M \to M'$ is a quasi-isomorphism in 
$\dcat{K}^{\star}(\cat{Mod}\, A)$, and 
$\psi : P \to P'$ is a quasi-isomorphism in 
$\dcat{K}^{\star}(\cat{Mod}\, A)_{\mrm{flat}}$. 
Then 
\[ \phi \ot \psi : M \ot_A P \to M' \ot_A P' \]
is a quasi-isomorphism. 
\end{enumerate}
\end{lem}

\begin{proof}
(1) - (2). This is an exercise. 

\medskip \noindent
(3) First consider $P = P'$ and $\psi = 1_P$. Let $N$ be the cone on $\phi$. So
$N$ is acyclic, and there is a distinguished triangle 
$M \xar{\phi} M' \to N \to M[1]$ in $\dcat{K}^{\star}(\cat{Mod}\, A)$.
We apply the triangulated functor $- \ot_A P$, and get a distinguished triangle 
\[ M \ot_A P \xar{\phi \ot 1_P} M'  \ot_A P \to N  \ot_A P \to M[1]  \ot_A P .
\]
Since $N  \ot_A P$ is acyclic, we see that $\phi \ot 1_P$ is a
quasi-isomorphism.

Next we consider $M = M'$ and $\phi = 1_M$. Choose a K-flat resolution 
$Q \to M$. We get a commutative diagram 
\[ \UseTips \xymatrix @C=10ex @R=5ex {
Q \ot_A P
\ar[r]
\ar[d]
&
Q \ot_A P'
\ar[d]
\\
M \ot_A P
\ar[r]^{1_M \ot \psi}
&
M \ot_A P' . 
} \]
By the previous paragraph the three unmarked arrows are quasi-isomorphisms; and
hence $1_M \ot \psi$ is also a quasi-isomorphism.
\end{proof}

\begin{thm}
Write $F := - \ot_A -$. 
The left derived functor $(\eta,\, \opn{L}^{\star, \star} F)$ exists for 
$\star \in \{ -, \tup{\textvisiblespace} \}$. 
We denote it by $- \ot^{\mrm{L}}_A -$.
If either $M$ or $N$ is K-flat, then the morphism 
\[ \eta : M \ot^{\mrm{L}}_A N \to M \ot_A N  \]
is an isomorphism.
\end{thm}

\begin{dfn}
We write 
\[ \opn{Tor}_i^A(M, N) := \opn{H}^{-i} ( M \ot^{\mrm{L}}_A N) \in \cat{Mod}\, A
. \]
\end{dfn}

\begin{prop}
Assume $A$ is noetherian. If $M, N \in \dcat{D}^{-}_{\mrm{f}}(\cat{Mod}\, A)$,
then $M \ot^{\mrm{L}}_A N \in \dcat{D}^{-}_{\mrm{f}}(\cat{Mod}\, A)$.
\end{prop}

\begin{proof}
Exercise. (Hint: choose free resolutions $P \to M$ and $Q \to N$,  with $P, Q$
bounded above complexes of finitely generated free module.)
\end{proof}

\begin{dfn} \label{dfn:143}
The ring $A$ is called {\em regular}, or {\em of finite global cohomological
dimension}, if there is a natural number $d$ such that 
\[ \opn{Ext}^i_{A}(M, N) = 0 \]
for all $M, N \in \cat{Mod}\, A$ and all $i > d$.

The smallest such number $d$ is called the {\em global cohomological dimension}
of $A$.
\end{dfn}

\begin{rem}
For a commutative noetherian ring $A$ of finite Krull dimension, being regular
in this sense is equivalent to the usual definition, namely that for every prime
ideal $\mfrak{p}$ the ring $A_{\mfrak{p}}$ is a regular local ring. 
Thus for instance $\Z$, or any field $\K$, are regular. If A is regular, then 
the polynomial ring $A[t_1, \ldots, t_n]$ is also regular. 
\end{rem}

\begin{rem}
The definition above makes sense, and is extremely useful, also for
noncommutative rings. 
\end{rem}

\begin{prop}
If $A$ is regular, then $\dcat{K}^{\mrm{b}}(\cat{Mod}\, A)$ has enough
K-projectives.
\end{prop}

\begin{proof}[Sketch of proof]
Let $M$ be a bounded complex. 
Choose a projective resolution $P \to M$ in 
$\dcat{K}^{-}(\cat{Mod}\, A)$. Then for sufficiently small $i$ the
smart truncation 
$P' := \opn{smt}^{\geq i}(P)$
still gives a quasi-isomorphism $P' \to M$ (this is easy), and also consists of
projectives (this is the hard part: to show that the module 
$\opn{Y}^{i}(P)$ is projective -- cf.\ proof of Proposition \ref{prop:141}).
\end{proof}

\begin{cor}
If $A$ is a regular ring, then the left derived functor 
$\opn{L}^{\mrm{b}, \mrm{b}} (- \ot_A -)$ exists.
Namely if $M$ and $N$ are bounded complexes, then 
$M \ot^{\mrm{L}}_A N$ is bounded.
\end{cor}

\begin{dfn}
A complex $M \in \dcat{D}^{\mrm{b}}(\cat{Mod}\, A)$ has {\em finite flat
dimension} if 
there is a natural number $i_0$ such that 
$\opn{Tor}_i^A(M, N) = 0$
for all $N \in \cat{Mod}\, A$ and $i < i_0$.
\end{dfn}

\begin{prop} \label{prop:141}
The following are equivalent for $M \in \dcat{D}^{\mrm{b}}(\cat{Mod}\, A)$:
\begin{enumerate}
\rmitem{i} $M$ has finite flat dimension.
\rmitem{ii} There is a quasi-isomorphism $P \to M$ with $P$ a bounded complex of
flat $A$-modules. 
\end{enumerate}
\end{prop}

\begin{proof}
The implication (ii) $\Rightarrow$ (i) is trivial. For the opposite
implication, 
let $i_0$ be such that 
$\opn{H}^i (M \ot^{\mrm{L}}_A N) = 0$
for all $N \in \cat{Mod}\, A$ and $i < i_0$.
In particular $\opn{H}^i(M) = 0$ for $i < i_0$.
Let us choose a flat resolution $P' \to M$ ($P$ is a bounded above
complex of flat modules). Consider the smart truncation 
$P := \opn{smt}^{\geq i_0}(P')$. 
We get  quasi-isomorphisms $P \to P' \to M$, and $P$ is a bounded complex.
Now 
\[ P = \opn{smt}^{\geq i_0}(P') = 
\bigl( \cdots \to 0 \to \opn{Y}^{i_0}(P') \xar{\d} {P'}^{i_0+1} \xar{\d}
{P'}^{i_0+2} \to \cdots \bigr) . \]
We will prove that $P^{i_0} = \opn{Y}^{i_0}(P')$ is flat. 

Look at the stupid truncation of $P$:
\[ \opn{stt}^{> i_0}(P) = 
\bigl( \cdots \to 0 \xar{} P^{i_0+1} \xar{\d} P^{i_0+2}
\to \cdots \bigr) . \]
 we get an exact sequence of complexes,
and hence a distinguished triangle:
\[ \opn{stt}^{> i_0}(P) \to P \to P^{i_0}[-i_0] \to \opn{stt}^{> i_0}(P)[1] . \]
Take any $A$-module $N$. 
By assumption 
\[ \opn{H}^i (P \ot^{\mrm{L}}_A N) = \opn{H}^i (M \ot^{\mrm{L}}_A N)  = 0 \]
for all $i < i_0$. Because $\opn{stt}^{> i_0}(P)$ is a bounded complex of
flat modules concentrated in degrees $> i_0$, we also have 
\[ \opn{H}^i \bigl ( \opn{stt}^{> i_0}(P), \ot^{\mrm{L}}_A N \bigr) = 0 \]
for all $i \leq i_0$. There is an exact sequence 
\[ \opn{H}^i (P \ot^{\mrm{L}}_A N) \to 
\opn{H}^i \bigl( P^{i_0}[-i_0] \ot^{\mrm{L}}_A N \bigr) \to
\opn{H}^i \bigl( \opn{stt}^{> i_0}(P)[1] \bigr) . \]
Renumbering we get the exact sequence 
\[ \opn{H}^{i + i_0} (P \ot^{\mrm{L}}_A N) \to 
\opn{H}^i \bigl( P^{i_0} \ot^{\mrm{L}}_A N \bigr) \to
\opn{H}^{i + i_0 + 1} \bigl( \opn{stt}^{> i_0}(P) \bigr) . \]
The end terms vanish for $i < 0$; and hence also the middle term. 
We see that 
\[ \opn{Tor}_j^A( P^{i_0} , N) = \opn{H}^{-j} \bigl( P^{i_0} \ot^{\mrm{L}}_A N
\bigr) = 0 \]
for $j > 0$, so $P^{i_0}$ is flat. 
\end{proof}

\subsection{The right derived Hom functor}
Let $\cat{M}$ be a $\K$-linear abelian category. (E.g.\ 
$\cat{M} = \cat{Mod}\, A$ and $\K = A$ for a commutative ring $A$.)
As in Proposition \ref{prop:142}, the biadditive bifunctor 
\[ F := \opn{Hom}_{\cat{M}}(-, -) : \cat{M}^{\mrm{op}} \times \cat{M}
\to \cat{Mod}\, \K \]
extends to a bitriangulated bifunctor 
\[ F := \opn{Hom}_{\cat{M}}(-, -) : \dcat{K}^{}(\cat{M})^{\mrm{op}}
\times \dcat{K}^{}(\cat{M}) \to  
\dcat{K}^{}(\cat{Mod}\, \K) , \]
using the product totalization. Namely 
\[ F(M, N)^i = \opn{Hom}_{\cat{M}}(M, N)^i = \prod_{j \in \Z}\ 
\opn{Hom}_{\cat{M}}(M^j, N^{i+j}) \in \cat{Mod}\, \K . \]
The boundedness is a bit confusing; for instance, if 
$M \in \dcat{K}^{-}(\cat{M})$ and 
$N \in \dcat{K}^{+}(\cat{M})$, then 
$\opn{Hom}_{\cat{M}}(M, N) \in \dcat{K}^{+}(\cat{M})$.

\begin{lem} \mbox{}
\begin{enumerate}
\item Suppose $\phi : M \to M'$ is a quasi-isomorphism in 
$\dcat{K}^{}(\cat{Mod}\, A)$, and 
$\psi : I \to I'$ is a quasi-isomorphism in 
$\dcat{K}^{}(\cat{Mod}\, A)_{\mrm{inj}}$. 
Then 
\[ \opn{Hom}(\phi, \psi) : 
\opn{Hom}_A(M, I') \to \opn{Hom}_A(M', I)  \]
is a quasi-isomorphism. 

\item Suppose $\phi : P \to P'$ is a quasi-isomorphism in 
$\dcat{K}^{}(\cat{Mod}\, A)_{\mrm{proj}}$, and 
$\psi : N \to N'$ is a quasi-isomorphism in 
$\dcat{K}^{}(\cat{Mod}\, A)$. 
Then 
\[ \opn{Hom}(\phi, \psi) : 
\opn{Hom}_A(P, N') \to \opn{Hom}_A(P', N)  \]
is a quasi-isomorphism.
\end{enumerate}
\end{lem}

\begin{proof}
This is like the proof of Lemma \ref{lem:140}(3).
\end{proof}

\begin{thm} \label{thm:141}
Assume 
$\dcat{K}^{\star}(\cat{M})$ has enough K-projectives, or
$\dcat{K}^{\dagger}(\cat{M})$ has enough K-injectives, for 
$\star \in \{ - , \tup{\textvisiblespace} \}$ and
$\dag \in \{ + , \tup{\textvisiblespace} \}$.
Let $F := \opn{Hom}_{\cat{M}}(-,-)$. 
Then the right derived functor 
$(\mrm{R}^{\star, \dagger} F,\, \eta)$ exists, and we denote it by 
$\opn{RHom}_{\cat{M}}(-, -)$.
If either $M$ and $M'$ are K-projective, or $N$ and $N'$ are
K-injective, then 
\[ \eta : \opn{Hom}_{\cat{M}}(M, N) \to \opn{RHom}_{\cat{M}}(M, N) \]
is an isomorphism.
\end{thm}

The proof is like that of  Theorems \ref{thm:101} and
\ref{thm:121}.


\begin{dfn}
For $M, N \in \dcat{D}^{}(\cat{M})$ we write
\[ \opn{Ext}^i_{\cat{M}}(M, N) := \opn{H}^i ( \opn{RHom}_{\cat{M}}(M, N) ) .
\]
\end{dfn}

\begin{thm} \label{thm:143}
In the situation of Theorem \tup{\ref{thm:141}}, there is a bifunctorial
isomorphism 
\[ \opn{Ext}^i_{\cat{M}}(M, N) \cong 
\opn{Hom}_{\dcat{D}^{}(\cat{M})}(M, N[i])  \]
for $M \in \dcat{D}^{\dagger}(\cat{M})$ and $N \in \dcat{D}^{\star}(\cat{M})$.
\end{thm}

\begin{proof}
We will prove the K-injective case -- the K-projective case is similar. Let
$\ze_N : N \to I(N)$ be a K-injective
resolution, which can be made functorial (see proof of Theorem \ref{thm:121}).
We get bifunctorial isomorphisms of abelian groups
\[
\begin{aligned}
& \opn{Ext}^i_{\cat{M}}(M, N) \xar{\opn{Ext}^i_{\cat{M}}(1_M, \ze_N)}
\opn{Ext}^i_{\cat{M}}(M, I) \cong 
\mrm{H}^i (\opn{Hom}_{\cat{M}}(M, I)) 
\\
& \quad \cong \mrm{H}^0 (\opn{Hom}_{\cat{M}}(M, I[i])) =
\opn{Hom}_{\dcat{K}^{}(\cat{M})}(M, I[i]) 
\\
& \quad \xar{Q}
\opn{Hom}_{\dcat{D}^{}(\cat{M})}(M, I[i]) 
\xar{\opn{Hom}_{\cat{M}}(1_M, \ze_{N[i]}^{-1})}
\opn{Hom}_{\dcat{D}^{}(\cat{M})}(M, N[i]) .
\end{aligned} \]
\end{proof}

\begin{exer}
Prove that there is a trifunctorial morphism 
\[ \opn{Ext}^{i_1}_{\cat{M}}(M_0, M_1) \times
\opn{Ext}^{i_2}_{\cat{M}}(M_1, M_2) \to
\opn{Ext}^{i_1 + i_2}_{\cat{M}}(M_0, M_2) . \]
This is the cup product.

On the other hand, composition of morphisms in $\dcat{D}^{}(\cat{M})$ give us 
\[ \opn{Hom}_{\dcat{D}^{}(\cat{M})}(M_0, M_1[i_1]) \times 
\opn{Hom}_{\dcat{D}^{}(\cat{M})}(M_1, M_2[i_2]) \to
\opn{Hom}_{\dcat{D}^{}(\cat{M})}(M_0, M_2[i_1 + i_2]) . \]
Prove that this composition agrees with the cup product via the isomorphisms of
Theorem \ref{thm:143}.
\end{exer}

\begin{prop} \label{prop:145}
Assume $A$ is a noetherian commutative ring. If 
$M \in \dcat{D}^{-}_{\mrm{f}}(\cat{Mod}\, A)$
and 
$N \in \dcat{D}^{+}_{\mrm{f}}(\cat{Mod}\, A)$,
then $\opn{RHom}_{\cat{M}}(M, N) \in \dcat{D}^{+}_{\mrm{f}}(\cat{Mod}\, A)$.
\end{prop}

\begin{proof}
Exercise. 
\end{proof}

\begin{exer}
Find a counterexample for the boundedness conditions: 
$M, N \in \lb \dcat{D}^{}_{\mrm{f}}(\cat{Mod}\, A)$ s.t.\ 
$\opn{H}^i ( \opn{RHom}_{\cat{M}}(M, N))$ is not finitely generated for some
$i$. 
\end{exer}

We end this section with an adjunction formula. We state it for commutative
rings for convenience. 

\begin{prop} \label{prop:144}
Consider a homomorphism $A \to B$ of commutative rings, and complexes 
$L \in \dcat{D}^{}_{}(\cat{Mod}\, A)$ and 
$M, N \in \dcat{D}^{}_{}(\cat{Mod}\, B)$. 
There is an isomorphism 
\[ \opn{RHom}_{B}(N, \opn{RHom}_{A}(M, L)) \cong
\opn{RHom}_{A}( M \ot^{\mrm{L}}_B N, L) \]
in $\dcat{D}^{}_{}(\cat{Mod}\, B)$, functorial in $L, M , N$.
\end{prop}

\begin{proof}
Choose a K-injective resolution $L \to I$ over $A$, and a K-projective
resolution $P \to N$ over $B$. 
Then 
\[ \opn{RHom}_{B}(N, \opn{RHom}_{A}(M, L)) \cong
\opn{Hom}_{B}(P, \opn{Hom}_{A}(M, I)) \]
and 
\[ \opn{RHom}_{A}( M \ot^{\mrm{L}}_B N, L) \cong 
\opn{Hom}_{A}( M \ot^{}_B P, I) . \]
These are functorial isomorphisms in $\dcat{D}^{}_{}(\cat{Mod}\, B)$.
But the usual adjunction formula for modules tells us that 
\[ \opn{Hom}_{B}(P, \opn{Hom}_{A}(M, I)) \cong 
\opn{Hom}_{A}( M \ot^{}_B P, I) \]
in $\dcat{C}^{}_{}(\cat{Mod}\, B)$.
\end{proof}

\cleardoublepage
\section{Dualizing Complexes}

\subsection{Definition, duality} \mbox{}
Here $A$ is a noetherian commutative ring. 

\begin{dfn}
A complex $N \in \dcat{D}^{\mrm{b}}(\cat{Mod}\, A)$ has {\em finite injective
dimension} if 
there is a natural number $i_0$ such that 
$\opn{Ext}^i_A(M, N) = 0 $
for all $M \in \cat{Mod}\, A$ and $i > i_0$.
\end{dfn}

\begin{prop} 
The following are equivalent for $N \in \dcat{D}^{\mrm{b}}(\cat{Mod}\, A)$:
\begin{enumerate}
\rmitem{i} $M$ has finite injective dimension.
\rmitem{ii} There is a quasi-isomorphism $N \to I$ with $I$ a bounded
complex of injective $A$-modules. 
\end{enumerate}
\end{prop}

\begin{proof}
Like Proposition \ref{prop:141}.
\end{proof}

\begin{dfn}
A {\em dualizing complex} over $A$ is a complex 
$R \in \dcat{D}^{\mrm{b}}_{\mrm{f}}(\cat{Mod}\, A)$ with  these properties:
\begin{enumerate}
\rmitem{i} $R$ has finite injective dimension.
\rmitem{ii} The canonical morphism 
$\th : A \to \opn{RHom}_{A}(R, R)$
is an isomorphism. 
\end{enumerate}
\end{dfn}

The canonical morphism 
$\th : A \to \opn{RHom}_{A}(R, R)$
can be realized as follows: choose a bounded injective resolution $R \to I$. 
Then 
\[ \opn{RHom}_{A}(R, R) \cong \opn{Hom}_{A}(I, I) = \opn{End}_{A}(I) . \]
Now $\opn{End}_{A}(I)$ is a ring (in fact a DG algebra), and the morphism 
$\th : A \to \opn{End}_{A}(I)$ is the ring homomorphism.


\begin{exa} \label{exa:141}
Suppose $A$ is a regular ring. Then the complex 
$R := A$ has finite injective dimension. Since $R$ is K-projective, the
morphism 
\[ \eta_R : \opn{Hom}_{A}(R, R) \to \opn{RHom}_{A}(R, R) \]
is an isomorphism. But 
\[ \opn{Hom}_{A}(R, R) = \opn{Hom}_{A}(A, A) = \opn{End}_{A}(A) = A , \]
and canonical morphism $\th$ is represented by the identity map of $A$, so it
is an isomorphism. We see that $R = A$ is a dualizing complex. 
\end{exa}


Given a dualizing complex $R$, let 
\begin{equation} \label{eqn:141}
D_R := \opn{RHom}_{A}(-, R) . 
\end{equation}
This is a triangulated functor 
\[ D_R : \dcat{D}^{\mrm{b}}(\cat{Mod}\, A)^{\mrm{op}} \to 
\dcat{D}^{\mrm{b}}(\cat{Mod}\, A) . \]
If we fix a bounded injective resolution $R \to I$, then we get an isomorphism
of triangulated functors  
\[  D_R \cong \opn{Hom}_{A}(-, I) . \]

As in every ``duality situation'', here there is a morphism
\begin{equation} \label{eqn:143}
\th : \bsym{1}_{\dcat{D}^{}(\cat{Mod}\, A)} \to D_R \o D_R 
\end{equation}
 of triangulated functors from $\dcat{D}^{}(\cat{Mod}\, A)$ to itself.  
In terms of the injective resolution $M \to I$, for any complex $M$ the
morphism 
\[ \th_M : M \to (D_R \o D_R)(M) \cong 
\opn{Hom}_A ( \opn{Hom}_A ( M, I ) , I ) \] 
is the obvious one; namely for $m \in M^k$ and 
$\phi \in \opn{Hom}_A ( M, I )^l$, 
\[ \th(m)(\phi) = (-1)^{k+l} \phi(m) \in I^{k+l} . \]

\begin{thm} \label{thm:150}
Let $R$ be a dualizing complex over $A$. 
\begin{enumerate}
\item Let $M \in \dcat{D}^{\mrm{b}}_{\mrm{f}}(\cat{Mod}\, A)$. Then the
complex $D_R(M)$ is in $\dcat{D}^{\mrm{b}}_{\mrm{f}}(\cat{Mod}\, A)$, and the
morphism 
\[ \th_M : M \to (D_R \o D_R)(M) \]
is an isomorphism. 

\item The functor 
\[ D_R : \dcat{D}^{\mrm{b}}_{\mrm{f}}(\cat{Mod}\, A)^{\mrm{op}} \to 
\dcat{D}^{\mrm{b}}_{\mrm{f}}(\cat{Mod}\, A) \]
is a duality \tup{(}i.e.\ a contravariant equivalence\tup{)} 
of triangulated categories.
\end{enumerate}
\end{thm}

\begin{proof}
(1) We fix a bounded injective resolution $R \to I$. Say $I$ is concentrated in
degrees $\{ l_0, \ldots, l_1 \}$. Since $M$ and $R$ are bounded, so is 
$\opn{Hom}_A ( M, I ) \cong D_R(M)$. 

Next we will prove that $D_R(M)$ has finitely generated cohomology modules. 
Let us choose a resolution $P \to M$, where $P$ is a bounded above complex
of finitely generated free modules. So 
$D_R(M) \cong \opn{Hom}_A (P, I)$.  We will prove that 
$\opn{H}^k ( \opn{Hom}_A (P, I) )$ is a finitely generated module for every
$k$. Let us choose some $k$. We already saw that the module 
$\opn{H}^k ( \opn{Hom}_A (P, I) )$ only depends on the stupid truncation 
\[ (\opn{stt}^{\leq l_1 - k +2} \o \opn{stt}^{\geq l_0 - k -2})(P) \]
of $P$. Thus we can assume that $P$ is a bounded complex of finitely generated
free modules. 

It remains to prove that for any bounded complex $P$ 
of finitely generated free modules, the complex $\opn{Hom}_A (P, I)$ has
finitely generated cohomology modules. Say $P$ is concentrated in degrees 
$\{ k_0, \ldots, k_1 \}$. The proof is by induction on $k_1 - k_0$. 
If $k_1 = k_0$ then we get $P = P^{k_1}[-k_1]$, and $P^{k_1}$ is a
free module of finite rank, say rank $r$. But then 
\[ \opn{Hom}_A (P, I) \cong I[-k_1]^{\oplus r} \]
and 
\[ \opn{H}^k ( \opn{Hom}_A (P, I) ) \cong \opn{H}^{k - k_1}(I)^{\oplus r} , \]
which is finitely generated by assumption.

If $k_1 > k_0$, then using stupid truncation of $P$ we get
a short exact sequence of complexes of free modules, and hence a distinguished
triangle in  $\dcat{D}(\cat{Mod}\, A)$:
\[ P' \to P \to P'' \to P'[1] \]
where 
$P' := P^{k_1}[-k_1]$
and 
$P'' := \opn{stt}^{\leq k_1 - 1}(P)$.
{}From this we get a distinguished
triangle in  $\dcat{D}(\cat{Mod}\, A)$:
\[ D_I( P'' ) \to D_I(P) \to D_I( P' ) \to D_I(P'')[1] . \]
The induction hypothesis is that $D_I( P'' )$ and $D_I( P')$ have f.g.\ 
cohomologies; and therefore the same is true for $P$. 

Finally we have to prove that $\th_M$ is an isomorphism. As above we choose a
resolution $P \to M$, and we will prove that 
\[ \th_P : P \to (D_I \o D_I)(P) = 
\opn{Hom}_A ( \opn{Hom}_A ( P, I ) , I ) \]
is a quasi-isomorphism. In order to show that 
\[ \opn{H}^k(\th_P) : \opn{H}^k(P) \to 
\opn{H}^k( \opn{Hom}_A ( \opn{Hom}_A ( P, I ) , I )) \]
is an isomorphism, we can safely replace $P$ with a suitable stupid truncation
(depending on $k, k_0, k_1$). Thus, as above, we can assume that $P$ is a
bounded complex of f.g.\ free modules. By stupid truncation and induction, also
as above, we can assume that $P$ is a single f.g.\ free module, say 
$P \cong A^{\oplus r}$. Then, by additivity, we can assume that $r = 1$; but 
we know that  $\th_A$ is a quasi-isomorphism.

\medskip \noindent
(2) This is clear from part (1).
\end{proof}

\begin{exa}
Let's see what happens in the trivial case when $A = \K$ is a field.
We take the dualizing complex $R := \K$. 

Any complex $M \in \dcat{D}^{}(\cat{Mod}\, \K)$ is {\em formal}, namely it is
isomorphic to its cohomology $\opn{H}(M)$, that is viewed as a complex with
zero differential. As for morphisms: given $M, N$ with zero differentials,
morphisms in $\dcat{D}^{}(\cat{Mod}\, \K)$ are just homomorphisms of graded
$\K$-modules. 

A complex $M \in  \dcat{D}^{\mrm{b}}_{\mrm{f}}(\cat{Mod}\, K)$ is isomorphic to
its cohomology $N := \opn{H}(M)$, which is a bounded complex of finite rank
$\K$-modules, with zero differential. Thus 
\[ D_{\K}(M) = \opn{RHom}_{\K}(M, \K)  \cong D_{\K}(N) = \opn{Hom}_{\K}(N, \K) =
D(N) , \]
where $D(N)$ is our notation for the ``classical dual''. Now $D(N)$ is
also bounded complex of finite rank
$\K$-modules, with zero differential. And the canonical homomorphism 
$\th : N \to D(D(N))$ is an isomorphism.

Note that $\th : M \to D(D(M))$ is only a quasi-isomorphism, but doesn't have
to be an isomorphism. (Can you find an example?). 
\end{exa}

\begin{exa}
Take $A := \Z$ and $R := \Z$. This is a dualizing complex, but since $R$ is
not K-injective, we have to resolve it. The best think to do is to take the
minimal injective resolution $R \to I$, where 
\[ I := \bigl( \cdots \to 0 \to \mbb{Q} \to \mbb{Q} / \Z \to 0 \to \cdots \bigl)
 \]
is concentrated in degrees $0, 1$. 

Take any $M \in \cat{Mod}_{\mrm{f}}\, \Z$, i.e.\ a finitely generated abelian
group. Then 
\[ D_{\Z}(M) = \opn{RHom}_{\Z}(M, \Z)  \cong D_{I}(M) = 
\opn{Hom}_{\Z}(M, I) .  \]
The complex $D_I(M)$ is also concentrated in degrees $0, 1$. 
Then the complex $D_I(D_I(M))$ is  concentrated in degrees $-1, 0, 1$. 
It is not so easy to see that 
\[ \th : M \to D_I(D_I(M)) \]
is a quasi-isomorphism! We can either rely on Theorem \ref{thm:150}, or we can
use a trick. 

Choose a (noncanonical) decomposition
$M \cong G \oplus H$, with $G$ free and $H$ finite. 
For $G$ consider its classical dual $D(G) := \opn{Hom}_{\Z}(G, Z)$. 
The homomorphism $D(G) \to D_I(G)$ is a quasi-isomorphism. Hence 
$D_I(D_I(G)) \to D_I(D(G))$ and  
$D(D(G)) \to D_I(D(G))$ are quasi-isomorphisms; and they are compatible with the
$\th$'s from $G$. We know that $\th : G \to D(D(G))$ is an isomorphism.

As for the finite group $H$, we have 
$D_I(H) \cong D'(H)[-1]$, where $D'(H) := \lb \opn{Hom}_{\Z}(H, \mbb{Q} / \Z)$.
Thus $D_I(D_I(H)) \cong D'(D'(H)[-1])[-1] \cong H$. 
\end{exa}

\subsection{Existence} \mbox{}

Let $A \to B$ be a ring homomorphism. For any 
$M \in \cat{Mod}\, A$ we view $\opn{Hom}_A(B, M)$ as an object of 
$\cat{Mod}\, B$ in the obvious way. We get an additive functor 
\[ \opn{Hom}_A(B, -) : \cat{Mod}\, A \to \cat{Mod}\, B . \]
This has a right derived functor 
\[ \opn{RHom}_A(B, -) : \dcat{D}(\cat{Mod}\, A) \to 
\dcat{D}(\cat{Mod}\, B) ,  \]
which is calculated using K-injective resolutions. 
There is no ambiguity in this notation, since when we forget the $B$-module
structure we get the complex 
$\opn{RHom}_A(B, M) \in \dcat{D}(\cat{Mod}\, A)$ that we had before (Theorem
\ref{thm:141}).

A ring homomorphism $A \to B$ is called {\em finite} if it makes $B$ into a
finitely generated $A$-module. 

\begin{prop} \label{prop:147}
Let $A \to B$ be a finite ring homomorphism, and let 
$R_A \in \dcat{D}^{\mrm{b}}_{\mrm{f}}(\cat{Mod}\, A)$ be a dualizing complex.
Then the complex 
\[ R_B :=  \opn{RHom}_A(B, R_A) \in  \dcat{D}^{\mrm{b}}_{\mrm{f}}(\cat{Mod}\, A)
\]
is a dualizing complex over $B$. 
\end{prop}

\begin{proof}
By Proposition \ref{prop:145} the complex $R_B$ has finitely generated
cohomology modules (over $A$ and hence over $B$).

We note that if $I$ is an injective $A$-module, then $\opn{Hom}_A(B, I)$ is an
injective $B$-module. 
Let us choose a bounded injective resolution 
$R_A \to I$ over $A$. Then 
$R_B \to \opn{Hom}_A(B, I)$
is a bounded injective resolution of $R_B$ over $B$. 
This proves that $R_B$ is bounded and has finite injective dimension. 


Next we calculate
$\th : B \to \opn{RHom}_B(R_B, R_B)$. We must show that 
\[ \th : B \to \opn{Hom}_B \bigl( \opn{Hom}_A(B, I), \opn{Hom}_A(B, I)
\bigr) \]
is a quasi-isomorphism. By adjunction we have isomorphisms
\[ \begin{aligned}
& \opn{Hom}_B \bigl( \opn{Hom}_A(B, I), \opn{Hom}_A(B, I) \bigr) \cong 
\opn{Hom}_A \bigl( B \ot_B \opn{Hom}_A(B, I), I \bigr) 
\\
& \quad \cong 
\opn{Hom}_A \bigl( \opn{Hom}_A(B, I), I \bigr) 
\end{aligned} \]
that are compatible with $\th$. Because $I \cong R_A$ is dualizing, and $B$ is
a finitely generated $A$-module, 
\[ \th : B \to \opn{Hom}_A \bigl( \opn{Hom}_A(B, I), I \bigr) \]
is a quasi-isomorphism.
\end{proof}

\begin{lem} \label{lem:145}
Let $L, M, N \in \dcat{D}^{\mrm{b}}_{}(\cat{Mod}\, A)$ s.t.\ $L$ has finitely
generated cohomologies and $N$ has finite flat dimension. Then 
\[  \opn{RHom}_A (L, M \ot^{\mrm{L}}_A N) \cong 
\opn{RHom}_A (L, M)  \ot^{\mrm{L}}_A N , \]
and this isomorphism is functorial.
\end{lem}

\begin{proof}
Choose resolutions $P \to L$ and $Q \to N$, where $P$ is a bounded above
complex of finitely generated free modules, and $Q$ is a bounded complex of
flat modules. 
Then there are functorial isomorphisms 
\[ \opn{RHom}_A (L, M \ot^{\mrm{L}}_A N) \cong 
\opn{Hom}_A (P, M \ot^{}_A Q) \]
and 
\[ \opn{RHom}_A (L, M)  \ot^{\mrm{L}}_A N \cong 
\opn{Hom}_A (P, M) \ot^{}_A Q . \]
Because of the finiteness of the free modules $P^i$, we see that the canonical
homomorphism 
\[ \opn{Hom}_A (P, M) \ot^{}_A Q \to 
\opn{Hom}_A (P, M \ot^{}_A Q) \]
is an isomorphism in $\dcat{D}^{}_{}(\cat{Mod}\, A)$.
\end{proof}

\begin{lem} \label{lem:146}
A complex $M \in \dcat{D}^{\mrm{b}}(\cat{Mod}\, A)$ has finite injective
dimension iff there is a natural number $i_0$ such that 
$\opn{Ext}^i_A(M, N) = 0$
for every {\em cyclic} $A$-module $M$ and every $i > i_0$.
\end{lem}

This lemma does not require $A$ to be noetherian.

\begin{proof}
Exercise. (Hint: use the fact that a module $I$ is injective iff 
$\opn{Ext}^1_A(M, I) = 0$ for every cyclic $A$-module $M$.
Cf.\ the Baer criterion Theorem \ref{thm:151}.)
\end{proof}

An $A$-algebra $B$ is called a {\em localization} of $A$ if 
$B \cong A_S = A[S^{-1}]$ for some multiplicatively closed set $S \subset A$. 

\begin{prop} \label{prop:148}
Let $B$ be a localization of $A$, and let $R_A$ be a dualizing complex over
$A$. Then the complex $R_B := B \ot_A R_A$ is a dualizing complex over $B$.
\end{prop}

\begin{proof}
Clearly $R_B$ is bounded, Since $B$ is a flat $A$-module, we have
$\opn{H}^k(R_B) \cong B \ot_A \opn{H}^k(R_A)$, so these are finitely generated
$B$-modules. 

I will give two proofs that $R_B$ has finite injective dimension. First, let's
choose a bounded injective resolution $R_A \to I$. By the structure theory for
injective modules over noetherian rings, each $I^k$ is a direct sum of
indecomposable injective $A$-modules. Each such indecomposable module is of the
form $J(\mfrak{p})$ for a prime ideal $\mfrak{p}$. Let $S$ be such that 
$B \cong A_S$. If 
$\mfrak{p} \cap S \neq \emptyset$ then $B \ot_A J(\mfrak{p}) = 0$; and if 
$\mfrak{p} \cap S = \emptyset$ then 
$B \ot_A J(\mfrak{p}) \cong J(\mfrak{p})$, which is an indecomposable injective
$B$-module. Thus $B \ot_A I$ is a bounded complex of injective
$B$-modules,  and $R_B \cong B \ot_A I$.

Here is another proof. By Lemma \ref{lem:146} it suffices to prove the vanishing
of \lb $\opn{Ext}^i_A(N, R_B)$ for cyclic $B$-modules $N$ and for $i \gg 0$. 
Choose a f.g.\ $A$-submodule $M \subset N$ s.t.\ $B M = N$; then, by flatness
of $B$, the homomorphism $B \ot_A M \to N$ is bijective. 
Now 
\begin{equation} \label{eqn:147}
\begin{aligned}
& \opn{RHom}_B (N, R_B) \cong 
\opn{RHom}_B (B \ot_A M, B \ot_A R_A) 
\\
& \quad \cong
\opn{RHom}_A (M, B \ot_A R_A) \cong 
\opn{RHom}_A (M, R_A) \ot_A B  .
\end{aligned} 
\end{equation}
So 
\[ \opn{Ext}^i_A(N, R_B) \cong \opn{Ext}^i_A (M, R_A) \ot_A B , \]
and the latter vanishes for $i \gg 0$ (independent of $M$). 

Finally we need to prove that 
\[ \th_B : B \to \opn{RHom}_B (R_B, R_B) \]
is an isomorphism. Taking $N := R_B$ in (\ref{eqn:147})
we get 
\[ \opn{RHom}_B (R_B, R_B) \cong \opn{RHom}_A (R_A, R_A) \ot_A B , \]
 and $\th_B$ corresponds in this isomorphism to $\th_A \ot 1_B$. 
Since $\th_A$ is an isomorphism, so is $\th_B$. 
\end{proof}

A ring homomorphism $A \to B$ is called {\em finite type} if $B$ is finitely
generated as $A$-algebra. A homomorphism $A \to B$ is called {\em essentially
finite type} if $B$ is a localization of a finite type $A$-algebra. 

\begin{cor}
Suppose $A$ is a regular ring, and $B$ is an essentially finite
type $A$-algebra. Then $B$ has a dualizing complex. 
\end{cor}

\begin{proof}
Let $C$ be a finite type $A$-algebra s.t.\ $B$ is a localization of $C$. 
Let $D := A[t_1, \ldots, t_n]$ be a polynomial ring s.t.\ $C$ is a quotient of
$C$. Now $D$ is a regular ring, so it has a dualizing complex
(see Example \ref{exa:141}). By Proposition
\ref{prop:147} the ring $C$ has a dualizing complex, and by Proposition
\ref{prop:148} the ring $B$ has a dualizing complex.
\end{proof}

\begin{rem}
Actually there is a more general existence result: if $A$ has a dualizing
complex, and $B$ is an essentially finite type $A$-algebra, then $B$ has a
dualizing complex. 
\end{rem}


\subsection{Uniqueness of dualizing complexes}
As before $A$ is a noetherian ring. 

Let $P$ be a finitely generated projective module. For a prime ideal 
$\mfrak{p}$ the module $A_{\mfrak{p}}$-module $P_{\mfrak{p}}$ is free, say of
rank $r(\mfrak{p})$. We get a function $r : \opn{Spec} A \to \N$, and this
function is locally constant. Thus, if $\opn{Spec} A$ is connected, then $r$ is
constant. (If $\opn{Spec} A$ is disconnected, i.e.\ 
$A \cong A_1 \times A_2$ with both factors nonzero, then the rank
can change between connected components.) 

\begin{dfn}
An $A$-module $L$ is called {\em invertible} if there is an $A$-module
$L^{\vee}$ such that $L \ot_A L^{\vee} \cong A$.
\end{dfn}

\begin{prop} \label{prop:150}
The following are equivalent for an $A$-module $L$:
\begin{enumerate}
\rmitem{i} $L$ is invertible. 

\rmitem{ii} $L$ is a finitely generated projective $A$-module of rank $1$. 
\end{enumerate}
\end{prop}

\begin{proof}
Exercise. (Hint: Morita Theorem and Nakayama Lemma. You will see that 
$L^{\vee} \cong \opn{Hom}_A(L, A)$; it is called the inverse of $L$.)
\end{proof}

The Picard group $\opn{Pic}(A)$ is the group of isomorphism classes
of invertible $A$-modules. The operation is tensor product.
 
\begin{thm} \label{thm:153}
Assume $\opn{Spec} A$ is connected. 
Let $R, R' \in \dcat{D}^{\mrm{b}}_{\mrm{f}}(\cat{Mod}\, A)$, with
$R$ a dualizing complex. The following are equivalent:
\begin{enumerate}
\rmitem{i} $R'$ is a dualizing complex. 
\rmitem{ii} $R' \cong R \ot_A L[n]$ for some invertible module $L$ and some
integer $n$.
\end{enumerate}
\end{thm}

First some lemmas. In these lemmas we assume that $R'$ is a dualizing complex,
and that $\opn{Spec} A$ is connected. 
Define $D := \opn{RHom}_A(-, R)$, $D' := \opn{RHom}_A(-, R')$ and 
\begin{equation} \label{eqn:150}
 M := \opn{RHom}_A(R, R') \cong D'(D(A)) 
 \in \dcat{D}^{\mrm{b}}_{\mrm{f}}(\cat{Mod}\, A) .
\end{equation}

\begin{lem} \label{lem:150}
There is a functorial isomorphism 
\[ \al_N :  M \ot^{\mrm{L}}_A N \to (D' \o D)(N) \]
for $N \in  \dcat{D}^{-}_{\mrm{f}}(\cat{Mod}\, A)$.
\end{lem}

\begin{proof}
Let's choose bounded injective resolutions $R \to I$ and $R' \to I'$. So 
$D \cong \lb \opn{Hom}_A(-, I)$, $D' \cong \opn{Hom}_A(-, I')$ and 
$M \cong \opn{Hom}_A(I, I')$. Given $N$, let's choose a resolution 
$P \to N$ where $P$ is a bounded above complex of f.g.\ free modules. 
Then there is an obvious homomorphism of complexes
\[ \al_P : \opn{Hom}_A(I, I') \ot_A P \to 
\opn{Hom}_A (\opn{Hom}_A (P, I), I') . \]
We want to show that $\al$ is a quasi-isomorphism.

To show that $\opn{H}^k(\al_P)$ is an isomorphism for any particular $k \in \Z$,
we can replace $P$ with a stupid truncation 
$\opn{stt}^{\geq l}(P)$
for sufficiently small $l$, depending on $k$ of course, as was done in previous
proofs. But $\opn{stt}^{\geq l}(P)$ is bounded. So we reduce to the case of a
bounded complex f.g.\ free modules $P$. Then, by the
distinguished triangles gotten from truncation of $P$, we reduce to the case of
a free module $P$. This reduces to the case $P = A$, in which $\al_P$ is
bijective. 
\end{proof}

\begin{lem}[Kunneth trick] \label{lem:152} 
Given $M, M' \in \dcat{D}^{-}_{}(\cat{Mod}\, A)$, let 
\[ k := \sup\, \{ l \mid \opn{H}^l(M) \neq 0 \} \]
and 
\[ k' := \sup\, \{ l \mid \opn{H}^l(M') \neq 0 \} , \]
which belong to $\Z \cup \{ -\infty\}$. Then 
$\opn{H}^{l}(M \ot^{\mrm{L}}_A M') = 0$ for $l > k + k'$, and 
\[ \opn{H}^k(M) \ot_A \opn{H}^{k'}(M') \cong 
\opn{H}^{k+k'}(M \ot^{\mrm{L}}_A M') . \]
\end{lem}

\begin{proof}
This can be seen from the Kunneth spectral sequence. Or, in a more elementary
way, this can be seen by by choosing free resolutions $P \to M$ and $P' \to M'$,
s.t.\ $P$ is concentrated in degrees $\leq k$, and $P'$ is concentrated in
degrees $\leq k'$. Then $P \ot_A P'$ is concentrated in
degrees $\leq k + k'$. An easy calculation, using right exactness of $\ot$,
shows that 
\[ \opn{H}^k(P) \ot_A \opn{H}^{k'}(P') \cong 
\opn{H}^{k+k'}(P \ot^{}_A P') . \]
\end{proof}

\begin{lem} \label{lem:151}
Suppose $M, M' \in \dcat{D}^{-}_{\mrm{f}}(\cat{Mod}\, A)$
satisfy $M \ot^{\mrm{L}}_A M' \cong A$. Then 
$M \cong L[-k]$ and $M' \cong L'[-k']$, for some invertible modules $L$ and $L'$
that satisfy $L \ot_A L' \cong A$, and some integers $k$ and $k'$ that satisfy 
$k + k' = 0$.
\end{lem}

\begin{proof}
First assume $A$ is a local ring. Take $k, k'$ as in Lemma \ref{lem:152}, and
let $L := \opn{H}^k(M)$ and $L' := \opn{H}^{k'}(M')$. These are nonzero f.g.\
$A$-modules, so by the Nakayama Lemma 
$L \ot_A L' \neq 0$. Therefore by Lemma \ref{lem:152} we see that 
$\opn{H}^{k+k'}(M \ot^{\mrm{L}}_A M') \neq 0$. But 
$M \ot^{\mrm{L}}_A M' \cong A$. We conclude that $k + k' = 0$ and
$L \ot_A L' \cong A$. 

By Proposition \ref{prop:150} the modules $L$ and $L'$
are projective. Let $P \to M$ and $P' \to M'$ be resolutions as in the proof of 
Lemma \ref{lem:152}. Consider the surjection $P^k \to \opn{H}^k(P) \cong L$; 
it can be split, and we get an isomorphism of complexes 
$P \cong Q \oplus L[-k]$, where $Q$ has cohomology concentrated in degrees 
$< k$. Similarly $P' \cong Q' \oplus L'[-k']$, where $Q'$ has cohomology
concentrated in degrees $< k'$. Now the complex 
\[ (L[-k] \ot_A L'[-k']) \oplus
(Q \ot_A L'[-k']) \oplus (L[-k] \ot_A Q') \oplus (Q \ot_A Q') \cong 
P \ot_A P' \]
has no cohomology in degrees $< 0$, whereas 
$Q \ot_A L'[-k']$ and $L[-k] \ot_A Q'$ have cohomology only in degrees $< 0$.
Therefore both these complexes are acyclic. Since $L$ and $L'$ are invertible, 
it follows that $Q$ and $Q'$ are also acyclic. What we need to pick up from
this is: $\opn{H}^k(M)$ is a free $A$-module of rank $1$, and 
$\opn{H}^l(M) = 0$ for all $l \neq k$. 

Now for the general case. For any $\p \in \opn{Spec} A$ the localized complexes 
$M_{\p} := A_{\p} \ot_A M$ and $M'_{\p} := A_{\p} \ot_A M'$ 
in $\dcat{D}^{-}_{\mrm{f}}(\cat{Mod}\, A_{\p})$ satisfy 
\[ M_{\p} \ot^{\mrm{L}}_{A_{\p}} M'_{\p} \cong 
A_{\p} \ot_A (M \ot^{\mrm{L}}_A M') \cong A_{\p} . \]
So there is an integer $k(\p)$ s.t.\ 
$\opn{H}^{k(\p)}(M_{\p}) \cong \opn{H}^{k(\p)}(M)_{\p}$ is a free
$A_{\p}$-module of rank $1$, and 
$\opn{H}^{l}(M_{\p}) \cong \opn{H}^{l}(M)_{\p} = 0$ for all $l \neq k(\p)$.

Because the $A$-module $\opn{H}^{k(\p)}(M)$ is finitely presented, 
and $\opn{H}^{k(\p)}(M)_{\p}$ is free of rank $1$, it follows
that there is an open neighborhood $U$ of $\p$ (in $\opn{Spec} A$, for the
Zariski topology) s.t.\ $\opn{H}^{k(\p)}(M)_{\q}$ is free of rank $1$ for all
$\q \in U$. This means that 
$k(\q) = k(\p)$ for all $\q \in U$. We see that the function 
$k : \opn{Spec} A \to \Z$ is locally constant. Since $\opn{Spec} A$ is
connected, this must be a constant function, i.e.\ there is some 
$k \in \Z$ s.t.\ $k(\p) = k$ for all $\p$. It follows that 
$\opn{H}^{l}(M) = 0$ for all $l \neq k$, the f.g.\ module
$L := \opn{H}^{k}(M)$ is projective of rank $1$, and by truncation 
$M \cong P \cong L[-k]$ in  $\dcat{D}(\cat{Mod}\, A)$.

Likewise there is $k' \in \Z$ s.t.\ $\opn{H}^{l}(M') = 0$ for all $l \neq k'$,
the f.g.\ module $L' := \opn{H}^{k'}(M')$ is projective of rank $1$, and
$M' \cong P' \cong L'[-k']$ in  $\dcat{D}(\cat{Mod}\, A)$.
Since $L[-k] \ot_A L'[-k'] \cong A$ 
we get $k + k' = 0$ and $L \ot_A L' \cong A$. 
\end{proof}

\begin{proof}[Proof of Theorem \tup{\ref{thm:153}}]
(ii) $\Rightarrow$ (i): This is the easy part. Clearly $R'$ is bounded. 
Its cohomologies are 
$\opn{H}^k(R') \cong \opn{H}^{k+n}(R) \ot_A L$
are finitely generated. If $R \to I$ is a bounded injective resolutions, then 
$R' \to (I \ot_A L)[n]$ is a bounded injective resolution. And the homomorphism
of complexes 
\[ \th :  A \to \opn{Hom}_A \bigl( (I \ot_A L)[n] , (I \ot_A L)[n] \bigr) \]
is a quasi-isomorphism, since 
\[ \opn{Hom}_A \bigl( (I \ot_A L)[n] , (I \ot_A L)[n] \bigr) \cong 
\opn{Hom}_A (I, I) \]
as complexes. 

(i) $\Rightarrow$ (ii): Consider the complex $M = \opn{RHom}_A(R, R')$ from
(\ref{eqn:150}). Similarly define $M' := \opn{RHom}_A(R', R)$. 
By Lemma \ref{lem:150} there are isomorphisms 
\[ M' \ot^{\mrm{L}}_A M \cong (D \o D')(M) \cong (D \o D' \o D' \o D)(A) . \]
But by Theorem \ref{thm:150} we know that 
$D \o D \cong \bsym{1} \cong D' \o D'$. Thus 
\[ M' \ot^{\mrm{L}}_A M \cong A . \]
Now by Lemma \ref{lem:151}, $M \cong L[n]$ for some $L$ and $n$. And
\[ R' \cong D'(A) \cong (D' \o D \o D)(A) \cong 
(D' \o D) (D(A)) \cong M \ot^{\mrm{L}}_A D(A) \cong L[n] \ot_A R .  \]
\end{proof}

\begin{cor}
Assume $A$ has at least one dualizing complex, and $\opn{Spec} A$ is connected.
Then the group 
$\opn{Pic}(A) \times \Z$ acts simply transitively on the set of isomorphism
classes of dualizing complexes.
\end{cor}

\begin{proof}
The action of $(L, n) \in \opn{Pic}(A) \times \Z$ on dualizing complexes
is $R \mapsto R \ot_A L[n]$. We have shown that this is a transitive action. 
If 
$R \ot_A L[n] \cong R \ot_A L'[n']$, 
then applying the functor $D = \opn{RHom}_A(-, R)$ we get 
$L[n] \cong L'[n']$, so $n = n'$ and $L \cong L'$.
\end{proof}

\begin{rem}
The is a notion of dualizing complex over noncommutative rings. 
See \cite{Ye1, VdB, Ye4, YZ1}. 
\end{rem}


\subsection{An example}
Consider the ring 
\[ A = \mbb{R}[t_1, t_2, t_3] / (t_3 t_1, t_3 t_2) . \]
This is the coordinate ring of an affine algebraic variety 
$X \subset \mbf{A}^3_{\mbb{R}}$
which is the union of a plane and a line, meeting at a point. See figure
\ref{fig:2a}. 

\begin{figure}
\includegraphics[scale=0.5]{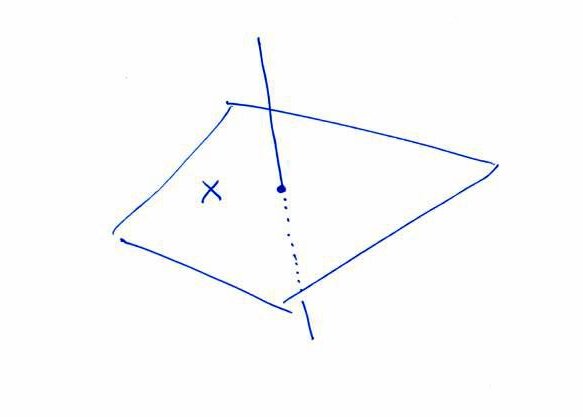}
\caption{} 
\label{fig:2a}
\end{figure}

We know that $A$ has a dualizing complex $R_A$, since it is a finite type
$\mbb{R}$-algebra. 
I will show that the dualizing complex $R_A$ must live in two adjacent
degrees; namely there is some $i$ s.t.\ 
$\mrm{H}^i(R_A)$ and $\mrm{H}^{i+1}(R_A)$ are nonzero. 

Let us denote by $Y_1$ the line in $X$; so 
$Y_1 = \opn{Spec} B_1$, where 
\[ B_1 = A / (t_1, t_2) \cong \mbb{R}[t_1, t_2, t_3] / (t_1 , t_2)
\cong \mbb{R}[t_3]  . \]
Let $Y_2$ the plane in $X$; so 
$Y_2 = \opn{Spec} B_2$, where 
\[ B_2 =  A / (t_3) \cong \mbb{R}[t_1, t_2, t_3] / (t_3) \cong  
\mbb{R}[t_1, t_2]  . \]
And let $Z$ be the point of intersection. So $Z =  \opn{Spec} C$, where 
\[ C = A / (t_1, t_2, t_3) \cong \mbb{R}[t_1, t_2, t_3] / (t_1 , t_2, t_3)
\cong \mbb{R}  . \]
Note that the rings $B_1, B_2, C$ are regular. 

{}From Proposition \ref{prop:147}  we know that 
\[ R_C := \opn{RHom}_A(C, R_A) \]
is a dualizing complex over $C$. Since $C$ is a field, we must have by 
Example \ref{exa:141} and Theorem \ref{thm:153} that 
$R_C \cong C[i]$ in $\dcat{D}^{}(\cat{Mod}\, C)$ for some $i$. Let us shift
$R_A$ by $-i$, so that now $R_C \cong C$. 

Next let's look at the complex 
\[ R_{B_1} := \opn{RHom}_A(B_1, R_A) \]
is a dualizing complex over $B_1$. Since $B_1$ is a regular ring, we must
have  that 
$R_{B_1} \cong L[i_1]$ in $\dcat{D}^{}(\cat{Mod}\, B_1)$ for some $i_1$ and some
invertible module $L$. But actually the Picard group of 
$B_1 \cong \mbb{R}[t_3]$
is trivial, so $R_{B_1} \cong B_1[i_1]$.

We know by one of the adjunction formulas that 
\[ \begin{aligned}
& C \cong R_C = \opn{RHom}_A(C, R_A) \cong 
 \opn{RHom}_{B_1}(C, \opn{RHom}_A(B_1, R_A)) 
\\
& \quad \cong 
\opn{RHom}_{B_1}(C, R_{B_1}) \cong \opn{RHom}_{B_1}(C, B_1[i_1]) .
\end{aligned} \]
But on the other hand we know that 
$\opn{Ext}^j_{B_1}(C, B_1) \cong C$ for $j = 1$, and 
$\opn{Ext}^j_{B_1}(C, B_1) = 0$ for $j \neq 1$; so
\[ \opn{RHom}_{B_1}(C, B_1) \cong C[-1] . \]
By combining these equations we obtain 
\[ C \cong \opn{RHom}_{B_1}(C, B_1[i_1])  
\cong \opn{RHom}_{B_1}(C, B_1)[i_1] \cong C[-1 + i_1] . \]
The conclusion is that $i_1 = 1$. 

Similarly we show that $i_2 = 2$.

The next step is to localize away from the singular locus, namely away from
$Z$. For the line $Y_1$ this means inverting $t_3$. On the level of rings we
have
\[ B_1[t_3^{-1}] \cong A[t_3^{-1}] \]
as $A$-algebras. The dualizing complexes satisfy 
\[  
\begin{aligned}
& B_1[t_3^{-1}] \ot_{B_1} R_{B_1} = 
B_1[t_3^{-1}] \ot_{B_1}  \opn{RHom}_A(B_1, R_A) 
\\
& \quad \cong A[t_3^{-1}] \ot_{A}  \opn{RHom}_A(B_1, R_A) 
\\ & \quad 
\cong  \opn{RHom}_A(B_1, A[t_3^{-1}] \ot_{A}   R_A) 
\\ & \quad 
\cong  \opn{RHom}_{A[t_3^{-1}]}( A[t_3^{-1}] \ot_{A}  B_1, A[t_3^{-1}] \ot_{A}  
R_A) 
\\ & \quad 
\cong   A[t_3^{-1}] \ot_{A}  R_A .
\end{aligned} \]
But $R_{B_1} \cong B_1[1]$. Therefore 
\[ A[t_3^{-1}][1] \cong A[t_3^{-1}] \ot_{A}  R_A . \]
We conclude that 
\[  A[t_3^{-1}] \ot_{A}  \opn{H}^{-1}(R_A) \cong A[t_3^{-1}] ; \]
and hence 
\[ \opn{H}^{-1}(R_A) \neq 0 . \]

Similarly, when we invert $t_1$ we get 
\[ B_2[t_1^{-1}] \cong A[t_1^{-1}] \]
as $A$-algebras. (We can invert $t_2$ and get the same result). 
A similar calculation gives
\[  A[t_1^{-1}] \ot_{A}  \opn{H}^{-2}(R_A) \cong A[t_1^{-1}] ; \]
and hence 
\[ \opn{H}^{-2}(R_A) \neq 0 . \]


\cleardoublepage
\section{Dualizing Complexes in Algebraic Geometry}

\subsection{Definition}
Let $X$ be a noetherian scheme (e.g.\ an algebraic variety over an
algebraically closed field $\K$). We denote by $\cat{Mod}\, \mcal{O}_X$ the 
category of $\mcal{O}_X$-modules. This is an abelian category with enough
injectives. Therefore the bifunctor 
\[ \mcal{H}om_{\OX}(-,-) : (\cat{Mod}\, \mcal{O}_X)^{\mrm{op}} \times
\cat{Mod}\, \mcal{O}_X \to \cat{Mod}\, \mcal{O}_X \]
has a right derived functor 
\[ \mrm{R} \mcal{H}om_{\OX}(-,-) :
\dcat{D}^{}(\cat{Mod}\, \OX)^{\mrm{op}}  \times 
\dcat{D}^{+}(\cat{Mod}\, \OX) \to
\dcat{D}^{}(\cat{Mod}\, \OX) , \]
calculated using K-injective resolutions of the second argument. 
(Actually we can use K-injective resolutions even for unbounded complexes, but
we won't need that.)

There are inclusions of abelian categories
\[ \cat{Coh}\, \mcal{O}_X \subset \cat{QCoh}\, \mcal{O}_X 
\subset \cat{Mod}\, \mcal{O}_X . \]
They are the coherent sheaves and the quasi-coherent sheaves. 
We denote by $\dcat{D}^{\mrm{b}}_{\mrm{c}}(\cat{Mod}\, \OX)$
the full subcategory of $\dcat{D}(\cat{Mod}\, \OX)$ consisting of bounded
complexes with coherent cohomology sheaves.  

Note that the inclusion 
$\dcat{D}^{\mrm{b}}(\cat{Coh}\, \OX) \to 
\dcat{D}^{\mrm{b}}_{\mrm{c}}(\cat{Mod}\, \OX)$ 
is an equivalence (by Remark \ref{rem:200}); but that's not too useful, since 
$\cat{Mod}\, \OX$ doesn't have enough injectives. 
What is useful is the fact that for an affine open set 
$\opn{Spec} A = U \subset X$ we have an
equivalence of triangulated categories
\[ \mrm{R} \Gamma(U, -) : 
\dcat{D}^{\mrm{b}}_{\mrm{c}}(\cat{Mod}\, \mc{O}_U) \to 
\dcat{D}^{\mrm{b}}_{\mrm{f}}(\cat{Mod}\, A) . \]
The adjoint equivalence is the functor 
sending an $A$-module $M$ to the associated quasi-coherent sheaf 
$\mc{M} = \mc{O}_U \ot_A M$. 

\begin{dfn}
A {\em dualizing complex} on $X$ is a complex 
$\mcal{R} \in \dcat{D}^{\mrm{b}}_{\mrm{c}}(\cat{Mod}\, \OX)$ with  these
properties:
\begin{enumerate}
\rmitem{i} $\mc{R}$ has finite injective dimension.
\rmitem{ii} The canonical morphism 
$\th : \OX \to \mrm{R} \mcal{H}om_{\OX}(\mc{R}, \mc{R})$
is an isomorphism. 
\end{enumerate}
\end{dfn}

Like Theorem \ref{thm:150}, the triangulated functor 
\[ \mrm{R} \mcal{H}om_{\OX}(-, \mc{R}) : 
\dcat{D}^{\mrm{b}}_{\mrm{c}}(\cat{Mod}\, \OX) \to 
\dcat{D}^{\mrm{b}}_{\mrm{c}}(\cat{Mod}\, \OX) \]
is a duality (a contravariant equivalence).  And like Theorem 
\ref{thm:153}, if $X$ is connected, then a complex 
$\mc{R}' \in \dcat{D}^{\mrm{b}}_{\mrm{c}}(\cat{Mod}\, \OX)$
is dualizing iff 
\[ \mc{R}' \cong \mc{R} \ot_{\OX} \mc{L}[n] \]
for some invertible sheaf $\mc{L}$ and some integer $n$. Thus the group 
$\opn{Pic}(X) \times \Z$ acts simply transitively on the set of isomorphism
classes of dualizing complexes (if it is nonempty). 

Being a dualizing complex is something that can be checked locally. 
Suppose $X = \bigcup_i U_i$ is a finite affine open covering, and 
$A_i := \Gamma(U_i, \OX)$. A complex 
$\mc{R} \in \dcat{D}^{\mrm{b}}_{\mrm{c}}(\cat{Mod}\, \OX)$ 
is dualizing iff for every $i$ the complex 
\begin{equation} \label{eqn:200}
 R_i := \mrm{R} \Gamma(U_i, \mc{R}) \in
\dcat{D}^{\mrm{b}}_{\mrm{f}}(\cat{Mod}\, A_i) 
\end{equation}
is dualizing.

\subsection{Existence}
What about existence? Let's assume for simplicity that $X$ is of finite type
over a finite dimensional regular noetherian ring  $\K$. 


Choose a finite affine open covering 
$X = \bigcup U_i$, with $U_i = \opn{Spec} A_i$. 
We know that  $A_i$ has a dualizing complex $R_{A_i}$. 
Then the complex of quasi-coherent sheaves 
\[ \mc{R}_{U_i} := \mc{O}_{U_i} \ot_{A_i} R_{A_i} \]
is a dualizing complex on $U_i$. 
We would like to glue these dualizing complexes into a
dualizing complex $\mc{R}_X$, satisfying 
\[ \mc{R}_X|_{U_i} \cong \mc{R}_{U_i} \]
in $\dcat{D}(\cat{Mod}\, \mc{O}_{U_i})$.
For this we need two things:
\begin{enumerate}
\item The dualizing complexes $\mc{R}_{U_i}$ should admit descent data.

\item Descent for complexes is effective. 
\end{enumerate}

Here is what these conditions mean. 

\begin{enumerate}
\item A descent datum is a collection 
$\bigl( \{ \mc{M}_{i} \} , \{ \phi_{i, j} \} \bigr)$
consisting of complexes $\mc{M}_i \in \dcat{D}(\cat{Mod}\, \mc{O}_{U_i})$,
and isomorphisms
\[ \phi_{i, j} : \mc{M}_{i}|_{U_i \cap U_j} \to  \mc{M}_{j}|_{U_i \cap U_j}
\]
in $\dcat{D}(\cat{Mod}\, \mc{O}_{U_i \cap U_j})$.
These isomorphisms must satisfy the cocycle condition 
\[ \phi_{j, k} \o \phi_{i, j}  = \phi_{i, k} \]
in  $\dcat{D}(\cat{Mod}\, \mc{O}_{U_i \cap U_j \cap U_k})$.

\item Descent for complexes if effective if for any descent data as in
(1) there should exist a complex
$\mc{M} \in \dcat{D}(\cat{Mod}\, \mc{O}_{X})$,
and isomorphisms 
\[ \phi_{i} :   \mc{M}|_{U_i} \to \mc{M}_{i}  \]
in $\dcat{D}(\cat{Mod}\, \mc{O}_{U_i})$, 
s.t.\ 
\[ \phi_{j} = \phi_{i, j} \o  \phi_{i}  \]
in $\dcat{D}(\cat{Mod}\, \mc{O}_{U_i \cap U_j})$.
\end{enumerate}

A big problem: even if we were able to somehow choose the dualizing complexes 
$\mc{R}_{U_i}$ cleverly such that they would be isomorphic on double
intersections; and even if we could arrange for descent data
$\{ \phi_{i, j} \}$; still descent for
complexes is not effective! The problem is that the isomorphisms $\phi_{i,j}$
are only homotopy classes of morphism of complexes (and inverses thereof). It
is much harder, and may be impossible in general, to find descent data in 
$\dcat{C}(\cat{Mod}\, \mc{O}_{U_i \cap U_j})$...

The solution by Grothendieck in \cite{RD} was this. In order to select
``correct'' dualizing complexes on affine schemes, and at the same time produce
descent data, he used global duality. This is a very difficult and cumbersome
undertaking.

In order to make descent effective,
Grothendieck noticed that there is a functor 
\[ \opn{E} :  \dcat{D}^{+}_{\mrm{qc}}(\cat{Mod}\, \OX) \to 
\dcat{C}^{+}(\cat{QCoh}\, \OX) \]
called the Cousin functor. If $\mc{R}$ is one of the correct dualizing
complexes, then $\opn{E}(\mc{R})$ is a functorial minimal injective resolution
of $\mc{R}$. An isomorphism 
\[ \phi_{i, j} : \mc{R}_{U_i}|_{U_i \cap U_j} \to  \mc{R}_{U_j}|_{U_i \cap U_j}
\]
in $\dcat{D}(\cat{Mod}\, \mc{O}_{U_i \cap U_j})$ 
will then induces an isomorphism 
\[ \opn{E}(\phi_{i, j}) : \opn{E}(\mc{R}_{U_i})|_{U_i \cap U_j} \to 
\opn{E}(\mc{R}_{U_j})|_{U_i \cap U_j} \]
in $\dcat{C}(\cat{Mod}\, \mc{O}_{U_i \cap U_j})$.
The isomorphisms 
$\opn{E}(\phi_{i, j})$ will satisfy the cocycle condition, so these Cousin
complexes can be glued!

If $\mc{R}$ is a ``correct'' dualizing complex, then its Cousin complex 
$\mc{K} := \opn{E}(\mc{R})$ is called a {\em residual complex}. 
The reason for the name is that the coboundary operator in $\mc{K}$ can be
described in terms of residue maps. This is quite easy on a smooth curve over a
field (Serre). For arbitrary integral schemes over perfect fields this was done
in \cite{Ye2}.

\subsection{Rigid dualizing complexes and perverse coherent sheaves}
Let me outline another method for proving existence of dualizing complexes on
schemes. 

Assume $\K$ is a field. (All I'll say next
works when $\K$ is just a finite dimensional regular noetherian ring,
such as $\Z$; but there technical complications that I want to avoid.)

In 1997 Michel Van den Bergh \cite{VdB} introduced the notion of {\em rigid
dualizing complex} over a noetherian $\K$-algebra $A$. A dualizing complex $R$
is rigid if there is an isomorphism 
\[ \rho : R \iso \opn{RHom}_{A \ot_{\K} A}(A, R \ot_{\K} R) \]
in $\dcat{D}^{}(\cat{Mod}\, A)$. Here $A$ is a module over $A \ot_{\K} A$ by
the ring homomorphism $A \ot_{\K} A \to A$, corresponding to the diagonal.
A {\em rigid dualizing complex} over $A$ relative to $\K$ is a pair $(R, \rho)$
as above. 

A rigid dualizing complex  $(R, \rho)$ is {\em unique up to a unique rigid
isomorphism}. Namely if $(R', \rho')$ is another rigid dualizing complex, there
there is a unique isomorphism $\phi : R \to R'$  in $\dcat{D}^{}(\cat{Mod}\, A)$
such that the diagram 
\[ \UseTips  \xymatrix @C=6ex @R=6ex {
R
\ar[r]^(0.2){\rho}
\ar[d]_{\phi}
&
\opn{RHom}_{A \ot_{\K} A}(A, R \ot_{\K} R)
\ar[d]^{\phi \ot \phi}
\\
R'
\ar[r]^(0.2){\rho'}
&
\opn{RHom}_{A \ot_{\K} A}(A, R' \ot_{\K} R') 
} \]
is commutative. 

Any essentially finite type $\K$-algebra admits a rigid dualizing complex 
$(R_A, \rho_A)$. 

If $A \to B$ is a localization, and $(R_A, \rho_A)$ is a rigid dualizing complex
over $A$, then the complex $B \ot_A R_A$ has an induced rigidifying isomorphism.
Thus if $(R_B, \rho_B)$ is a rigid dualizing complex over $B$, there is a
unique rigid isomorphism 
\[ (B \ot_A R_A, \rho_A) \cong (R_B, \rho_B)  \]
in $\dcat{D}^{}(\cat{Mod}\, B)$.
 
Now consider a finite type $\K$-scheme $X$. Choose an affine open covering 
$X = \bigcup U_i$, and let $A_i := \Gamma(U_i, \OX)$. Let  
$(R_{i}, \rho_{i})$ be the rigid dualizing complex of $A_i$, and let 
$\mc{R}_i$ be corresponding dualizing complex on $U_i$.  Rigidity provides
descent isomorphisms $\{ \phi_{i, j} \}$ that satisfy the cocycle condition.

Let's go back to the affine situation. Let 
$\dcat{D}^{\mrm{b}}_{\mrm{f}}(\cat{Mod}\, A)^0$ be the full subcategory of 
$\dcat{D}^{\mrm{b}}_{\mrm{c}}(\cat{Mod}\, A)$ consisting of complexes $M$ s.t.\
$\opn{H}^i(M) = 0$ for all $i \neq 0$. We know that 
\[ \cat{Mod}_{\mrm{f}}\, A \to \dcat{D}^{\mrm{b}}_{\mrm{f}}(\cat{Mod}\, A)^0 \]
is an equivalence.

The rigid dualizing complex $R_A$ of $A$ induces a duality 
\[ D = \opn{RHom}_A(-, R) : 
\dcat{D}^{\mrm{b}}_{\mrm{f}}(\cat{Mod}\, A) \to 
\dcat{D}^{\mrm{b}}_{\mrm{f}}(\cat{Mod}\, A) . \]
We define 
$^{\mrm{p}}\dcat{D}^{\mrm{b}}_{\mrm{f}}(\cat{Mod}\, A)^0$ to be the image under 
$D$ of the category $\dcat{D}^{\mrm{b}}_{\mrm{f}}(\cat{Mod}\, A)^0$.
A complex $M \in {}^{\mrm{p}}\dcat{D}^{\mrm{b}}_{\mrm{f}}(\cat{Mod}\, A)^0$
is called a {\em perverse finitely generated $A$-module}. Thus a
complex $M \in \dcat{D}^{\mrm{b}}_{\mrm{f}}(\cat{Mod}\, A)$ is a perverse
module iff $M \cong D(N)$ for some $N \in \cat{Mod}_{\mrm{f}}\, A$.

Since 
\[ D : \cat{Mod}_{\mrm{f}}\, A \to 
{}^{\mrm{p}}\dcat{D}^{\mrm{b}}_{\mrm{f}}(\cat{Mod}\, A)^0 \]
is a duality, the latter is an abelian category. 

Now take our f.t.\ $\K$-scheme $X$. For any affine open set 
$U = \opn{Spec} A \subset X$ we have the category of perverse $A$-modules 
${}^{\mrm{p}}\dcat{D}^{\mrm{b}}_{\mrm{f}}(\cat{Mod}\, A)^0$. 

We say that a complex 
$\mc{M} \in  \dcat{D}^{\mrm{b}}_{\mrm{c}}(\cat{Mod}\, \OX)$
is a {\em perverse coherent sheaf} if for every affine open set 
$U = \opn{Spec} A $ the complex 
$\mrm{R} \Gamma(U, \mc{M})$ belongs to 
${}^{\mrm{p}}\dcat{D}^{\mrm{b}}_{\mrm{f}}(\cat{Mod}\, A)^0$. 
The full subcategory of
$\dcat{D}^{\mrm{b}}_{\mrm{c}}(\cat{Mod}\, \OX)$
on the perverse coherent sheaves is denoted by 
${}^{\mrm{p}}\dcat{D}^{\mrm{b}}_{\mrm{c}}(\cat{Mod}\, \OX)^0$. 

It is easy to see that 
${}^{\mrm{p}}\dcat{D}^{\mrm{b}}_{\mrm{c}}(\cat{Mod}\, \OX)^0$ is an abelian
category (dual to $\cat{Coh}\, \OX$). What is more interesting is that the
assignment 
\[ V \mapsto 
{}^{\mrm{p}}\dcat{D}^{\mrm{b}}_{\mrm{c}}(\cat{Mod}\, \mc{O}_V)^0 , \]
for $V \subset X$ open, is a {\em stack of abelian categories}; cf.\
\cite[Definition 4.2]{YZ2} or \cite[Proposition 10.2.9]{KS2}. 
This means that descent is effective! 

Now let's pick some finite affine open covering 
$X = \bigcup U_i$. The rigid dualizing complexes $\mc{R}_i$ on $U_i$ that we got
above are perverse coherent sheaves, and hence they glue to a global perverse
coherent sheaf $\mc{R}$. This is the desired {\em rigid dualizing complex of
$X$}

\begin{rem}
The stack property of perverse sheaves is the reason they  are called
perverse {\em sheaves}. 
The original definition of perverse sheaves is in \cite{BBD}. 
See also \cite[Chapter X]{KS2}, which is where I learned about this subject.
\end{rem}

\begin{rem}
Dualizing complexes make sense for noncommutative $\K$-algebras. 
See \cite{Ye1}. For the noncommutative version of Theorem \ref{thm:153} see
\cite{Ye4}.

Rigid dualizing complexes also make sense for noncommutative
algebras -- see \cite{VdB, Ye4, YZ1}. They should also make sense for
noncommutative DG algebras. They are closely related to the concept of {\em
Calabi-Yau algebra}: a $\K$-algebra $A$ is CY of dimension $n$ if it is
regular, and  $A[n]$ is a rigid dualizing complex. This concept was popularized
by Kontsevich, in connection with homological mirror symmetry.

For rigid dualizing complexes on noncommutative schemes see \cite{YZ2}.
The results in this subsection are taken from there. 
(The commutative version of rigid dualizing complexes on schemes, relative to a
regular base ring, will hopefully be written up this summer).
\end{rem}


\cleardoublepage

\end{document}